\documentclass{article}
\usepackage {amsmath, amsfonts, amssymb, mathrsfs, amsthm,stmaryrd, sectsty,hyphenat,url,accents,graphicx,comment}
\usepackage{xcolor}
\usepackage{palatino}
\usepackage[all]{xy}
\usepackage[toc,page]{appendix}
\usepackage{tocloft}
\usepackage{enumitem}

\def\mf#1{\mathfrak{#1}}
\def\mc#1{\mathcal{#1}}
\def\mb#1{\mathbb{#1}}
\def\tx#1{\textrm{#1}}

\newcommand{\mrm}[1]{{\mathrm{#1}}}

\def\R{\mathbb{R}}

\def\C{\mathbb{C}}
\def\Q{\mathbb{Q}}

\def\Z{\mathbb{Z}}

\def\lmod{\backslash}

\def\hat{\widehat}

\def\from{\leftarrow}

\def\into{\hookrightarrow}

\def\<{\langle}
\def\>{\rangle}
\def\Wedge{\bigwedge}

\newenvironment{mytitle}
{\begin{center}\large\sc}
{\end{center}}

\renewcommand{\hom}{\mrm{Hom}}
\newcommand{\Mod}{\mrm{Mod}}

\DeclareMathOperator{\Gr}{Gr}
\DeclareMathOperator{\Rep}{Rep}
\DeclareMathOperator{\spec}{Spec}

\newtheorem{thm}{Theorem}[subsection]
\newtheorem{lem}[thm]{Lemma}
\newtheorem{pro}[thm]{Proposition}
\newtheorem{cor}[thm]{Corollary}
\newtheorem{fct}[thm]{Fact}

\theoremstyle{definition}
\newtheorem{rem}[thm]{Remark}
\newtheorem{dfn}[thm]{Definition}

\sectionfont{\center\sc\normalsize}
\subsectionfont{\bf\normalsize}

\numberwithin{equation}{subsection}

\begin{document}

\begin{mytitle}
  On the refined local Langlands conjecture for discrete $L$-parameters of inner forms of quasi-split disconnected real reductive groups
\end{mytitle}

\begin{center}
  Tasho Kaletha and Paul Mezo\\[12pt]
  with an appendix by Dougal Davis
\end{center}

\begin{abstract}
  Given a quasi-split connected reductive $\R$-group $G$ and a finite group $A$ acting on $G$ by $\R$-automorphisms that preserve an $\R$-pinning, we construct for each discrete $L$-parameter for $G$ a corresponding $L$-packet of irreducible discrete series representations on each inner forms $\tilde G_z(\R)$ of the disconnected group $\tilde G = G \rtimes A$. We prove that these $L$-packets satisfy the endoscopic character identities with respect to normalized transfer factors. This proves the conjectural refined local Langlands correspondence for inner forms of quasi-split disconnected real reductive groups, as recently formulated by the first author.
\end{abstract}

\tableofcontents

\section{Introduction}

\subsection{General discussion}

Let $F$ be a local field and $G$ a connected reductive $F$-group. The refined local Langlands correspondence is a conjecture consisting of two parts. The first part predicts an injective map from the set of isomorphism classes of irreducible admissible complex representations of $G(F)$ to the set of enhanced $L$-parameters. These are $\hat G$-conjugacy classes of tuples $\{(\varphi,\rho)\}$, where $\varphi : L_F \to {^LG}$ is an $L$-parameter, and $\rho$ is an irreducible representation of a group closely related to the centralizer of $\varphi$ in $\hat G$. Here $\hat G$ is the complex Langlands dual group of $G$, $^LG$ is the $L$-group, and $L_F$ is the Langlands group of $F$, thus equal to the Weil group $W_F$ if $F$ is archimedean, and equal to $W_F \times \tx{SL}_2(\C)$ if $F$ is non-archimedean. This injection depends on some choices, namely a Whittaker datum $\mf{w}$ on the quasi-split inner form $G$ of $G$, and a rigid inner twist $(\xi,z) : G \to G$. The injection can be made into a bijection by putting all inner forms of $G$ together. The second part of the conjecture predicts that this injection satisfies endoscopic character identities with a precise normalization of the transfer factors in terms of $\mf{w}$ and $(\xi,z)$. We refer the reader to \cite{KalSimons} for a discussion and precise formulation when $F$ has characteristic zero, and to \cite{Dillery23} when $F$ has positive characteristic.

For various applications it is desirable to extend this conjecture to disconnected groups $\tilde G$ whose identity component $G$ is reductive. Such an extension was proposed in \cite{KalLLCD} under the assumption that $F$ has characteristic zero and $\tilde G$ arises as an inner form $\tilde G_{\bar z}$ of a disconnected group $\tilde G = G \rtimes A$, where $G$ is a quasi-split connected reductive $F$-group, $A$ is a finite group of $F$-automorphisms (\emph{not} assumed abelian) whose action on $G$ stabilizes an $F$-pinning (the same pinning for all elements of $A$), and $\bar z \in Z^1(F,G/Z(G)^A)$. The conjecture, stated as \cite[Conjecture 7.2.1]{KalLLCD}, again consists of two parts. The first part stipulates an injection from the irreducible admissible representations of $\tilde G_{\bar z}(F)$ to the set of enhanced $L$-parameters, which are again tuples $\{(\varphi,\rho)\}$, with $\varphi : L_F \to {^LG}$ being an $L$-parameter as in the connected case, but $\rho$ now being an irreducible representation of a group closely related to the centralizer of $\varphi$ in $\hat G \rtimes A^{[\bar z]}$; the tuple is taken up to conjugation by $\hat G \rtimes A^{[\bar z]}$. This injection again depends on the choices of a Whittaker datum $\mf{w}$ for $G$, now assumed $A$-admissible, and a rigid inner twist $(\xi,z) : \tilde G \to \tilde G_{\bar z}$. The second part of the conjecture predicts that this injection satisfies endoscopic character identities with a precise normalization of the transfer factors in terms of $\mf{w}$ and $(\xi,z)$. 

The formulation of the conjecture in the disconnected setting is strikingly analogous to that in the connected setting. Indeed, if one deletes all the ``tildes'' from the statement, it would be almost impossible to tell if one is talking about connected or disconnected groups. As such, it provides a clean and uniform language for much (but not quite all) of the theory of twisted endoscopy put forth in \cite{KS99}, which can be seen as restricting attention to individual $G_{\bar z}$-cosets in $\tilde G_{\bar z}$. However, considering the full disconnected group goes beyond the classical picture of twisted endoscopy, because it also contains the way individual cosets interact with each other. Questions about how representations of $G_z(F)$ extend to $\tilde G_z(F)$, or various individual cosets, are seamlessly handled by the framework. Note that while the identity $\tilde G(F) = G(F) \rtimes A$ holds, the topological group $\tilde G_{\bar z}(F)$ is a-priori just an extension of $A^{[\bar z]}$ by $G_{\bar z}(F)$, and this extension does not come equipped with the structure of a semi-direct product, and we have no reason to believe that such structure on $G_{\bar z}(F)$ exists in general.

\subsection{Main results and methods}

The main result of this paper is the verification of this conjecture in the setting of $F=\R$ and $\varphi$ being a discrete $L$-parameter. To state the results more precisely, let $G$ be a connected quasi-split $\R$-group, $A$ be a finite (possibly non-abelian) group acting on $G$ by $\R$-automorphisms stabilizing an $\R$-pinning, $\bar z \in Z^1(\R,G/Z(G)^A)$, and $\varphi : W_\R \to {^LG}$ be a discrete $L$-parameter. Let $\tilde G_{\bar z}$ be the inner form of $\tilde G=G \rtimes A$ determined by $\bar z$ as in \cite[\S3.2]{KalLLCD}. Choose an $A$-admissible Whittaker datum for $G$ in the sense of \cite[\S3.5]{KalLLCD} and a cocycle $z \in Z^1(\mc{P}^\tx{rig}_\R \to \mc{E}^\tx{rig}_\R,Z(G)^A \to G)$ lifting $\bar z$, where the cohomology used here is the one described in \cite{KalRI} but we are using the notation of \cite{TaibiIHES2022}; the reader unfamiliar with this cohomology is encouraged to imagine $z \in Z^1(\R,G)$ and assume that such a $z$ exists. Write $\Gamma$ for the Galois group of $\C/\R$.
\begin{enumerate}
  \item We construct an $L$-packet $\Pi_\varphi(\tilde G_{\bar z})$ consisting of isomorphism classes of irreducible essentially discrete series representations of the real Lie group $\tilde G_{\bar z}(\R)$.
  \item We construct an injection $\Pi_\varphi(\tilde G_{\bar z}) \to \tx{Irr}(\pi_0(\tilde S_\varphi^{+,[z]}))$, where $\tilde S_\varphi^{[z]}$ is the centralizer of $\varphi$ in $\hat G \rtimes A^{[z]}$, $A^{[z]}$ is the stabilizer in $A$ of the cohomology class of $z$, and $\tilde S_\varphi^{+,[z]}$ is the preimage in $\tilde S_{\varphi}^{[z]}$ in $\hat{\bar G} \rtimes A$, with $\hat{\bar G}$ being the projective limit of all finite covers of $\hat G$. The image of this injection consists of (possibly not all) irreducible representations that transform under $\pi_0(Z(\hat {\bar G})^+)$ by the character corresponding to $[z]$, where $Z(\hat {\bar G})^+$ is the preimage of $Z(\hat G)^\Gamma$.
  \item Writing $\tilde\pi \mapsto \rho_{\tilde\pi}$ for the above injection, we can form for any semi-simple $\tilde s \in \tilde S_\varphi^{+,[z]}$ the refined endoscopic datum $(H,\mc{H},\tilde s,\eta)$, choose arbitrarily a $z$-pair $(H_1,{^L\eta_1})$, and consider the virtual characters
  \[ \Theta_{\varphi,z}^{\tilde s,\mf{w}} = e(G_z)\sum_{\tilde\pi \in \Pi_\varphi(\tilde G_z)} \tx{tr}(\rho_{\tilde\pi}(\tilde s)) \Theta_{\tilde\pi} \quad \tx{and} \quad
   S\Theta_{\varphi_1} = \sum_{\pi_1 \in \Pi_{\varphi_1}(H_1)} \Theta_{\pi_1} \]
   on $\tilde G_{\bar z}(\R)$ and $H_1(\R)$ respectively, where $\varphi_1 : W_\R \to {^LH_1}$ is the composition of $\varphi : W_\R \to \mc{H}$ with $^L\eta_1 : \mc{H} \to {^LH_1}$. Here $e(G_z)$ is the Kottwitz sign of $G_z$, see \cite{Kot83}. We prove
   \[ \Theta_{\varphi,z}^{\tilde s,\mf{w}}(\tilde f) = S\Theta_{\varphi_1}(f_1), \]
   where $\tilde f \in \mc{C}^\infty_c(\tilde G_z(\R))$ and $f_1 \in \mc{C}^\infty_c(H_1(\R))$ are matching functions with respect to the transfer factor normalized using $\mf{w}$ and $z$, in the sense of \cite[Lemma 4.12.1]{KalLLCD}.
\end{enumerate}

We emphasize that the real Lie group $\tilde G_{\bar z}(\R)$ is generally not in the Harish-Chandra class, so Harish-Chandra's original construction of discrete series representations doesn't apply. Instead, for the construction of the $L$-packets and their internal structure we use Duflo's construction \cite{Duflo82} of irreducible representations of arbitrary Lie groups, specialized to the case of a disconnected reductive group $\tilde G$ and essentially discrete series representations thereof. According to this construction such a representation arises from a pair $(f,\tau)$, where $f$ is a regular elliptic $\R$-point of the dual Lie algebra of $\tilde G$, $\tilde S$ is the centralizer of $f$ in $\tilde G$, and $\tau$ is an irreducible genuine representation of a certain metaplectic double cover of $\tilde S(\R)$ with differential $f$; here $f$ will become the infinitesimal character of the resulting representation. In order to relate such pairs to members of the set $\tx{Irr}(\pi_0(\tilde S_\varphi^{+,[z]}))$ we use the special case of the conjectures of \cite{KalLLCD} for disconnected groups whose identity component is a torus (which is proved in loc. cit.), and we apply these results to $\tilde S$. For this, the version of these results that appear in the thesis \cite{YiLuoThesis} of Yi Luo is quite important, because this version is more clear and explicit than the original version and allows us to connect these results better to our setting here.

Connecting these two inputs we are able to obtain an injection 
\[ \Pi_\varphi(\tilde G_{\bar z}) \to \tx{Irr}(\pi_0(\tilde S_\varphi^{+,[z]})) \] 
as required by 2. above. To our great surprise this turns out to \emph{not} be the correct injection. A first indication is given by the fact that part of the statement of \cite[Conjecture 7.2.1]{KalLLCD} is that the trivial representation of $\pi_0(\tilde S_\varphi^{+,[z]})$ ought to correspond to the unique representation $\tilde\pi_\mf{w}$ in $\Pi_\varphi(\tilde G)$ whose restriction $\pi_\mf{w}$ to $G(\R)$ is irreducible and $\mf{w}$-generic and such that $\tilde\pi_\mf{w}$ is the extension of $\pi_\mf{w}$ to $\tilde G(\R) = G(\R) \rtimes A$ obtained by Whittaker normalization. Under the above injection the representation $\tilde\pi$ of $\tilde G(\R)$ that corresponds to the trivial representation of $\pi_0(\tilde S_\varphi^{+,[z]})$ does have the property that its restriction to $G(\R)$ is $\pi_\mf{w}$, but in general it is not equal to the Whittaker extension of $\pi_\mf{w}$. In other words, $\tilde\pi$ and $\tilde\pi_\mf{w}$ are generally two distinct extensions of $\pi_\mf{w}$ to $\tilde G(\R)$. A key result of this paper, Theorem \ref{thm:rtci-sign} proved by Dougal Davis in Appendix \ref{app:a}, is that 
\[ \tilde\pi = \tilde\pi_\mf{w} \otimes \epsilon_{\tilde G}, \]
for a certain explicitly defined character 
\[ \epsilon_{\tilde G} : \tilde G(\R)/G(\R) \to \{\pm 1\}. \]
This character depends only on the group $\tilde G$, and not on the parameter $\varphi$. The definition of this character is elementary and is given in \S\ref{sub:signchar}. The proof of the above identity is however not elementary, and was a major obstacle to the completion of this paper. The proof given by Davis uses Beilinson--Bernstein localization, the theory of associated cycles of Harish-Chandra modules, and Ginzburg's description of the characteristic cycles of direct images of D-modules under open immersions.

The proof of the endoscopic character identities proceeds very roughly as in the outline of \cite{DPR}. In fact, \cite{DPR} was initially conceived as a warm-up in the process of writing this paper. Thus, the first step is as follows. First, reduce the proof to the setting of regular elliptic elements by a version of Harish--Chandra's uniqueness theorem for invariant eigendistributions. Second, convert the desired identity from the language of distributions to the language of functions using the Weyl integration formula. Third, prove that identity by analyzing the character of essentially discrete series representations and the structure of the transfer factor, and combining these with the combinatorics of transfer of conjugacy classes in endoscopy.

The details of this proof are however quite different and more involved than in the case of connected groups. One of the difficulties centers around the abstract notion of norms in twisted endoscopy. This notion is more subtle than the transfer of conjugacy classes. For example, a regular semi-simple element of $\tilde G_z(\R)$ need not have a stable conjugate lying in $\tilde G(\R)$, while the relationship between elements of $H(\R)$ and $\tilde G_z(\R)$ involves passing through $\tilde G$. To deal with such difficulties, we first have to develop some foundational results about the abstract notion of norms in twisted endoscopy, and relate this notion to that of Cartan subspaces of twisted spaces that is developed in \cite{HL17}. These results work over an arbitrary local field. We then specialize to $F=\R$ and elliptic elements and prove more specific results in this setting. The reduction of the character identity to elliptic elements uses a result \cite{Ren97} of Renard that is a twisted analog of the Harish-Chandra uniqueness theorem, but must be supplemented with a careful subdivision of the space of conjugacy classes in $\tilde G_z(\R)$. Switching from distributions to functions involves the twisted form of the Weyl integration formula proved by Henniart--Lemaire in \cite{HL17}. The character of essentially discrete series representations of disconnected groups has been described by Bouaziz in \cite{Bouaziz87}. The structure of the transfer factor is described partly in \cite{KS99}, with some corrections coming from \cite{KS12}, and partly in \cite{KalLLCD}, where the cohomology of the Galois gerbe defined in \cite{KalRI} is extended to the setting of complexes of tori. These are the main inputs of the proof of the character identity.

\subsection{Outline of the material}

In \S\ref{sec:dbl} we describe two constructions of double covers of real Lie groups arising as $\R$-points of certain affine algebraic groups, the main application in our case being maximal tori in connected reductive groups and their generalizations to disconnected reductive groups. The first construction is elementary and is a minor variation of the $\rho$-cover construction used for example in \cite{AV16}, while the second construction uses the metaplectic group and is an important ingredient in Duflo's construction \cite{Duflo82} of irreducible representations of real Lie groups. We show that the two constructions produce the same outcome. We also review some basic facts about genuine characters of double covers of disconnected tori.

In \S\ref{sec:dsaut} we discuss the interaction between irreducible essentially discrete series representations of $G(\R)$, where $G$ is a quasi-split connected reductive $\R$-group, and automorphisms of $G$ preserving a pinning. We recall here the concept of a Harish-Chandra parameter of such a representation, which is a $G(\R)$-conjugacy class of pairs $(S,\tau)$, with $S$ an elliptic maximal torus of $G$ and $\tau$ a genuine character of the double cover of $S(\R)$ with regular differential. We then define the concepts of an $A$-stable parameter (one that is stable under $A$) and an $A$-admissible parameter (one that contains an $A$-stable representative), and show that the latter notion is related to genericity with respect to an $A$-admissible Whittaker datum (a $G(\R)$-conjugacy class of pairs $(B,\psi)$ consisting of an $\R$-Borel subgroup and a generic character of its unipotent radical, that is not only $A$-stable, but contains an $A$-stable representative).

In \S\ref{sec:dsdisc} we discuss irreducible essentially discrete series representations of a disconnected reductive group $\tilde G(\R)$. In that section work in the generality of an affine algebraic $\R$-group $\tilde G$ whose identity component $G$ is reductive and pose no further conditions. We review Duflo's construction \cite{Duflo82} of irreducible representations of $\tilde G(\R)$ in the special case of essentially discrete series that is relevant to us. We also review Bouaziz's work \cite{Bouaziz87} on the character formula for such representations, which we will use in the proof of the endoscopic character identities. Finally, we define in \S\ref{sub:signchar} the sign character $\epsilon_{\tilde G}$ that is needed to obtain the correct internal structure of the $L$-packets.

In \S\ref{sec:sqlp} we accomplish the first two goals from the above list, namely the construction of the $L$-packet $\Pi_\varphi(\tilde G_z)$ and its parameterization in terms of $\tx{Irr}(\tilde S_\varphi^{+,[z]})$. In fact, the first goal alone is trivial to achieve, as one can define $\Pi_\varphi(\tilde G_z)$ to consist of all those irreducible representations $\tilde\pi$ of $\tilde G_z(\R)$ whose restriction to $G_z(\R)$ has an irreducible constituent belonging to the $L$-packet $\Pi_\varphi(G_z)$ that is already available from the classical theory developed by Langlands. It is the second goal that is the subtle point, especially because the third goal, namely the character identities, can hold for at most one injection $\Pi_\varphi(\tilde G_z) \to \tx{Irr}(\tilde S_\varphi^{+,[z]})$, as one sees using the linear independence of characters. To obtain this injection, we review in \S\ref{sub:luo} the results on the refined local Langlands conjecture for disconnected tori, originally obtained in \cite[\S8]{KalLLCD} and refined in Yi Luo's thesis \cite{YiLuoThesis}. We combine Luo's refinement with Duflo's construction, and apply twisting by the sign character $\epsilon_{\tilde G}$, in order to obtain the desired injection $\Pi_\varphi(\tilde G_z) \to \tx{Irr}(\tilde S_\varphi^{+,[z]})$.

In \S\ref{sec:charid} we prove the endoscopic character identity. After first proving general results about endoscopic norms and conjugacy classes, we reduce the desired identity of distributions to an identity of functions restricted to the set of regular elliptic elements of $[G \rtimes a]_z(\R)$, where $a \in A$ is the inverse of the image of the endoscopic element $\tilde s$ under the map $\tilde S_\varphi \to A$. To prove this identity we study the $G$-side using Bouaziz's character formula and the results of Luo's thesis, and we study the $H$-side by analyzing the structure of the transfer factors.

The appendix contains the proof of Theorem \ref{thm:rtci-sign}, on which the internal structure of the $L$-packet, and hence the validity of the endoscopic character identities, hinges, in the core case when $G$ is a semi-simple simply connected group.

\subsection{Acknowledgements}

This paper has benefited from discussions with Jeffrey Adams, Bill Casselman, and David Vogan, whom we thank heartily. We are especially grateful to Dougal Davis for providing the appendix which proves Theorem \ref{thm:rtci-sign}.

\section{Double covers of tori} \label{sec:dbl}

\subsection{Explicit double covers} \label{sub:edc}

Let $S$ be an affine algebraic $\R$-group equipped with a finite set $X \subset \tfrac{1}{2}X^*(S)$  such that $0 \notin X$ and for $x_1,x_2 \in X$ we have $x_1-x_2 \in X^*(S)$. For $y \in X^*(S)$ and $s \in S$ we write $s^y := y(s) \in \C^\times$. Define 
\[ S(\R)_X = \{(s,(s_x)_{x\in X})|s \in S(\R),s_x \in \C^\times, s_x^2=s^{(2x)}, s_{x_1}/s_{x_2}=s^{(x_1-x_2)}\}. \]
The map $p_X : S(\R)_X \to S(\R)$ sending $(s,(s_x))$ to $s$ produces the exact sequence
\[ 1 \to \{\pm 1\} \to S(\R)_X \to S(\R) \to 1, \]
with $-1$ being the element $(1,(-1)_{x \in X})$.

For each $x \in X$ we denote by $(\cdot)^x : S(\R)_X \to S(\R)$ the map sending $(s,(s_y)_{y\in X})$ to $s_x$. This map is a genuine character. Its square sends $(s,(s_y)_{y\in X})$ to $s^{(2x)}$. Given $\dot s \in S(\R)_X$ write $s=p_X(\dot s)$ and record the identities
\[ ((\dot s)^x)^2=s^{2x},\qquad \dot s^x/\dot s^y = s^{x-y}.\]

For each $x \in X$ the diagram
\[ \xymatrix{
  S(\R)_X\ar[r]^{(-)^x}\ar[d]_{p_X}&\C^\times\ar[d]^{(-)^2}\\
  S(\R)\ar[r]^{(-)^{2x}}&\C^\times
}\]
is Cartesian. This identifies $S(\R)_X$ with $S(\R)_{\{x\}}$ for any $x \in X$. The advantage of using $S(\R)_X$ is to emphasize that this construction does not depend on the choice of $x \in X$.

The above construction arises often as follows. Let $G$ be a connected reductive $\R$-group and let $S \subset G$ be a maximal $\R$-torus. For each Borel $\C$-subgroup of $G$ containing $S$ let $\rho_B=\tfrac{1}{2}\sum_{\alpha>0}\alpha$, where the sum runs over the $B$-positive absolute roots of $S$ in $G$. The cover $S(\R)_X$ associated to $X=\{\rho_B|B\}$ coincides with the $\rho$-double cover of Adams-Vogan \cite{AV16}. When $S$ is elliptic, which will be the main case of interest in this paper, it also coincides with the double cover constructed in \cite[\S3.1]{KalDC}, as explained in \cite[\S5.2]{KalDC}. For any choice of $B$ the function
\[ D_B(\dot s) = \prod_{\alpha>0} (\dot s^{\alpha/2} - \dot s^{-\alpha/2}) \]
is a well-defined genuine function on $S(\R)_X$ and we have
\[ D_B(\dot s) \cdot \dot s^{-\rho_B} = \prod_{\alpha>0} (1 - s^{-\alpha}). \]
For any $w \in W_G(S)$ we have $wD_B=D_{wB}=\epsilon(w)D_B$, where $\epsilon : W_G(S) \to \{\pm1\}$ is the sign character of the Weyl group.

In this paper, we will use a slight generalization of the preceding set-up, where $\tilde G$ is an affine algebraic $\R$-group whose identity component $G$ is reductive, $S \subset G$ is a maximal $\R$-torus, and $S \subset \tilde S \subset N_{\tilde G}(S)$. We will apply the above construction to $\tilde S$ and the set $X = \{\rho_B|B\}$, where now $B$ runs only over those Borel $\C$-subgroups of $G$ containing $S$ that are normalized by $\tilde S$; this set will be non-empty in our applications. Note that then $2\rho_B$ has a natural extension to $\tilde S$, namely as $\det(-|\mf{u}_B)$, where $\mf{u}_B$ is the Lie algebra of the unipotent radical of $B$. We will denote the double cover $\tilde S(\R)_X$ arising from this particular choice of $X$ by $\tilde S(\R)_G$.

\subsection{The symplectic form $B_f$} \label{sub:bf}

A second construction of covers will come from the metaplectic group. We review here the relevant symplectic structure.

Let $\mf{g}$ be a complex reductive Lie algebra and let $f \in \mf{g}^*$. Then $B_f(X,Y) = f([X,Y])$ defines an alternating form on $\mf{g}$. 
\begin{lem} \label{lem:radbf}
The radical of $B_f$ is the Lie algebra of the centralizer $Z_G(f)$ of $f$ in $G$.   
\end{lem}
\begin{proof}
For any $g \in Z_G(f)$ we have $f(\tx{Ad}(g)^{-1}X)=f(X)$, and differentiating at $g=1$ in direction $Y$ we obtain $f([Y,X])=0$ for all $Y \in \tx{Lie}(Z_G(f))$. Conversely, if $Y \in \mf{g}$ such that $f([Y,X])=0$ for all $X \in \mf{g}$, then $f(\tx{Ad}(\exp(Y))(X)) = f(e^{\tx{ad}(Y)}(X))=\sum_{n=0}^\infty \frac{f(\tx{ad}(Y)^n(X))}{n!}=f(X)$, and we conclude $\exp(Y) \in Z_G(f)$, therefore $Y \in \tx{Lie}(Z_G(f))$.
\end{proof}

Assume now that $f$ is regular semi-simple. This means that $Z_\mf{g}(f)=:\mf{s}$ is a Cartan subalgebra. 

\begin{cor}
The form $B_f$ descends to a non-degenerate alternating form on $\mf{g}/\mf{s}$.
\end{cor}

\begin{lem} \label{lem:killf}
\begin{enumerate}
  \item The linear form $f$ kills all root spaces for $\mf{s}$.
  \item The linear form is non-zero on each coroot.
\end{enumerate}
\end{lem}
\begin{proof}
(1) If $X \in \mf{g}_\alpha$ then for all $s \in \mf{s}$ we have $f(X)=f(\tx{ad}(s)X)=\alpha(s)f(X)$, forcing $f(X)=0$.

(2) For a root $\alpha$ consider the linear form $B_f(\mf{g}_\alpha,-)$. According to (1), it kills all $\mf{g}_\beta$ for $\beta \neq -\alpha$, as well as $\mf{s}$. If $f(H_\alpha)=0$, then this linear form also kills $\mf{g}_{-\alpha}$, hence all of $\mf{g}$, showing that $\mf{g}_\alpha$ is in the radical of $B_f$, which by Lemma 
\ref{lem:radbf} means $\mf{g}_\alpha \subset Z_\mf{g}(f)$, contradicting the regularity of $f$.
\end{proof}

\begin{cor} \label{cor:lagf}
The nilpotent radical of a Borel subalgebra containing $\mf{s}$ is a Lagrangian subspace for $B_f$ (when projected to $\mf{g}/\mf{s}$).
\end{cor}
\begin{proof}
Fix a Borel subalgebra $\mf{b}$ containing $\mf{s}$ and write $\alpha>0$ when $\alpha$ is an absolute root for $\mf{s}$ such that $\mf{g}_\alpha \subset \mf{b}$. To see that the unipotent radical of $\mf{b}$ is isotropic it is enough to show that for $\alpha,\beta>0$ and $X \in \mf{g}_\alpha$, $Y \in \mf{g}_\beta$, we have $B_f(X,Y)=0$. But $[X,Y]$ lies in $\mf{g}_{\alpha+\beta}$, which may or may not be zero, but either way is killed by $f$ by Lemma \ref{lem:killf}(1).

To see that it is maximally isotropic note that any larger subspace of $\mf{g}/\mf{s}$ will contain a vector $X$ lying in the opposite unipotent radical. This vector will be a sum of root vectors, and let $X_{-\alpha} \in \mf{g}_{-\alpha}$ be a summand of this vector. Taking $X_{\alpha} \in \mf{g}_\alpha$ so that $[X_\alpha,X_{-\alpha}]=H_\alpha$ we see $B_f(X_\alpha,X_{-\alpha}) \neq 0$ by Lemma \ref{lem:killf}(2).
\end{proof}

\subsection{The metaplectic double cover} \label{sub:mdc}

Let $\tilde G$ be an affine algebraic $\R$-group whose identity component $G$ is reductive. Let $f \in \mf{g}^*(\R)$ be a regular semi-simple element and let $\tilde S \subset \tilde G$ be its centralizer. Then $S=\tilde S \cap G$ is a maximal $\R$-torus of $G$. The vector space $\mf{g}/\mf{s}$ carries the symplectic form $B_f$ reviewed in \S\ref{sub:bf}. We denote by $\tx{Sp}_f$ the symplectic group $\tx{Sp}(\mf{g}/\mf{s},B_f)(\R)$, and by $\tx{Mp}_f$ its two-fold metaplectic cover. The action of $\tilde S(\R)$ on $\mf{g}$ produces a homomorphism $\tilde S(\R) \to \tx{Sp}_f$ and we define, following \cite{Duflo82}, $\tilde S(\R)_f$ to be the pull-back of
\[ \tilde S(\R) \to \tx{Sp}_f \from \tx{Mp}_f. \]
Then $\tilde S(\R)_f$ is a double cover of $\tilde S(\R)$, and comes equipped with a collection of genuine characters $\rho_{f,\mf{u}}$, indexed by the set of unipotent radicals $\mf{u}$ of Borel $\C$-subalgebras containing $\mf{s}$ and stable under $\tilde S$ (provided such $\mf{u}$ exists). Indeed, the image of such a $\mf{u}$ in $\mf{g}/\mf{s}$ is a maximal Lagrangian subspace for the symplectic form $B_f$, by Corollary \ref{cor:lagf}. Let $\tx{Mp}(\mf{g}/\mf{s})_{f,\mf{u}}$ be the stabilizer of $\mf{u}$ in $\tx{Mp}(\mf{g}/\mf{s})_f$. In \cite[I.6]{Duflo82}, using the theory of the Weil representation, Duflo defines a genuine character $\rho_{f,\mf{u}} : \tx{Mp}(\mf{g}/\mf{s})_{f,\mf{u}} \to \C^\times$, and shows in \cite[\S I.10(28)]{Duflo82} the property
\begin{equation} \label{eq:duflo1}
  \rho_{f,\mf{u}}(\dot x)^2 = \det(x|\mf{u}),\qquad \dot x \in \tx{Mp}(\mf{g}/\mf{s})_{f,\mf{u}}.
  \end{equation}
In this paper we will be interested in the case when $f$ is elliptic, equivalently $S$ is an elliptic maximal torus of $G$, which is to say that $S/Z_G$ is anisotropic. For each absolute root $\alpha \in R(S,G)$, the value $f(H_\alpha)$ is then imaginary. Following the convention used by \cite{Bouaziz87} we define a set of positive roots in $R(S,G)$ by
\[ R_f^+=\{\alpha \in R(S,G)|if(H_\alpha)<0\}. \]
This set of positive roots determines a Borel $\C$-subgroup containing $S$ and we write $\mf{u}_f$ for the Lie algebra of its unipotent radical. This $\mf{u}_f$ is stable under $\tilde S$ by construction, so the set of $\mf{u}$ that are stable under $\tilde S$ is non-empty. Moreover, the Lagrangian subspace of $\mf{g}/\mf{s}$ defined as the image of such $\mf{u}$ is ``totally complex'', i.e. it is transverse to its complex conjugate, and therefore \eqref{eq:duflo1} uniquely characterizes the genuine character $\rho_{f,\mf{u}}$ of $\tx{Mp}(\mf{g}/\mf{s})_{f,\mf{u}}$, see \cite[\S I.8]{Duflo82}.

The character $\rho_{f,\mf{u}}$ changes with respect to $\mf{u}$ by the following formula \cite[I.10 Case (ii)]{Duflo82}
\begin{equation} \label{eq:duflo2}
\rho_{f,\mf{u}_1}(\dot x)\rho_{f,\mf{u}_2}(\dot x)^{-1}=\det(x|\mf{u}_1/(\mf{u}_1 \cap \mf{u}_2)),\quad \dot x \in \tx{Mp}(\mf{g}/\mf{s})_{f,\mf{u}_1} \cap \tx{Mp}(\mf{g}/\mf{s})_{f,\mf{u}_2}.
\end{equation}
We pull back this character to $\tilde S(\R)_f$ and denote it still by $\rho_{f,\mf{u}}(x)$. Note that, while \eqref{eq:duflo1} still holds for this pull back, it no longer determines a unique character of $\tilde S(\R)_f$, because we can always multiply such a character by a sign character of $\tilde S(\R)$. Equation \eqref{eq:duflo2} holds for all $\dot x \in \tilde S(\R)_f$. The differential of the character $\rho_{f,\mf{u}} : \tilde S(\R)_f \to \C^\times$ is equal to $\rho_{\mf{u}}=\tfrac{1}{2}\sum_{\alpha>0}\alpha$.

\subsection{Comparing the metaplectic cover with the explicit cover}

In order to compare the constructions of \S\ref{sub:edc} and \S\ref{sub:mdc}, we will use the following abstract lemma.

Let $A$ be an abelian group and $A_1,A_2$ two extensions of $A$ by $\{\pm1\}$. Let $X$ be a finite set and for each $x \in X$ assume given a genuine character $\rho_{x,i} : A_i \to \C^\times$ for $i=1,2$. Assume that the following identities hold between characters of $X$:

\begin{eqnarray}
\label{eq:id1} \forall x \in X:&& \rho_{x,1}^2=\rho_{x,2}^2\\
\label{eq:id2} \forall x,x' \in X: &&\rho_{x,1}/\rho_{x',1}=\rho_{x,2}/\rho_{x',2}
\end{eqnarray}

\begin{lem} \label{lem:cover_id}
There exists a unique isomorphism $i : A_1 \to A_2$ of extensions such that $\rho_{1,x}=\rho_{2,x} \circ i$ for all $x \in A$.
\end{lem}
\begin{proof}
Fix $x \in X$. It is immediate that the right square in the commutative diagram
\[ \xymatrix{
1\ar[r]&\{\pm1\}\ar@{=}[d]\ar[r]&A_i\ar[d]^{\rho_{x,i}}\ar[r]&A\ar[d]^{\rho_{x,i}^2}\ar[r]&1\\
1\ar[r]&\{\pm1\}\ar[r]&\C^\times\ar[r]^{(-)^2}&\C^\times\ar[r]&1
}\]
is Cartesian. The uniqueness of the pull-back up to unique isomorphism implies the existence of a unique isomorphism $i : A_1 \to A_2$ satisfying $\rho_{x,1} = \rho_{x,2} \circ i$, for the fixed $x \in X$. It follows from \eqref{eq:id2} that this identity holds for all $x \in X$.
\end{proof}

Let $\tilde G$ be an affine algebraic $\R$-group whose identity component $G$ is reductive. Let $f \in \mf{g}^*(\R)$ be a regular elliptic semi-simple element and let $\tilde S \subset \tilde G$ be its centralizer, so that $S=\tilde S \cap G$ is a maximal $\R$-torus of $G$. We have the metaplectic double cover $\tilde S(\R)_f$ and the explicit double cover $S(\R)_G$.

\begin{cor} \label{cor:covers}
There exist a unique isomorphism $\tilde S(\R)_f \to \tilde S(\R)_G$ of double covers of $S(\R)$, which identifies the genuine character $\rho_{f,\mf{u}}$ of $\tilde S(\R)_f$ with the genuine character $\rho_B$ of $\tilde S(\R)_G$, where $B$ is any Borel $\C$-subgroup normalized by $\tilde S$ and $\mf{u}$ is its unipotent radical.
\end{cor}

\subsection{Regular genuine characters} \label{sub:genchar}

We continue with the set-up of \S\ref{sub:mdc}. Thus let $\tilde G$ be an affine algebraic $\R$-group whose identity component $G$ is reductive. Let $f \in \mf{g}^*(\R)$ be a regular semi-simple elliptic element and let $\tilde S \subset \tilde G$ be its centralizer. Then $S=\tilde S \cap G$ is an elliptic maximal $\R$-torus of $G$. We identify $f$ with its image under the projection $\mf{g}^* \to \mf{s}^*$.

If $\tau : S(\R) \to \C^\times$ is a continuous character, then its pull-back $\tau_\tx{sc}$ to the compact group $S_\tx{sc}(\R)$ is unitary, so the differential $d\tau_\tx{sc}$ is an $\R$-linear map $\tx{Lie}(S_\tx{sc})(\R) \to \tx{Lie}(\mb{S}^1)$. Using the identifications $\tx{Lie}(S_\tx{sc})(\R)=[X_*(S_\tx{sc}) \otimes_\Z \C]^\Gamma = X_*(S_\tx{sc}) \otimes_\Z i\R$ and $\tx{Lie}(\mb{S}^1)=i\R$ we can identify $d\tau_\tx{sc}$ with an element of $X^*(S_\tx{sc})\otimes_\Z\R = i\cdot \tx{Lie}^*(S_\tx{sc})(\R)$. In fact, this element lies in $X^*(S_\tx{sc})$, as one sees by fixing an arbitrary isomorphism $S_\tx{sc}(\R) \cong (\mb{S}^1)^a$, and hence provides an algebraic character of $S_\tx{sc}(\C)$, whose restriction to $S_\tx{sc}(\R)$ equals $\tau_\tx{sc}$. 

We say that $\tau$ is \emph{regular} if $d\tau_\tx{sc}$ is regular, in which case we can define a set of positive roots by
\[ R_\tau^+=\{\alpha \in R(S,G)|d\tau_\tx{sc}(H_\alpha)>0\}. \]
This set of positive roots determines a Borel $\C$-subgroup containing $S$ and we write $\mf{u}_\tau$ for the Lie algebra of its unipotent radical.

The differential $d\tau$ of the full character $\tau$ is an $\R$-linear form $\mf{s}(\R) \to \C$. If $\tau$ is unitary, then this differential becomes an $\R$-linear form $\mf{s}(\R) \to i\R$, hence an element of $i\mf{s}^*(\R)$. We have $d\tau=if$ if and only if the composition of $\tau$ with the exponential map $\mf{s}(\R) \to S(\R)$ equals the character $e^{if}$. In that case
\[ R_f^+=R_\tau^- \quad \tx{and} \quad \mf{u}_f = \mf{\bar u}_\tau. \]
Note that, if $S(\R)$ is connected (for example if $S$ is anisotropic), then $\tau$ is in fact determined by $d\tau$, hence by $f$, and we can write $\tau_f$.

Consider now instead a continuous genuine character $\tau : S(\R)_f \to \C^\times$. Since the cover $S(\R)_f$ splits over $S_\tx{sc}(\R)$, the pull-back of $\tau_\tx{sc}$ provides a character of $S_\tx{sc}(\R)$ and we can define what it means for $\tau$ to be regular, as well as $R_\tau^+$, in the same way as above. We may again require $d\tau=if$ and have the systems of positive roots $R_f^+=R_\tau^-$. Now, in addition, we have the genuine characters $\rho_{f,\mf{u}_f}=\rho_{f,\mf{u}_\tau}^{-1}$. Then
\begin{equation} \label{eq:taulhd}
\tau_\lhd := \tau \otimes \rho_{f,\mf{u}_f} = \tau \otimes \rho_{f,\mf{u}_\tau}^{-1}  
\end{equation}
is a character of $S(\R)$. Its differential is
\[ d\tau_\lhd = d\tau+\rho_f = d\tau-\rho_\tau, \qquad \rho_f = \frac{1}{2}\sum_{\alpha \in R_f^+}\alpha, \quad \rho_\tau = \frac{1}{2}\sum_{\alpha \in R_\tau^+}\alpha. \]
Note that, for any $w \in \Omega(S,G)(\R)$ we have
\[ w(\tau_\lhd) = (w\tau)_\lhd. \]
We can perform the same construction in the disconnected setting. Let $X(f)$ be the set of finite-dimensional genuine representations of $\tilde S(\R)_f$ whose whose composition with the exponential map $\mf{s}(\R) \to S(\R)_f \subset \tilde S(\R)_f$ is isotypic for the character $e^{if}$. Note that, if $\tilde\tau \in X(f)$ is itself a character, then $d\tilde\tau=if \in i\mf{s}^*(\R)$. Using the genuine character $\rho_{f,\mf{u}_f}$ of $\tilde S(\R)_f$ we can form for any $\tilde\tau \in X(f)$ 
\begin{equation} \label{eq:ttaulhd}
  \tilde\tau_\lhd := \tilde\tau \otimes \rho_{f,\mf{u}_f},
  \end{equation}
which is a finite-dimensional representation of $\tilde S(\R)$ and is isotypic for the character $e^{if+\rho_f}$ of $\mf{s}(\R)$.

\section{Square-integrable representations and automorphisms} \label{sec:dsaut}

Let $G$ be a quasi-split connected reductive $\R$-group, $(T,B,\{X_\alpha\})$ an $\R$-pinning, and $A$ a finite group of automorphisms of $G$ preserving $(T,B,\{X_\alpha\})$.

\subsection{Square-integrable representations}

Let $G'$ be an inner form of $G$. Harish-Chandra's work establishes a bijection between the set of irreducible square-integrable representations of $G'(\R)$ and the set of $G'(\R)$-conjugacy classes of pairs $(S',\tau)$, where $S' \subset G'$ is an elliptic maximal torus, and $\tau : S'(\R)_\rho \to \C^\times$ is a genuine character, regular in the sense of \S\ref{sub:genchar}; see \cite[Theorem 2.7.7]{DPR} for a review. We shall call such a conjugacy class a \emph{Harish-Chandra parameter}.

Let $f=d\tau/i \in \tx{Lie}^*(S)(\R)$. The composition of $\tau$ with $\exp : \tx{Lie}(S)(\R) \to S(\R)$ equals $e^{if}$. If $S(\R)$ is connected (for example, when $S$ is anisotropic), then the exponential map is surjective and $\tau$ is uniquely determined by $f$. Since $\tau$ is regular, $S=\tx{Cent}(f,G)$. Therefore, when $S(\R)$ is connected, we may refer to the $G(\R)$-conjugacy class of $f$ also as a Harish-Chandra parameter.

\subsection{$A$-admissible elliptic maximal tori}

Recall (\cite[Definition 3.6.1]{KalLLCD}) that a maximal $\C$-torus $S \subset G$ is called $A$-admissible, if it is $A$-stable and contained in an $A$-stable Borel $\C$-subgroup of $G$, and a maximal $\R$-torus $S$ is called $A$-admissible, if $S_\C$ is an $A$-admissible maximal $\C$-torus.

\begin{pro} \label{pro:aa}
  Assume that $G$ contains an elliptic maximal $\R$-torus. Then $G^1:=G^{A,\circ}$ contains an elliptic maximal $\R$-torus, $G$ contains an $A$-admissible elliptic maximal $\R$-torus, and the map $(S,C) \mapsto (S^1,C^1)$, where $S^1 = S^{A,\circ}$, $C^1=C^{A,\circ}$, is a bijection between the set of $A$-stable pairs $(S,C)$ consisting of an elliptic maximal $\R$-torus $S \subset G$ and a Borel $\C$-subgroup $C$ of $G$ containing $S$, and the set of pairs $(S^1,C^1)$ consisting of an elliptic maximal $\R$-torus $S^1$ of $G^1$ and a Borel $\C$-subgroup $C^1$ of $G^1$ containing $S^1$.
\end{pro}
\begin{proof}
The problem is insensitive to passing to the simply connected cover of the derived subgroup of $G$, so we assume that $G$ is semi-simple and simply connected. Then $G$ is the product of its $\R$-simple factors, which are permuted by $A$. We may further assume the action of $A$ on these factors is transitive, because the problem is compatible with products. 

Suppose $G_1$ is an $\R$-simple factor, $A_1$ is its stabilizer in $A$, and $S_1$ is an $A_1$-stable elliptic maximal $\R$-torus of $G_1$ contained in an $A$-stable $\C$-Borel subgroup $C_1$ of $G_1$. The $A$-orbit of $(S_1,C_1)$ is a pair $(S,C)$ consisting of an $A$-stable elliptic maximal $\R$-torus $S \subset G$ and $C$ is an $A$-stable $\C$-Borel subgroup of $G$ containing $S$. All $A$-stable pairs $(S,C)$ arise this way, and we have $X^{A,\circ}=X_1^{A_1,\circ}$ where $X$ is any of $G,S,C$.
  
We may therefore assume that $G$ is $\R$-simple. Then $G$ is either absolutely simple, or $G=\tx{Res}_{\C/\R}H$ for a simple $C$-group $H$. In the latter case, maximal $\R$-tori in $G$ are, up to $G(\R)$-conjugation, of the form $S=\tx{Res}_{\C/\R}S_H$ for maximal $\C$-tori $S_H \subset H$. But such a torus $S$ is never anisotropic, so the assumption that $G$ contains an anisotropic maximal torus implies that $G$ is absolutely simple.

We may replace $A$ by the quotient that acts effectively on $G$. Then $A$ is cyclic, except when $G$ is of type $D_4$, in which case $A$ could be the symmetric group $S_3$ on $3$ letters. Let $a$ be a generator of $A$ when $A$ is cyclic, and a generator of the unique subgroup of index $2$ in $A$ when $A$ is $S_3$. Then we have $G^A=G^a$: this is a tautology when $A$ is cyclic, and when $A$ is the symmetric group on 3 letters we have that $A/\<a\>$ is a group of order $2$ acting by automorphisms on $G^a$ that preserve the pinning of $G^a$ inherited from that of $G$ as in \cite[Proposition 3.4.2(9)]{KalLLCD}, but $G^a$ is of type $G_2$ and has no non-trivial pinned automorphisms, so $A/\<a\>$ acts trivially on $G^a$.

We will now show that $G^a$ contains an anisotropic torus. To that end, we will use the $a$-stable $\R$-pinning $(T,B,\{X_\alpha\})$ and we let $\Omega=N_G(T)/T$. The assumption that $G$ contains an anisotropic torus is equivalent to the existence of an element $\omega\in Z^1(\Gamma,\Omega)$ such that $\omega(\sigma)\rtimes\sigma$ acts on $X^*(T)$ by $-1$. It is then clear that $\omega(\sigma)\rtimes\sigma$ commutes with $a$ in the automorphism group of $X^*(T)$, hence in the automorphism group of $T$. But by assumption $a$ commutes with $\sigma$, and we conclude that $\omega(\sigma) \in \Omega^a$, i.e. $\omega \in Z^1(\Gamma,\Omega^a)$. But $\Omega^a$ is the Weyl group of the maximal torus $T^a$ of $G^a$ by \cite[Proposition 3.4.2(5)]{KalLLCD}, and we conclude that $G^a$ does indeed have an anisotropic maximal $\R$-torus.

From \cite[Proposition 3.4.2(7)]{KalLLCD} we know that $(S,C) \mapsto (S^1,C^1)$ is a bijection between the set of $A$-stable Borel pairs of $G_\C$ and the set of Borel pairs of $G^a_\C$. This bijection is obviously $\Gamma$-equivariant, so $S$ is defined over $\R$ if and only if $S^1$ is. We claim that $S$ is anisotropic if and only if $S^1$ is. Indeed, if $S$ is anisotropic, then $S^1$ must also be, since $S^1 \subset S$. Conversely, if $S^1$ is anisotropic, then $\sigma(C^1)=\bar C^1$ is the $\C$-Borel subgroup of $G^1$ that is $S^1$-opposite to $C^1$. If $\bar C$ is the Borel $\C$-subgroup of $G$ that is $S$-opposite to $C$, then $\bar C$ is $a$-stable and $\bar C \cap G^1=\bar C^1$. At the same time $\sigma(C)$ is also $a$-stable and $\sigma(C) \cap G^1=\sigma(C^1)$. From the bijectivity statement we conclude $\sigma(C)=\bar C$. Therefore, if $g \in G(\C)$ is any element that carries $(T,B)$ to $(S,C)$, then $g^{-1}\sigma(g)$ projects to the longest element $\omega \in \Omega$. But the assumption that $G$ has an anisotropic torus implies that $\omega \rtimes \sigma$ acts as $-1$ on $X^*(T)$, and we conclude that $\sigma$ acts as $-1$ on $X^*(S)$, and hence $S$ is anisotropic.

Finally, taking any pair $(S^1,C^1)$ in $G^1$ with $S^1$ anisotropic we conclude the existence of an $A$-stable pair $(S,C)$ with $S$ anisotropic.
\end{proof}

\begin{cor} \label{cor:aeonj}
  Any two $A$-admissible elliptic maximal tori of $G$ are conjugate under $G^1(\R)$
\end{cor}
\begin{proof}
  This follows from Proposition \ref{pro:aa} and the fact that any two elliptic maximal tori of $G^1$ are conjugate under $G^1(\R)$.
\end{proof}

\begin{lem} \label{lem:cov1}
  Let $S$ be an $A$-admissible maximal torus of $G$ and let $S^1 = S \cap G^1$. The inclusion $S^1 \to S$ lifts canonically to an inclusion $S^1(\R)_{\rho^1} \to S(\R)_\rho$ of covers.
\end{lem}
\begin{proof}
  Choose an $A$-stable $\C$-Borel subgroup $C$ containing $S$ and let $C^1=C \cap G^1$. Form $\rho \in \tfrac{1}{2}X^*(S)$ with respect to $C$ and denote by $\rho$ also its image under the restriction map $X^*(S) \to X^*(S^1)$. Then we have an obvious lifting $S^1(\R)_\rho \to S(\R)_\rho$. 
  
  The restriction of $\rho$ is equal to one half the sum of the $C^1$-positive roots in the set of restrictions $R(S^1,G)$ of $R(S,G)$ to $S^1$. According to \cite[Proposition 3.4.2(8)]{KalLLCD}, $R(S^1,G)$ is a possibly non-reduced root system and $R(S^1,G^1)$ is its subsystem of non-divisible roots. Therefore, if $\rho^1$ denotes one half the sum of the $C^1$-positive roots in $R(S^1,G^1)$, then $\rho-\rho^1 \in X^*(S^1)$. This gives an identification $S^1(\R)_{\rho_1}=S^1(\R)_\rho$.
\end{proof}

\subsection{$A$-admissible Harish-Chandra parameters}

Recall that a Whittaker datum $\mf{w}$ for $G$ is a $G(\R)$-conjugacy class of pairs $(B,\psi)$, where $B \subset G$ is a Borel $\R$-subgroup, and $\psi : U(\R) \to \C^\times$ is a generic unitary character of the unipotent radical $U$ of $B$. Recall that $\mf{w}$ is called $A$-admissible if it contains a pair $(B,\psi)$ that is $A$-stable. This is in general stronger than requiring that the conjugacy class $\mf{w}$ is $A$-stable.

\begin{dfn}
  A Harish-Chandra parameter is called 
  \begin{enumerate}
    \item \emph{$A$-stable}, if it is stable under $A$ as a $G(\R)$-conjugacy class.
    \item \emph{$A$-admissible}, if it contains a representative $(S,\tau)$ that is $A$-stable.
  \end{enumerate}
\end{dfn}

\begin{rem}
  \begin{enumerate}
    \item An irreducible square-integrable representation $\pi$ of $G(\R)$ is $A$-stable up to isomorphism if and only if its Harish-Chandra parameter is $A$-stable.
    \item If $(S,\tau)$ is an $A$-stable representative of a Harish-Chandra parameter, then the Weyl chamber $R_\tau^+$ defined in \S\ref{sub:genchar} is $A$-stable, and we conclude that $S$ is an $A$-admissible elliptic maximal torus.
  \end{enumerate}
\end{rem}

\begin{pro} \label{pro:adgen}
  For an $A$-stable Harish-Chandra parameter the following are equivalent. 
  \begin{enumerate}
    \item It contains an $A$-stable representative $(S,\tau)$ such that each simple root of $R_\tau^+$ is non-compact.
    \item The corresponding square-integrable representation $\pi$ of $G(\R)$ is generic with respect to an $A$-admissible Whittaker datum.
  \end{enumerate}
\end{pro}
\begin{proof}
For the duration of the proof fix the unitary character $\Lambda : \R \to \C^\times$ given by $\Lambda(x)=e^{ix}$.

Assume that $(S,\tau)$ is an $A$-stable representative with each simple root of $R_\tau^+$ non-compact. Let $C$ be the corresponding $\C$-Borel subgroup of $G$ containing $S$. Let $(S^1,C^1)$ be the corresponding pair of $G^1$ as in Proposition \ref{pro:aa}, so that $S^1 \subset G^1$ is an elliptic maximal torus. According to \cite[Lemma 3.6.3]{KalLLCD}, all $C^1$-simple roots of $S^1$ in $G^1$ are still non-compact. Consider the restriction $\tau_1$ of $\tau$ along the inclusion $S^1(\R)_{\rho^1} \to S(\R)_\rho$ of Lemma \ref{lem:cov1}. It is still a regular genuine character, hence represents a Harish-Chandra parameter for $G^1$. The corresponding representation $\pi_1$ is then generic with respect to some Whittaker datum of $G^1$ according to \cite[Theorem 6.2(a,f)]{Vog78}, which using \cite[Lemma 3.1.2(2)]{DPR} and the chosen $\Lambda$ we can represent by a pinning of $G^1$. Then the $G^1(\R)$-orbit of $d\tau_1 \in (\mf{s}^1)^*$ meets the Kostant section $\mc{K}^1 \subset (\mf{g}^1)^*$ for that pinning according to \cite[Proposition 3.3.3]{DPR}. More precisely, the pinning leads to a regular nilpotent element $X \in (\mf{g}^1)^*$. We choose an isomorphism $\mf{g}^1 \to (\mf{g}^1)^*$ that is equivariant for the adjoint action of $G^1$ and the Galois group, use it to view $X$ as  a regular nilpotent element of $\mf{g}^1$, extend it to an $\mf{sl}_2$-triple $(X,H,Y)$ in $\mf{g}^1$, form $\mc{K}^1:=X+Z(Y,\mf{g}^1) \subset \mf{g}^1$, and view this as a subspace of $(\mf{g}^1)^*$ via the chosen isomorphism. As argued in \cite[\S3.2]{DPR}, the resulting subspace of $(\mf{g}^1)^*$ depends only on $X$, and not on the chosen isomorphism $\mf{g}^1 \to (\mf{g}^1)^*$.

According to \cite[Proposition 3.4.2(9)]{KalLLCD}, there is an associated $A$-stable pinning of $G$. The regular nilpotent element of $\mf{g}^*$ that it provides equals $X$ under the inclusion $(\mf{g}^1)^* = (\mf{g}^*)^A \into \mf{g}^*$ as the subspace of $A$-invariants. Choose an isomorphism $\mf{g} \to \mf{g}^*$ that is equivariant under the adjoint action of $G$ as well as the Galois action and the $A$-action. It restricts to an isomorphism $\mf{g}^1 \to (\mf{g}^1)^*$ invariant under the adjoint action of $G^1$ and the Galois action. We transport $X$ under this isomorphism and extend to an $\mf{sl}_2$-triple in $\mf{g}^1 \subset \mf{g}$. Then $\mc{K}^1=X+Z(Y,\mf{g}^1) \subset \mc{K}=X+Z(Y,\mf{g})$. We conclude that the $G^1(\R)$-orbit of $d\tau_1$ meets $\mc{K}$. But $d\tau$ is $A$-fixed, hence lies in $(\mf{s}^*)^A$, and thus equals $d\tau_1$. We apply again \cite[Proposition 3.3.3]{DPR} and conclude that the representation $\pi$ with Harish-Chandra parameter $(S,\tau)$ is generic for the Whittaker datum associated to $\Lambda$ and the $A$-invariant pinning of $G$.

Conversely, assume that $\pi$ has an $A$-stable Harish-Chandra parameter and is generic with respect to an $A$-admissible Whittaker datum, and realize the latter via an $A$-stable pinning of $G$ and the character $\Lambda$. Thus $d\tau$ meets the Kostant section $\mc{K}$ for that pinning. Since both the Harish-Chandra parameter is $A$-stable, so is the $G(\R)$-conjugacy class of $d\tau$. The intersection of the $G(\R)$-conjugacy class of $d\tau$ and $\mc{K}$ is then also $A$-stable, but it is a single point. Therefore, up to conjugating by $G(\R)$, we may assume that $d\tau$ is $A$-fixed. But then $S$, being the centralizer of $d\tau$, is $A$-stable, and $\tau$, which is determined within its $\Omega_\R$-orbit by $d\tau$, is also $A$-stable.
\end{proof}

\section{Square-integrable representations of disconnected real groups} \label{sec:dsdisc}

Let $\tilde G$ be an affine algebraic $\R$-group whose identity component $G$ is reductive.

Let $\tilde\pi$ be an irreducible admissible representation of the Lie group $\tilde G(\R)$. Its restriction to $G(\R)$ is of finite length and semi-simple, and the individual irreducible summands are permuted transitively by the action of $\tilde G(\R)/G(\R)$.

\begin{dfn}
We say that $\tilde\pi$ is (essentially) square-integrable if one, hence any, irreducible constituent of $\tilde\pi|_{G(\R)}$ is (essentially) square-integrable in the usual sense.
\end{dfn}

\subsection{Duflo's construction} \label{sub:duflo}

Let $f \in \mf{g}^*(\R)$ be a regular semi-simple elliptic element and let $\tilde S$ be its centralizer in $\tilde G$. Thus $S=\tilde S \cap G$ is an elliptic maximal torus of $G$. We have the metaplectic double cover $\tilde S(\R)_f$ reviewed in \S\ref{sub:mdc}.
We map $f$ to $\mf{s}^*(\R)$ under the natural projection $\mf{g}^* \to \mf{s}^*$. Recall from \S\ref{sub:genchar} the set $X(f)$ of finite dimensional genuine representations of $\tilde S(\R)_f$ that are isotypic for the character $e^{if}$ of $\mf{s}(\R)$.

For any $\tilde\tau \in X(f)$ Duflo constructs in \cite[\S III]{Duflo82} an admissible representation $\pi_{f,\tilde\tau}$ of $\tilde G(\R)$. We reproduce the summary of this construction that is given in \cite[\S5.4]{Bouaziz87}. Let $\pi_f$ be the irreducible discrete series representation of $G(\R)^0$ associated to $f$ by Harish-Chandra. Let $\mf{u}$ be the unipotent radical of a Borel $\C$-subalgebra of $G$ containing $S$ and normalized by $\tilde S$. Such a Borel subalgebra exists -- one can take for example the one determined by the system of positive roots $R_f^+=\{\alpha \in R(S,G)|if(H_\alpha)<0\}$, in which case we will write $\mf{u}_f$ for the Lie algebra of its unipotent radical. Its image in $\mf{g}/\mf{s}$ is a Lagrangian subspace for the symplectic form $B_f(X,Y)=f([X,Y])$. Let $\rho_\mf{u}=\tfrac{1}{2}\sum_{\alpha>0}\alpha$ be half the sum of the $\mf{u}$-positive absolute roots. There exists a unique natural number $q_\mf{u}^f$, which we will recall in \S\ref{sub:q}, such that the $(if+\rho_\mf{u})$-weight space of the $\mf{s}$-module $H_j(\mf{u},\pi_f^\infty)$ vanishes unless $j = q_\mf{u}^f$, in which case it is 1-dimensional. There exists a unique unitary genuine representation $\Pi_{f,\mf{u}}$ of $\tilde S(\R)_f$ on the vector space underlying $\pi_f$ (note that $\Pi_{f,\mf{u}}$ is denoted by $S_{f,\mf{u}}$ in loc. cit.) with the properties
\begin{enumerate}
  \item $\Pi_{f,\mf{u}}(\dot x)\pi_f(y)\Pi_{f,\mf{u}}(\dot x)^{-1}=\pi_f(xyx^{-1})$, for all $\dot x \in \tilde S(\R)_f$ with image $x \in \tilde S(\R)$,
  \item on the one-dimensional $(if+\rho_\mf{u})$-weight space for the action of $\mf{s}$ on $H_{q_\mf{u}^f}(\mf{u},\pi_f^\infty)$, the action of $\tilde S(\R)_f$ via $\Pi_{f,\mf{u}}$ is through the genuine character $\rho_{f,\mf{u}}$.
\end{enumerate}
The representation $\pi_f$ extends to a representation $\tilde\pi_f$ of the group $G(\R)^0 \rtimes \tilde S(\R)_f$, with $y \rtimes \dot x$ acting via the endomorphism $\pi_f(y) \circ \Pi_{f,\mf{u}}(\dot x)$ on the underlying vector space, where $y \in G(\R)^0$ and $\dot x \in \tilde S(\R)_f$. Inflate $\tilde\tau$ to $G(\R)^0 \rtimes \tilde S(\R)_f$ and form the tensor product $\tilde\tau \otimes \tilde\pi_f$. This tensor product descends to $G(\R)^0 \rtimes \tilde S(\R)$, and then further to $G(\R)^0 \cdot \tilde S(\R) \subset \tilde G(\R)$. The induction of this representation to $\tilde G(\R)$ is $\pi_{f,\tilde\tau}$. It is independent of the choice of $\mf{u}$.

\begin{lem} \label{lem:duflo1}
\begin{enumerate}
  \item $\pi_{f,\tilde\tau}$ is irreducible if and only if $\tilde\tau$ is.

  \item Consider an intermediate group $G \subset \tilde G_1 \subset \tilde G$. For any $c \in \tilde G(\R)$ let $\tilde S_1^c = c^{-1}\tilde Sc \cap \tilde G_1$, $f^c=\tx{Ad}^*(c)^{-1}f \in \mf{g}^*$, $\tilde\tau_1^c = \tilde\tau \circ \tx{Ad}(c)|_{\tilde S_1^c(\R)_{f^c}}$, and let $\pi_{f^c,\tilde\tau_1^c}$ be the representation of $\tilde G_1(\R)$ assigned by Duflo's construction to $f^c$ and $\tilde\tau_1^c$. Then 
  \[ \pi_{f,\tilde\tau}|_{\tilde G_1(\R)} = \bigoplus_{c \in \tilde S(\R) \lmod \tilde G(\R) / \tilde G_1(\R)} \pi_{f^c,\tilde\tau_1^c}. \]

  \item This construction produces all square-integrable representations of $\tilde G(\R)$.
  \item Two pairs $(f_1,\tilde\tau_1)$ and $(f_2,\tilde\tau_2)$ produce isomorphic irreducible representations if and only if they are $\tilde G(\R)$-conjugate.
\end{enumerate}
\end{lem}
\begin{proof}
(1) and (4) are \cite[III.5.Lemma 7]{Duflo82}, noting that an elliptic $f$ is ``standard'' in Duflo's terminology.

(2) Let $A=G(\R)^0 \cdot \tilde S(\R)$, $B=\tilde G_1(\R)$, and $C=\tilde G(\R)$. According to the Mackey formula
\[ \pi_{f,\tilde\tau}|_{\tilde G_1(\R)} = \bigoplus_{c \in A \lmod C/B} \tx{Ind}_{A^c \cap B}^{B}\tx{Res}_{A^c \cap B}^{A^c} (\tilde\tau \otimes\tilde\pi_f)^c, \]
where $A^c = c^{-1}Ac$. Now $G(\R)^0$ is normal in $\tilde G(\R)$, so $A^c=G(\R)^0 \cdot c^{-1}\tilde S(\R)c$ and $A^c \cap B=G(\R)^0 \cdot \tilde S_1^c(\R)$. Let us write $\pi_{f^c}$ for the irreducible discrete series representation of $G(\R)^0$ associated to $f^c$. The centralizer of $f^c$ in $\tilde G_1$ equals $\tilde S_1^c$ and we have the double cover $\tilde S_1^c(\R)_{f^c}$. Then the $c$-summand in the above direct sum equals the representation of $\tilde G_1(\R)$ associated by Duflo's construction to $f^c \in \mf{g}^*(\R)$ and $\tilde\tau^c|_{\tilde S_1^c(\R)_{f^c}}$, i.e. the representation $\pi_{f^c,\tilde\tau_1^c}$.

(3) By definition a square-integrable representation of $\tilde G(\R)$ is one whose restriction to $G(\R)^0$ contains a representation of the form $\pi_f$. The claim follows from Clifford theory.
\end{proof}

\subsection{Normalized vectors}
  
In this subsection $S$ is an elliptic maximal torus of a connected reductive $\R$-group $G$.

  \begin{dfn} \label{dfn:pn_normalized}
    Let $\alpha \in R(S,G)$. A root vector $X \in \mf{g}_\alpha$ is called \emph{normalized} if $[X,\sigma(X)]=\epsilon_\alpha \cdot H_\alpha$ with $\epsilon_\alpha \in \{\pm 1\}$, where $\sigma \in \Gamma$ denotes complex conjugation.
  \end{dfn}
  
  \begin{fct} \label{fct:pn_normalized}
  A normalized root vector always exists. Any two differ by multiplication by an element of $\mb{S}^1$. We have $\epsilon_\alpha=+1$ if and only if $\alpha$ is non-compact.
  \end{fct}

  \begin{lem} \label{lem:pn_normtrans}
    Let $\Delta \subset R(S,G)$ be a set of simple roots and let $\{X_\alpha\}$ and $\{X'_\alpha\}$ be two sets of normalized root vectors indexed by $\alpha \in \Delta$. There exists $s \in S(\R)$ such that $X'_\alpha=\tx{Ad}(s)X_\alpha$.
    \end{lem}
    \begin{proof}
    Any two normalized root vectors in $\mf{g}_\alpha$ differ by rescaling by an element of $\mb{S}^1$ according to Fact \ref{fct:pn_normalized}. Thus, if $\{X_\alpha\}$ and $\{X'_\alpha\}$ are two choices of normalized simple root vectors, there exists a collection of scalars $c_\alpha \in \mb{S}^1$ indexed by $\alpha \in \Delta$ such that $X'_\alpha=c_\alpha X_\alpha$. Let $s \in S_\tx{ad}(\R)$ be unique element such that $\alpha(s)=c_\alpha$. Thus $X'_\alpha=\tx{Ad}(s)X_\alpha$. By Lemma \ref{lem:elem} below, the map $S(\R) \to S_\tx{ad}(\R)$ is surjective. 
    \end{proof}
  
    \begin{lem} \label{lem:elem}
      Let $1 \to A \to B \to C \to 1$ be an exact sequence of diagonalizable $\R$-groups. Assume $C$ is connected and anisotropic. Then $B(\R) \to C(\R)$ is surjective and $H^1(\R,A) \to H^1(\R,B)$ is injective.
    \end{lem}
    \begin{proof}
    By the long exact sequence of Galois cohomology, the two claims are equivalent. We prove the first. The homomorphism $\tx{Lie}(B) \to \tx{Lie}(C)$ of complex vector spaces is surjective, and induces a surjective homomorphism of real vector spaces $\tx{Lie}(B)(\R) \to \tx{Lie}(C)(\R)$. Therefore, the homomorphism of real Lie groups $B(\R) \to C(\R)$ is submersive, hence open. By assumption $C$ is an anisotropic $\R$-torus, hence $C(\R)$ is connected. Thus $C(\R)$ is generated by any open neighborhood of the identity, and we conclude that $B(\R) \to C(\R)$ is surjective.
    \end{proof}

\subsection{The number $q_\mf{u}^f$} \label{sub:q}

We maintain the notation that $f \in \mf{g}^*(\R)$ is a regular semi-simple elliptic element, $S \subset G$ its centralizer, and $\mf{u}$ the nilpotent radical of a Borel $\C$-subalgebra containing $\mf{s}$. 

\begin{dfn} \label{dfn:qfu}
$q^f_\mf{u}$ is the number of negative eigenvalues of the matrix representing the Hermitian form $(X,Y) \mapsto i \cdot B_f(X,\sigma(Y))$ on $\mf{u}$, where $\sigma$ denote complex conjugation.
\end{dfn}

Note that, since $S$ is elliptic, complex conjugation sends $\mf{u}$ to the opposite radical, which makes the above form non-degenerate on $\mf{u}$ and $B_f$ is the symplectic form on $\mf{g}/\mf{s}$ recalled in \S\ref{sub:bf}.

\begin{dfn} \label{dfn:pn_rp}
\begin{enumerate}
  \item $R_f^+ = \{\alpha \in R| if(H_\alpha) < 0 \}$.
  \item $R_\mf{u}^+=\{\alpha \in R|\mf{g}_\alpha \subset \mf{u}\}$
  \item $\mf{u}_f$ is the unipotent radical for which $R_f^+=R_\mf{u}^+$.
\end{enumerate}
\end{dfn}

\begin{lem} \label{lem:pn_qvuf}
Let $w \in W$ be the Weyl element such that $w\mf{u}=\mf{u}_f$. Then
$q^f_\mf{u}$ equals the sum of $\ell(w)$ and the number of non-compact roots in $\mf{u}$, modulo $2$.
\end{lem}
\begin{proof}
For each $\alpha$ choose a normalized root vector $X_\alpha$ in the sense of Definition \ref{dfn:pn_normalized}. Using this basis for $\mf{u}$ we obtain for $iB_f$ a diagonal matrix, whose entries according to Fact \ref{fct:pn_normalized} are $+1$ when either $\alpha$ non-compact and $if(H_\alpha)>0$, or $\alpha$ compact and $if(H_\alpha)<0$, and $-1$ otherwise. Therefore $q^f_\mf{u}$ equals 
\[ \#\{\alpha\ \tx{non-comp}, if(H_\alpha)<0\} + \#\{\alpha\ \tx{comp}, if(H_\alpha)>0\}, \]
where all roots are taken from $\mf{u}$. Now
\[ \#\{\alpha\ \tx{comp}, if(H_\alpha)>0\} + \#\{\alpha\ \tx{non. comp}, if(H_\alpha)>0\} \]
equals the number of roots in $R_\mf{u}^+ \cap R_f^-$, and this equals $\ell(w)$. Therefore, modulo $2$,
\[ q_\mf{u}^f = \#\{\alpha\ \tx{non-comp}\} + \ell(w). \]
\end{proof}

\begin{cor} \label{cor:pn_qvuf}
For any $v \in W$ we have $q_{v\mf{u}}^{vf}=q_{\mf{u}}^f$, and $(-1)^{q_{\mf{u}}^{vf}}=(-1)^{q(G)}\epsilon(v)$.
\end{cor}

\begin{lem} \label{lem:pn_negative}
Consider the character $\chi : S_\tx{sc}(\R) \to \mb{S}^1$ given by the factorization of 
\[ \tx{Lie}(S_\tx{sc}(\R)) \to \mb{S}^1,\qquad X \mapsto e^{if(X)} \]
through $\exp : \tx{Lie}(S_\tx{sc}(\R)) \to S_\tx{sc}(\R)$. Then $d\chi \in X^*(S_\tx{sc}) \subset \tx{Lie}^*(S_\tx{sc})$ is the image of $if$ under $\tx{Lie}^*(G) \to \tx{Lie}^*(S) \to \tx{Lie}^*(S_\tx{sc})$. In particular, 
\[ R_f^+ = \{\alpha|\<H_\alpha,d\chi\> < 0\}. \]
\end{lem}
\begin{proof}
The first claim follows by differentiating the identity $\chi(\exp(X))=e^{if(X)}$ for $X \in \tx{Lie}(S_\tx{sc})(\R)$ at $X=0$ in the direction of $Y \in \tx{Lie}(S_\tx{sc})(\R)$, and using that the differential of $\exp$ is the identity. The second claim follows from Definition \ref{dfn:pn_rp}.
\end{proof}

\subsection{Bouaziz's character formula on elliptic elements: statement}

We continue with a regular semi-simple elliptic element $f \in \mf{g}^*(\R)$, $\tilde S=Z_{\tilde G}(f)$, and $\tilde\tau \in X(f)$. In \cite{Bouaziz87}, Bouaziz computed the character of $\pi_{f,\tilde\tau}$. We state his formula in the special case that will be relevant to us: that of regular semi-simple elliptic elements. We first review these notions.

Recall from \cite{Ste68end} that an automorphism $x$ of $G$ is called \emph{quasi-semisimple} (qss for short) if its restriction to $G_\tx{der}$ is semi-simple, equivalently if it preserves a Borel pair of $G$. A qss automorphism $x$ of $G$ is called \emph{regular}, if $(G^x)^\circ$ is a torus, and \emph{strongly regular}, if $G^x$ is abelian. In the latter case $(G^x)^\circ$ is a torus. In either case, the centralizer in $G$ of $(G^x)^\circ$ is a maximal $\R$-torus $T$ of $G$ stabilized by $x$. If $x$ is strongly regular, then $G^x \subset T$, hence $G^x=T^x$ and $T$ is the centralizer of $G^x$ in $G$.

\begin{lem} \label{lem:ell}
Let $x$ be a strongly regular qss automorphism of $G$ and let $T$ be the centralizer of $G^x$ in $G$. The following are equivalent.
\begin{enumerate}
  \item $T^x/Z(G)^x$ is anisotropic.
  \item $T/Z(G)$ is anisotropic.
\end{enumerate}
\end{lem}
\begin{proof}
$(2) \Rightarrow (1)$: $T^x/Z(G)^x$ is a subgroup of the anisotropic torus $T/Z(G)$, hence anisotropic.

$(1) \Rightarrow (2)$: By assumption $x$ is a strongly regular qss automorphism of $G_\tx{der}$. The centralizer of $(G_\tx{der})^x$ in $G$ is therefore a maximal torus of $G$ that contains $T$, hence equals $T$. It is therefore enough to prove that the centralizer of $(G_\tx{der})^x$ in $G_\tx{der}$ is anisotropic. We may therefore assume that $G$ is semi-simple.

Choose a Cartan involution $\theta$ stabilizing $T$ and let $K=G^\theta$, so that $K(\R)$ is a maximal compact subgroup of $G(\R)$. By assumption $G^x(\R)$ is a compact subgroup of $T(\R)$ and hence contained in $K(\R)$. Thus $G^x \subset K$. Now $G^x \cap K \subset T \cap K$ is contained in a maximal $\R$-torus of $K$. Since $G$ has discrete series, this maximal torus of $K$ is also a maximal torus of $G$. But the only maximal torus of $G$ that contains $G^x$ is $T$, and we conclude $T \subset K$, i.e. $T(\R)$ is compact.
\end{proof}

\begin{rem}
  The implication $(1) \Rightarrow (2)$ of Lemma \ref{lem:ell} used the assumption that $G$ has discrete series representations. Without this assumption, this implication is false. For example, let $G=\tx{Res}_{\C/\R}\tx{SL}_2$ and let $A$ denote the $\R$-automorphism coming from the $\R$-structure on $\tx{SL}_2$. Explicitly, if we present $G(\C)=\tx{SL}_2(\C) \times \tx{SL}_2(\C)$, then complex conjugation acts as $\sigma(g_1,g_2)=(\bar g_2,\bar g_1)$ while $A(g_1,g_2)=(g_2,g_1)$, where $g \mapsto \bar g$ is the action of complex conjugation on $\tx{SL}_2(\C)$ coming from the $\R$-structure of $\tx{SL}_2$. Let $S' \subset \tx{SL}_2$ be an anisotropic maximal torus. Then $S=\tx{Res}_{\C/\R}S'$ is an $A$-admissible maximal torus of $G$. For any $s=(s_1,s_2) \in S(\R)$ the element $sA$ is semi-simple (its square lies in $S$; alternatively, it normalizes $S$ and the Borel $\C$-subgroup $(B^\tx{opp},B)$ for any Borel $\C$-subgroup $B$ of $G'$ containing $S'$). It is strongly regular if the element $(sA)^2$ of $\tx{SL}_2(\C)$ is so. We have $S(\C)^{sA}=S'(\R) \cong \mb{S}^1$, while $S(\R)=S'(\C)=\C^\times$. Thus $S^{sA}$ is anisotropic, but $S$ itself is not.
\end{rem}

\begin{dfn} \label{dfn:ell}
  We shall call a strongly regular qss automorphism of $G$ \emph{elliptic} if they satisfy the equivalent conditions of Lemma \ref{lem:ell}. An element $x \in \tilde G(\R)$ will be called \emph{strongly regular semi-simple elliptic}, if it induces a strongly regular qss automorphism of $G$. 
\end{dfn}

Note that such an element can always be conjugated under $G(\R)$ so that any element that centralizes it also centralizes $f$, according to part (2) of Lemma \ref{lem:ell}. In other words, we may assume that $G^x \subset S$. Recall the notation $\tilde\tau_\lhd = \tilde\tau \cdot \rho_{f,\mf{u}_f}$ from \eqref{eq:ttaulhd}.

\begin{thm}[Bouaziz] \label{thm:bouaziz}
Let $x \in \tilde G(\R)$ be an elliptic strongly regular semi-simple element. Assume that $G^x \subset S$. Then
\[ \Theta_{\pi_{f,\tilde\tau}}(x) = (-1)^{q(G)}\sum_{\substack{w \in N_{\tilde G(\R)}(S)/\tilde S(\R) \\ w^{-1}xw \in \tilde S(\R)}}\frac{\tx{tr}(\tilde\tau_\lhd(w^{-1}xw))}{\det(1-w^{-1}xw|\mf{u}_f)}. \]
\end{thm}

This statement differs optically from the formulas in \cite{Bouaziz87}. In the following subsection we are going to discuss how to derive it from Bouaziz's statements. For now, we compare it with the classical statement for connected reductive groups due to Harish-Chandra.

\begin{cor}
  Assume that $\tilde G=G$ is connected semisimple and that $\tau$ is irreducible, thus a character $S(\R)_f \to \C^\times$. Let $x \in S(\R)$ be strongly regular. 
  \[ \Theta_{\pi_{\tau}}(x) = (-1)^{q(G)}\sum_{w  \in N_{G(\R)}(S)/S(\R)}\frac{e^{d\tau-\rho}(w^{-1}xw)}{\prod_{\alpha>0}(1-e^{-\alpha}(w^{-1}xw))}, \]
  where the positive system of roots is $\{\alpha|\<H_\alpha,d\tau\> > 0\}$ and $\rho$ is half of the sum of these positive roots.
  \end{cor}
  \begin{proof}
  Since $S(\R)_f$ is abelian, $\tau$ must be a character. As discussed in \S\ref{sub:genchar}, we have $d\tau=if$ and $d\tau_\lhd = d\tau-\rho$, where now we are writing $\rho=\rho_\tau$.
  \end{proof}

  \subsection{Bouaziz's character formula on elliptic elements: discussion}

In this subsection we discuss how Theorem \ref{thm:bouaziz} is a reformulation of (a special case of) the main result of \cite{Bouaziz87}.

We continue to assume that $\tilde G$ is a (possibly disconnected) affine algebraic $\R$-group with reductive identity component $G$, $f \in \mf{g}^*(\R)$ regular semi-simple elliptic, $\tilde S=Z_{\tilde G}(f)$. 

We are applying Bouaziz's main theorem (page 2 loc. cit.) to the case of the regular element $s=x$. In the notation there, the character of $\pi_{f,\tilde\tau}$ evaluated at $x$ is written as $\theta(x,f,\tilde\tau)(0)=\Theta(f,\tilde\tau)(x)$, and the main result  gives the formula
\begin{equation} \label{eq:b1}
  \sum_{f' \in \Omega_x^f} \varphi_{f,\tilde\tau}^x(f'),
\end{equation}
where $\Omega_x^f=\tx{Ad}(\tilde G(\R))^*f \cap \mf{z}^*$, $\mf{z}^* =(\mf{g}^*)^x$, and $\varphi_{f,\tilde\tau}^x(f')$ is described in \cite[(5.5.1)]{Bouaziz87}. We will first discuss the indexing set of the sum, and then move on to the summands.

\begin{lem}
  The map $g \mapsto \tx{Ad}^*(g)f$ is a bijection 
  \[ \{g \in N_{\tilde G(\R)}(S)/\tilde S(\R)\,|\, g^{-1}xg \in \tilde S(\R) \} \to \Omega_x^f. \]
\end{lem}
\begin{proof}
Any $g \in \tilde G(\R)$ with $g^{-1}xg \in \tilde S(\R)$ satisfies $\tx{Ad}^*(g)f \in \mf{z}^*$, so the map is well-defined. It is injective, because the stabilizer of $f$ in $\tilde G(\R)$ is $\tilde S(\R)$ by definition. To show surjectivelty, let $g \in \tilde G(\R)$ be such that $\tx{Ad}(g)^*f$ lies in $\mf{z}^*$, equivalently $\tx{Ad}(g)^*f$ is fixed by $\tx{Ad}(x)^*$. Then $f$ is fixed by $\tx{Ad}^*(g^{-1}xg)$, so $g^{-1}xg \in \tilde S(\R)$. We claim this implies that $g$ normalizes $S$. Indeed, we have assumed $(G^x)^\circ \subset S$, and hence $S$ is the centralizer in $G$ of $(G^x)^\circ$. On the other hand, $g^{-1}xg$ fixes $f$, hence normalizes $S$ as well as the Borel $\C$-subgroup containing $S$ that is determined by $f$. According to \cite[Theorem 1.1.A(4)]{KS99} the intersection $S \cap (G^{g^{-1}xg})^\circ$ is a maximal torus in $(G^{g^{-1}xg})^\circ$, hence equal to $(G^{g^{-1}xg})^\circ$, and moreover $S$ is the centralizer in $G$ of $(G^{g^{-1}xg})^\circ$. This proves the claim that $g$ normalizes $S$.  
\end{proof}

With this, \eqref{eq:b1} becomes
\begin{equation} \label{eq:b2}
  \sum_{\substack{w \in N_{\tilde G(\R)}(S)/\tilde S(\R) \\ w^{-1}xw \in \tilde S(\R)}} \varphi_{f,\tilde\tau}^x(wf).  
\end{equation}

We now turn to the summand $\varphi_{f,\tilde\tau}^x(wf)$ defined in \cite[(5.5.1)]{Bouaziz87}, following work of Duflo--Heckman--Vergne \cite{DHV84} in the connected case. We warn the reader that there is a typo in \cite[(5.5.1)]{Bouaziz87}: a tilde is missing over the term $l \cap \mf{q}_\C$ in the denominator, see Remark \ref{rem:typo} below. Thus, that formula should read 
\begin{equation} \label{eq:551}
 \varphi_{f_0,\tau_0}^s(f) = \frac{(-1)^{q^f_{\widetilde{l \cap \mf{q}_\C}}}}{\det(1-s)_{\widetilde{l \cap \mf{q}_\C}}}\tx{tr}\,((\tau_{0_f}\rho_l^f)(s)).
\end{equation}
We first review the notation in this formula, before adapting it to our needs. We will also momentarily adopt the notation of Bouaziz's paper, for example writing $\mf{g}^*$ in place of $\mf{g}^*(\R)$ and $\mf{g}^*_\C$ in place of $\mf{g}^*$. Then $f_0 \in \mf{g}^*$ is a fixed element, which in the language of Duflo and Bouaziz is ``well-polarizable''; in our setting $\mf{g}$ is reductive and ``well-polarizable'' is equivalent to regular semi-simple. For our purposes we assume in addition that $f_0$ is elliptic, which implies that $f_0$ is ``standard'' in the language of Duflo and Bouaziz, i.e. it admits a ``totally complex polarization''. Such a polarization is simply a Borel $\C$-subalgebra containing $f_0$, and totally complex refers to the fact that complex conjugation sends this Borel subalgebra to its opposite. These notions are reviewed in \S5.2 of Bouaziz's paper.

Bouaziz takes $\tilde S$ to be the centralizer of $f_0$ in $\tilde G$, and $\tau_0$ to be a finite\-dimensional genuine representation of the metaplectic cover $\tilde S(\R)_{f_0}$ (reviewed in \S\ref{sub:mdc} here), whose restriction to the image of the exponential map is isotypic for the character $e^{if_0}$. The existence of such $\tau_0$ makes $f_0$ ``admissible'' by definition, another notion introduced in \S5.2 of Bouaziz's paper. 

The element $s \in \tilde G(\R)$ in \eqref{eq:551} is elliptic, but in general not assumed regular. This allows one to use \eqref{eq:551} for the computation of character values at non-elliptic regular semi-simple elements, by multiplying $s$ with an element that centralizes it, and such that the product is regular. We will not need this flexibility, as our focus will be on regular elliptic elements.

The element $f \in \mf{g}^*(\R)$ is $\tilde G(\R)$-conjugate to $f_0$, i.e. an element of the $\tilde G(\R)$-conjugacy class $\Omega_{f_0}$ of $f_0$. It is required that $f$ is fixed by $\tx{Ad}(s)$, i.e. $f \in \Omega_{f_0}^s$.

By definition $\mf{q}_\C=(1-s)\mf{g}_\C \subset \mf{g}_\C$, so that $\mf{g}_\C = \mf{q}_\C \oplus \mf{g}_\C^s$.

The symbol $l$ denotes a subspace of $\mf{g}_\C$ whose image in $\mf{g}_\C/\mf{g}_\C(f)$ is Lagrangian for the symplectic form $B_f$ reviewed in \S\ref{sub:bf}. The tilde notation signifies taking image in $\mf{g}_\C/\mf{g}_\C(f)$. 

The number $q^f_{\widetilde{l \cap \mf{q}_\C}}$ is the number of negative eigenvalues of the matrix representing the Hermitian form $(X,Y) \mapsto iB_f(X,\sigma(Y))$ on the Lagrangian subspace $\widetilde{l \cap \mf{q}_\C} \subset \mf{g}_\C/\mf{g}_\C(f)$.

The term $\tau_{0f}$ is a representation of the metaplectic cover of the centralizer of $f$ in $\tilde G(\R)$, obtained by $g$-conjugating $\tau_0$ for any $g \in \tilde G(\R)$ satisfying $gf_0=f$, while $\rho_l^f$ is the metaplectic character of that same cover. Even though $\rho_l^f$ has $f$ as superscript, it depends only on the Lagrangian and on the centralizer of $f$, see discussion of \S5.2, especially line after (5.2.1), of loc. cit. 

This completes the review of \eqref{eq:551}. We will now reinterpret it in the setting that we are using it. In our case, we take $s=x \in \tilde S(\R)$ strongly regular (automatically elliptic). Bouaziz's $f_0$ is our $f$, his $f$ is our $wf$, and his $\tau_0$ is our $\tilde\tau$.

The radical of the symplectic form $B_{wf}$ equals $\mf{s}$ for any $w$ in the summation index. If we let $\mf{u}$ to be the nilpotent radical of a Borel $\C$-subalgebra that is stable under $s$, then $\mf{g}=\mf{u}^\tx{opp} \oplus \mf{s} \oplus \mf{u}$ and the regularity of $s$ implies $\mf{q}_\C = \mf{u}^\tx{opp} \oplus (1-s)\mf{s} \oplus \mf{u}$. As discussed in \S\ref{sub:bf}, the image of $\mf{u}$ in $\mf{g}/\mf{s}$ is a Lagrangian subspace for $B_{wf}$. Therefore we may take $l=\mf{u}$. As just mentioned above, the metaplectic character $\rho_l^{wf}$ depends on $wf$ only through its centralizer, which is $S$. Therefore $\rho_l^f=\rho_\mf{u}$ as reviewed in \S\ref{sub:mdc}. 

We claim that we may take $\mf{u}=w\mf{u}_fw^{-1}$ with $\mf{u}_f$ as in Definition \ref{dfn:pn_rp}. Indeed, $\mf{u}_f$ is stable under any element of $\tilde S(\R)$, in particular under $w^{-1}xw$, hence $w\mf{u}_fw^{-1}$ is stable under $x$.

With all of these adjustments in notation, \eqref{eq:551} becomes in our case
\[ \frac{(-1)^{q^{wf}_{w\mf{u}_fw^{-1}}}}{\det(1-w^{-1}xw|\mf{u}_f)}\rho_{w\mf{u}_fw^{-1}}(\dot x)\tx{tr}\,\tilde\tau(w^{-1}\dot xw), \]
where $\dot x \in \tilde S(\R)_f$ is any preimage of $x$.

According to Corollary \ref{cor:pn_qvuf} we have 
\[ (-1)^{q^{wf}_{w\mf{u}_fw^{-1}}} = (-1)^{q^f_\mf{u}}.\]
Furthermore
\[ \rho_{w\mf{u}_fw^{-1}}(\dot x)\tx{tr}\,\tilde\tau(w^{-1}\dot xw) = \rho_{\mf{u}_f}(w^{-1}\dot xw)\tx{tr}\,\tilde\tau(w^{-1}\dot xw) = \tilde\tau_\lhd(w^{-1}xw).  \]
Combining this with \eqref{eq:b2} we obtain the formula of Theorem \ref{thm:bouaziz}.

In the remainder of this section, we will obtain further information about some terms in this formula, which will be useful in the proof of our main theorem.

\begin{lem} 
  Let $s \in \tilde S$ and $\alpha \in R(S,G)$ be such that $\tx{Ad}^*(s)\alpha=\alpha$. Then $\tx{Ad}(s)$ preserves the root spaces $\mf{g}_\alpha$ and $\mf{g}_{-\alpha}$ and acts on them by scalars $c_{\alpha,s}$ and $c_{-\alpha,s}$ respectively, satisfying $c_{\alpha,s} \cdot c_{-\alpha,s}=1$.
\end{lem}
\begin{proof}
The conjugation action of $s$ preserves $S$ and acts on $\mf{s}$ and $\mf{s}^*$. From $\tx{Ad}^*(s)\alpha=\alpha$ we infer that $\tx{Ad}(s)$ fixes $H_\alpha$ and preserves the lines $\mf{g}_\alpha$ and $\mf{g}_{-\alpha}$. It must therefore act by scalars, and $H_\alpha=\tx{Ad}(s)H_\alpha =[\tx{Ad}(s)X_\alpha,\tx{Ad}(s)X_{-\alpha}]=c_{\alpha,s}c_{-\alpha,s}H_\alpha$ proves the claim.
\end{proof}

In other words, the root $\alpha : S \to \C^\times$ extends to a character $\tilde S_\alpha \to \C^\times$, where $S \subset \tilde S_\alpha \subset \tilde S$ is the stabilizer of $\alpha$ for the action of $\tilde S$ on $\tx{Lie}^*(S)$, and this extension is inverse to the extension of $-\alpha$. We will denote these extensions again by $\alpha$ and $-\alpha$, respectively. 

\begin{lem} \label{lem:pn_detroot}
Let $s \in \tilde S$ and let $\mc{O}$ be an orbit for the action of $\tx{Ad}^*(s)$ on $R(S,G)$. Let $\mf{g}_\mc{O}=\bigoplus_{\alpha \in \mc{O}}\mf{g}_\alpha$. Then
\begin{eqnarray*}
\det(s|\mf{g}_\mc{O})&=&(-1)^{|\mc{O}|-1}\alpha(s^{|\mc{O}|}).\\
\det(1-s|\mf{g}_\mc{O})&=&1-\alpha(s^{|\mc{O}|}).\\
\end{eqnarray*}
\end{lem}

Before we give the (very simple) proof we emphasize again that $\alpha$ denotes the \emph{extension} of the root $\alpha : S \to \mb{G}_m$ to the larger group $\tilde S_\alpha$ (discussed just before Lemma \ref{lem:pn_detroot}), and that the element $s^{|\mc{O}|} \in \tilde S_\alpha$ may not lie in $S$. This is related to the concept of ``roots of type R3'' in \cite{KS99}. The simplest example is the group $\tilde G = \tx{GL}_3 \rtimes \<\theta\>$, where $\theta$ is the pinned outer automorphism. Taking $s = 1 \rtimes \theta$ and $S \subset G$ the subgroup of diagonal matrices, the action $\tx{Ad}^*(s)$ on $R(S,G)=\{\pm\alpha,\pm\beta,\pm(\alpha+\beta)\}$ switches $\alpha,\beta$ and fixes $\gamma=\alpha+\beta$. If we let $\mc{O}=\{\gamma\}$, then $S_\gamma=S \rtimes \<\theta\>$, $s^{|\mc{O}|}=s \notin S$, and $\alpha : S_\gamma \to \mb{G}_m$ sends $s=1 \rtimes \theta$ to $-1$.

\begin{proof}
Choose $\alpha \in \mc{O}$ and $0 \neq X \in \mf{g}_\alpha$. Then $\{X,\tx{Ad}(s)X,\dots,\tx{Ad}(s^{n-1})X\}$ is a basis for $\mf{g}_\mc{O}$, where $n=|\mc{O}|$. In terms of this basis the matrix of the action of $\tx{Ad}(s)$ looks like
\[ \begin{bmatrix}
  0&1\\
  &0&1\\
  &&\ddots&\ddots\\
  &&&0&1\\
  c&\dots&&&0
\end{bmatrix}, 
\]
with $c=\alpha(s^n)$. 
\end{proof}

\begin{cor} \label{cor:pn_d4}
For $s \in \tilde S$,
\[ |\det(1-s|\mf{u}_f)| = |\det(1-s|\mf{g}/\mf{s})|^{1/2}. \]
\end{cor}
\begin{proof}
Since $\mf{u}_f$ is stable under $s$ it is the direct sum of $\mf{g}_\mc{O}$ for orbits $\mc{O}$ under the action of $\tx{Ad}^*(s)$ on $R(S,G)$, and Lemmas \ref{lem:pn_detroot} shows
\[ |\det(1-s|\mf{u}_f)| = \prod |1-\alpha(s^{|\mc{O}|})|, \]
the product running over the set of orbits of $\tx{Ad}^*(s)$ on $R^+_f$, and for each such orbit $\mc{O}$ we have chosen arbitrarily $\alpha \in \mc{O}$. Write $z_\mc{O}=\alpha(s^{|\mc{O}|})$. The ellipticity of $S$ implies that $|z_\mc{O}|=1$. Choosing a square root $z_\mc{O}^{1/2}$ and extracting it from each $(1-z_\mc{O})$ we obtain
\[ |\det(1-s|\mf{u}_f)| = \prod |z_\mc{O}^{1/2}-z_\mc{O}^{-1/2}|, \]
where $z_\mc{O}^{-1/2}$ denotes the inverse of the chosen square root. 

If we apply the same argument to $\mf{u}_{-f}$, we obtain
\[ |\det(1-s|\mf{u}_{-f})| = \prod |z_\mc{-O}^{1/2}-z_\mc{-O}^{-1/2}|, \]
where $\mc{O}$ again runs over the set of orbits of $\tx{Ad}^*(s)$ on $R^+_f$, but we are taking the negative $-\mc{O}$ of each such orbit. It is clear that $z_\mc{-O}=z_\mc{O}^{-1}$, and we are free to choose the square root of that to equal $z_\mc{O}^{-1/2}$. This shows
\[ |\det(1-s|\mf{g}/\mf{s})| =|\det(1-s|\mf{u}_f)| \cdot |\det(1-s|\mf{u}_{-f})| = |\det(1-s|\mf{u}_f)|^2.\qedhere \]
\end{proof}

\begin{lem} \label{lem:pn_indepu}
Let $\dot s \in \tilde S(\R)_f$ and let $\mf{u} \subset \mf{g}$ be the unipotent radical of a Borel $\C$-subalgebra stabilized by $s$. The number
\[ \frac{(-1)^{q_\mf{u}^f}\rho_{f,\mf{u}}(\dot s)}{\det(1-s|\mf{u})} \]
is independent of $\mf{u}$.
\end{lem}
\begin{proof}
Another $\mf{u}'$ equals $w\mf{u}$ for a unique $w \in W_G(S)$ that commutes with $s$ in the group $\tx{Aut}_\C(S)$. More precisely, let $n \in N_G(S)$ represent $w$. The commutator $n^{-1}sns^{-1}$ lies in $G = \tilde G^0$, and normalizes $S$, since each of the factors does, hence lies in $N_G(S)$. Now $s$ normalizes $n\mf{u}n^{-1}$ if and only if $n^{-1}sn$ normalizes $\mf{u}$, and since $s$ already normalizes it, this is equivalent to $n^{-1}sns^{-1}$ normalizing it. The latter belonging to $N_G(S)$, this is equivalent to it lying in $S$. Thus the condition of $s$ stabilizing $w\mf{u}$ is equivalent to $n^{-1}sns^{-1} \in N_G(S)$ lying in $S$ (this condition is independent of the choice of $n$), and that in turn is equivalent to the triviality of the image of $n^{-1}sns^{-1}$ in $\tx{Aut}_\C(S)$.

From Corollary \ref{cor:pn_qvuf} we have
\[ (-1)^{q_{w\mf{u}}^f} = (-1)^{q_{\mf{u}}^f} \cdot \epsilon(w). \]

From \eqref{eq:duflo2} we have
\[ \rho_{f,w\mf{u}}(\dot s)=\rho_{f,\mf{u}}(\dot s) \cdot \det(s|w\mf{u}/(\mf{u} \cap w\mf{u})). \]

We see that changing from $\mf{u}$ to $w\mf{u}$ multiplies the expression by
\[ \frac{\epsilon(w)\det(s|w\mf{u}/(\mf{u} \cap w\mf{u}))\det(1-s|\mf{u})}{\det(1-s|w\mf{u})}. \]
Writing
\begin{eqnarray*}
  \mf{u}&=&(\mf{u} \cap w\mf{u}) \oplus (\mf{u} \cap w\mf{\bar u})\\
  w\mf{u}&=&(\mf{u} \cap w\mf{u}) \oplus (\mf{\bar u} \cap w\mf{u})  
\end{eqnarray*}
the above scalar becomes
\[ \frac{\epsilon(w)\det(s|\mf{\bar u} \cap w\mf{u})\det(1-s|\mf{u} \cap w\mf{\bar u})}{\det(1-s|\mf{\bar u} \cap w\mf{u})}. \]
Using Lemma \ref{lem:pn_detroot} this is seen to equal
\[ \epsilon(w) \cdot \prod_\mc{O} (-1)^{n-1}\alpha(s^n)^{-1} \frac{1-\alpha(s^n)}{1-\alpha(s^n)^{-1}},\]
where $\mc{O}$ runs over the set of $s$-orbits of roots in the set $R_\mf{u}^+ \cap wR_\mf{u}^-$, and $n=|\mc{O}|$. Note that, since $s$ preserves $\mf{u}$ and commutes with $w$, it acts on this set. Each factor in the product equals $(-1)^n$, so the whole product equals $(-1)^{|R_\mf{u}^+ \cap wR_\mf{u}^-|}=\ell(w)$.
\end{proof}

\begin{rem} \label{rem:typo}
  Just below the statement of \cite[(5.5.1)]{Bouaziz87} it is claimed that the formula does not depend on the choice of Lagrangian subspace $l$, and for this the reference \cite{DHV84} is given. Lemma \ref{lem:pn_indepu} above is a separate verification of this claim. 

  Note however that the formula \cite[(5.5.1)]{Bouaziz87} contains a typo: it is missing the tilde over $l \cap \mf{q}_\C$ in the denominator. Indeed, Bouaziz defines $l$ to be any subspace of $\mf{g}$ whose projection modulo the radical of $B_f$ is a Lagrangian. Two options for $l$ are $\mf{u}$ and $\mf{s} \oplus \mf{u}$. But switching between these options does change the outcome, it multiplies it by $\det(1-s|\mf{s}/\mf{s}^s)$, which depends only on the coset of $s$ in $\tilde G$, but is generally not equal to $1$.

  Another indication of this typo is the last displayed formula before Section 6.2 of \cite{Bouaziz87}, where the denominator is $\det(1-\tx{Ad}(x))_\mf{n}$, and the $\mf{n}$ there is what $\mf{u}$ is here.
\end{rem}

The following lemma compares the terms in Bouaziz's formula with the classical Harish-Chandra formula in the connected case.

\begin{lem}
Assume that $\tilde G=G$ is connected. Then
\[ \frac{(-1)^{q_{\mf{u}_f}^{wf}}}{\det(1-s|\mf{u}_f)}\rho_{f,\mf{u}}(\dot s) = \frac{(-1)^{q(G)}\epsilon(w)}{\prod\limits_{\<H_\alpha,d\chi\>>0}(\alpha(\dot s)^{1/2}-\alpha(\dot s)^{-1/2}) } \] 
\end{lem}
\begin{proof}
According to Corollary \ref{cor:pn_qvuf} we have $(-1)^{q_{\mf{u}_f}^{wf}} = q(G)\epsilon(w)$. Lemma \ref{lem:pn_negative} shows that $\mf{u}_f$ is the \emph{opposite} of the unipotent radical determined by $\<H_\alpha,d\chi\>>0$.
\end{proof}

\subsection{Whittaker vectors and Duflo's construction} \label{sub:signchar}

In this subsection we assume that $G$ is quasi-split and maintain the assumption that $G(\R)$ has discrete series. Fix a non-trivial unitary character $\Lambda : \R \to \C^\times$ and an $\R$-pinning $\mc{P}$ of $G$. These lead to a Whittaker datum $\mf{w}$ for $G$. Let $A \subset \tilde G(\R)$ be the stabilizer of $\mc{P}$. Then $A \cap G(\R) \subset Z_G(\R)$. Thus, if $G$ is adjoint, then $\tilde G(\R) = G(\R) \rtimes A$. 

If an irreducible representation $\pi$ of $G(\R)$ extends to $\tilde G(\R)$, then all extensions differ by characters of $\tilde G(\R)/G(\R)$. In this subsection we will study two different such extensions in the case when $\pi$ is an irreducible $\mf{w}$-generic discrete series representation whose isomorphism class is $A$-stable, and show that they differ by a specific character. 

We begin by defining the relevant character. This character will be inflated from the adjoint quotient, so for the definition we assume without loss of generality that $G$ is adjoint, hence $\tilde G(\R) = G(\R) \rtimes A$. For $a \in A$ the fixed subgroup $G^a$ for the action of $a$ on $G$ is connected \cite[Lemma 3.1]{Ree10}. Define
\[ \epsilon_{\tilde G}(a) = (-1)^{q(G)-q(G^a)}.\]

\begin{lem}
  The inflation of $\epsilon_{\tilde G}$ to $\tilde G(\R)$ does not depend on $\mc{P}$.
\end{lem}
\begin{proof}
  Any other pinning is of the form $g\mc{P}g^{-1}$ for some $g \in G(\R)$. Using that pinning replaces $A$ by $gAg^{-1}$. Now $a \in A$ and $gag^{-1} \in gAg^{-1}$ map to the same element of $\tilde G(\R)/G(\R)$, and at the same time $q(G^a)=q(G^{gag^{-1}})$ since conjugation by $g$ gives an automorphism $G^a \to G^{gag^{-1}}$ of $\R$-groups.
\end{proof}

\begin{lem}
  The map $\epsilon_{\tilde G}$ is a group homomorphism.
\end{lem}
\begin{proof}
  The assumption that $G$ has discrete series implies that $G$ has an anisotropic maximal torus, and Proposition \ref{pro:aa} implies that $(S,C) \mapsto (S^A,C^A)$ is a bijection from the set of $A$-stable pairs consisting of an anisotropic maximal $\R$-torus $S \subset G$ and a Borel $\C$-subgroup $S \subset C \subset G$ to the set of analogous pairs for $G^A$. A root of $S$ is automatically imaginary, and is non-compact if and only if its restriction to $S^A$ is non-compact, according to \cite[Lemma 3.6.3]{KalLLCD}. Now $q(G)$ is the number of non-compact roots in $R(S,C)$, while $q(G^a)$ is the number of non-compact roots in $R(S^a,C^a)$. Restriction from $S$ to $S^a$ provides a bijection from the set of $a$-orbits in $R(S,C)$ that are not of type $R_3^a$ (in the sense of \cite[\S1.3]{KS99}, where we have included the relevant automorphism $a$ in the notation) to the set $R(S^a,C^a)$.
  
  We claim that $C$ can be chosen so that an $a$-orbit in $R(S,C)$ consisting of non-compact elements is automatically not of type $R_3^a$. Indeed, we can choose $C^A$ so that all simple roots of $S^A$ are non-compact, which implies that for the corresponding $C$ all simple roots of $S$ are non-compact. Given $\alpha,\beta \in R(S,C)$, $\alpha+\beta$ is compact if and only if $\alpha,\beta$ are either both compact or both non-compact. By our choice of $C$ we see that the compact roots are precisely those of even height. But any root of type $R_3^a$ has even height (it is a sum of two roots that are swapped by a power of $a$), and the claim is proven.

  The claim shows that $q(G^a)$ is the number of $a$-orbits if imaginary non-compact elements of $R(S,C)$. Write $X$ for the set of imaginary non-compact elements of $R(S,C)$, $n=\#X$, $\sigma_a$ for the permutation of $X$ induced by the action of $a$, and $|\sigma_a|$ for the number of orbits of $\sigma_a$ on $X$. Then 
  \[ q(G)-q(G^a) = n-|\sigma_a|, \]
  which implies
  \[ (-1)^{q(G)-q(G^a)}=\tx{sgn}(\sigma_a|X); \]
  indeed, if we decompose $X$ into orbits under $\sigma_a$, then each orbit $O$ contributes $|O|$ to $q(G)$ and $1$ to $q(G^a)$, hence the factor $(-1)^{|O|-1}$ to $(-1)^{q(G)-q(G^a)}$, but this factor is also the sign of $\sigma_a$ acting on $O$. 
\end{proof}

We now come to the main result of this subsection. For this, we drop the assumption that $G$ is adjoint. We will however still maintain the assumption that $\tilde G(\R) = G(\R) \rtimes A$. Let $\pi$ be an irreducible discrete series representation of $G(\R)$ that is $\mf{w}$-generic and whose isomorphism class is $A$-stable. Since the Whittaker datum $\mf{w}$ is $A$-admissible, according to Proposition \ref{pro:adgen} there exists an $A$-stable representative $(S,\tau)$ of its Harish-Chandra parameter. 

There are two ways to extend $\pi$ to a representation of $\tilde G(\R)$. One is to extend $\tau$ to a character of $S(\R) \rtimes A$ by letting it be trivial on $A$, and then apply Duflo's construction (\S\ref{sub:duflo}) to obtain an irreducible representation $\tilde\pi^D$ of $\tilde G(\R)$. Lemma \ref{lem:duflo1} shows that the restriction of $\tilde\pi^D$ to $G(\R)$ is $\pi$. The other way is to let, for each $a \in A$, $\tilde\pi^W(a) : \pi \circ a^{-1} \to \pi$ be the unique isomorphism whose action on the 1-dimensional space of Whittaker functions on $V_\pi^\infty$ is trivial. This makes sense, because a Whittaker functional with respect to $\pi$ is also one with respect to $\pi\circ a^{-1}$. Defining $\tilde\pi^W(g \rtimes a) = \pi(g)\circ\tilde\pi^W(a)$, we obtain the second extension $\tilde\pi^W$ of $\pi$ to $\tilde G(\R)$.

The following is the main result of this subsection, whose proof will be reduced here to the core case of a semi-simple simply connected group, and that core case will be handled in Appendix \ref{app:a}.

\begin{thm} \label{thm:rtci-sign}
  \[ \pi^W(A) = (-1)^{q(G)-q(G^A)}\pi^D(A). \]
\end{thm}
\begin{proof}
  Let $G_\tx{sc}$ be the simply connected cover of the derived subgroup of $G$. We claim that the theorem holds for $G$ if and only if it holds for $G_\tx{sc}$. To see this, we will use \cite[Lemma 2.7.5, Theorem 2.7.7]{DPR}.

    Let $G(\R)^\natural \subset G(\R)$ be the image of $G_\tx{sc}(\R)$ and let $Z_G$ be the center of $G$. Then $G(\R)^\natural \cdot Z_G(\R)$ is of finite index in $G(\R)$, hence the restriction of $\pi$ to $G(\R)^\natural \cdot Z_G(\R)$ is a finite direct sum of representations of the form $\pi^\natural\otimes\omega$, where $\omega$ is the central character of $\pi$ and $\pi^\natural$ is an irreducible discrete series representation of $G(\R)^\natural$. Furthermore, 
    \[ \pi=\tx{Ind}_{G(\R)^\natural \cdot Z_G(\R)}^{G(\R)}(\pi^\natural\otimes\omega), \]
    for any choice of $\pi^\natural$ in the finite direct sum.

    The representations $\pi^\natural$ that occur in this direct sum form a single orbit under the conjugation action of $G(\R)$ and therefore belong to the same $L$-packet (when considered as representations of $G_\tx{sc}(\R)$ via pull-back). Since $\pi$ is $\mf{w}$-generic for the Whittaker datum $\mf{w}$ arising from the fixed pinning and character $\R \to \C^\times$, one of the $\pi^\natural$ must also be $\mf{w}_\tx{sc}$-generic, where $\mf{w}_\tx{sc}$ is the Whittaker datum for $G_\tx{sc}$ obtained from the same pinning and character. According to \cite[Lemma 3.3.6]{DPR} there is a unique $\pi^\natural$ in its $L$-packet that is generic for any fixed Whittaker datum. Thus, $\mf{w}_\tx{sc}$-genericity singles out precisely one $\pi^\natural$ among the constituents of the restriction of $\pi$ to $G(\R)^\natural$.
    
    Since $A$ preserves the isomorphism class of $\pi$ and also the Whittaker datum $\mf{w}_\tx{sc}$, it preserves the unique $\mf{w}_\tx{sc}$-generic $\pi^\natural$. We thus have $\pi^{\natural,*}(A) : \pi^\natural \to \pi^\natural \circ A^{-1}$ for $*=W,D$. We claim that 
    \[ \pi^*(A) = \tx{Ind}_{G(\R)^\natural \cdot Z_G(\R)}^{G(\R)} \pi^{\natural,*}(A),\qquad *=W,D. \]
    To see this we note that the case $*=W$ is governed by $\tx{Hom}_{U(\R)}(V_\pi^\infty,\C_\mf{w})$, where $U$ is the unipotent radical of the $\R$-Borel subgroup that is part of the Whittaker datum, while the case $*=D$ is governed by $H_{q(G)}(\mf{u},V_\pi^\infty)_{if+\rho}$, where $\mf{u}$ is the Lie algebra of the unipotent radical of the complex Borel subgroup determined by the infinitesimal character $f$. Both of these groups depend only on the restriction of $\pi$ to $G(\R)^\natural$, and their 1-dimensionality and vanishing properties show that they are unchanged if we replace $\pi$ with the unique $\mf{w}$-generic $\pi^\natural$. This proves the claim.

    The claim implies that $\pi^W(A)=\pi^D(A)$ if and only if $\pi^{\natural,W}(A)=\pi^{\natural,D}(A)$. On the other hand, it is clear from the definition of $q(G)$ and $q(G^A)$ that these numbers do not change if we replace $G$ by $G_\tx{sc}$.

    We have thus proved the claim that the theorem holds for $G$ if and only if it holds for $G_\tx{sc}$. The proof in the case that $G$ is semi-simple and simply connected is given in the appendix, see Proposition \ref{pro:signs}. 
\end{proof}

\section{Square-integrable $L$-packets of real disconnected groups} \label{sec:sqlp}

In this section we will construct the refined local Langlands correspondence for discrete series representations of rigid inner forms of quasi-split disconnected groups over $\R$. Thus, unlike in the previous sections, the disconnected groups considered here will be less general. The precise conditions are spelled out in \S\ref{sub:sqpac}.

\subsection{The local Langlands correspondence for disconnected tori} \label{sub:luo}

Here we shall briefly recall the local Langlands correspondence for disconnected tori, which was constructed in \cite[\S8]{KalLLCD} and reinterpreted in a more natural way in the thesis \cite{YiLuoThesis} of Yi Luo.

Let $T$ be a torus over a local field $F$ of characteristic zero; in this paper we are particularly interested in $F=\R$, but this special case will not bring any simplification. Let $A$ be a finite group acting on $T$. Let $\varphi : W_F \to {^LT}$ be an $L$-parameter for $T$ and $z \in Z^1(\mc{E}^\tx{rig},T)$, where $\mc{E}^\tx{rig}$ is the $F$-gerbe constructed in \cite{KalRI}. We write $[\varphi] : T(F) \to \C^\times$ for the character associated to the $\hat T$-conjugacy class $[\varphi]$ of $\varphi$ by the usual local Langlands correspondence, $[z] : \pi_0([\hat{\bar T}]^+) \to \C^\times$ for the character associated to the class $[z] \in H^1(\mc{E},T)$ by the generalized Tate-Nakayama duality \cite[\S4]{KalRI}, $\tx{Irr}([T \rtimes A]_z,[\varphi])$ for the set of irreducible representations of the inner form $[T \rtimes A]_z(F)$ whose restriction to $T_z(F)=T(F)$ contains the character $[\varphi]$, and $\tx{Irr}(\pi_0([\hat{\bar T} \rtimes A^{[z]}]_\varphi^+),[z])$ for the set of those irreducible representations of the preimage $[\hat{\bar T} \rtimes A^{[z]}]_\varphi^+$ of the centralizer of $\varphi$ in $[\hat T \rtimes A^{[z]}]$ that factor through the component group and whose restriction to $\pi_0([\hat{\bar T}]^+)$ contains $[z]$. It is shown in \cite[\S8]{KalLLCD} that there is a a natural bijection
\[ \tx{Irr}([T \rtimes A]_z,[\varphi]) \leftrightarrow \tx{Irr}(\pi_0([\hat{\bar T} \rtimes A^{[z]}]_\varphi^+),[z]), \]
characterized uniquely by character identities. An elementary argument reduces to the construction of a bijection
\begin{equation} \label{eq:llcdtor}
\tx{Irr}([T \rtimes A^{[z],[\varphi]}]_z,[\varphi]) \leftrightarrow \tx{Irr}(\pi_0([\hat{\bar T} \rtimes A^{[z],[\varphi]}]_\varphi^+),[z]),
\end{equation}
where $A^{[\varphi],[z]}$ is the subgroup of $A$ consisting of those elements that preserve the $\hat T$-conjugacy class $[\varphi]$ and the cohomology class $[z]$. This bijection is then obtained by constructing a natural isomorphism between the two pushouts
\[ \xymatrix{
  1 \ar[r]& T(F) \ar[r]\ar[d]^{[\varphi]}&[T \rtimes A^{[z],[\varphi]}]_z(F)\ar[r]& A^{[z],[\varphi]}\ar[r]& 1\\
&\C^\times
}\]
and 
\[ \xymatrix{
  1 \ar[r]& \pi_0([\hat{\bar T}]^+) \ar[r]\ar[d]^{[z]}&[\hat{\bar T} \rtimes A^{[z],[\varphi]}]_\varphi^+ \ar[r]& A^{[z],[\varphi]}\ar[r]& 1\\
&\C^\times
}\]
The construction of this isomorphism given in \cite[\S8]{KalLLCD} relies on a number of auxiliary choices, which are then shown to not influence the result. 

A cleaner interpretation of the bijection \eqref{eq:llcdtor} is given in section 5.3 of the thesis of Yi Luo, especially equation (193). We recall his result as follows.

\begin{pro}[Yi Luo] \label{pro:luo}
  Let $\pi \in \tx{Irr}([T \rtimes A^{[z],[\varphi]}]_z,[\varphi]) $ and $\rho \in \tx{Irr}(\pi_0([\hat{\bar T} \rtimes A^{[z],[\varphi]}]_\varphi^+),[z])$ correspond under the bijection \eqref{eq:llcdtor}. Then, after replacing $\pi$ and $\rho$ by isomorphic representations, both of them act on the same finite-dimensional complex vector space $V$ and 
  \[ \rho(\tilde s) \circ \pi(\tilde t) = \<(\varphi_0^{-1},s),(z^{-1},t)\>_\tx{TN}^{-1} \]
  for any $\tilde s = s \rtimes a^{-1} \in [\hat{\bar T} \rtimes A^{[z],[\varphi]}]_\varphi^+$ and $\tilde t = t \rtimes a \in [T \rtimes A^{[z],[\varphi]}]_z(F)$, where $\varphi(w) = \varphi_0(w) \rtimes w \in \hat T \rtimes W_F$.
\end{pro}

\subsection{Square-integrable $L$-packets} \label{sub:sqpac}

Let $G$ be a quasi-split connected reductive $\R$-group, and let $A$ be a finite group of automorphisms of $G$ that preserves a fixed pinning $\mc{P}$. Let $\tilde G = G \rtimes A$. Let $z \in Z^1(u \to \mc{E}^\tx{rig},Z(G)^A \to G)$ and let $\tilde G_z$ be the $z$-twist of $\tilde G$. Thus, in the language of \cite{KalLLCD}, $\tilde G$ is a quasi-split disconnected reductive group and $\tilde G_z$ is a rigid inner form thereof. The group we will be working with is $\tilde G_z$.

Fix an non-trivial unitary character $\Lambda : \R \to \C^\times$. Then $\mc{P}$ and $\Lambda$ lead to an $A$-admissible Whittaker datum $\mf{w}$ for the quasi-split group $G$. Write $\hat G$ for the common dual group of $G$ and $G_z$, and $^LG=\hat G \rtimes \Gamma$.

Let $\varphi : W_\R \to {^LG}$ be a discrete $L$-parameter. Langlands has associated an $L$-packet $\Pi_\varphi(G_z)$ of essentially discrete series representations of $G_z(\R)$. We define $\Pi_\varphi(\tilde G_z)$ to be the set of irreducible representations of $\tilde G_z(\R)$ whose restriction to $G_z(\R)$ intersects $\Pi_\varphi(G_z)$. 

The group $A$ acts on $H^1(u \to \mc{E}^\tx{rig},Z(G)^A \to G)$ and we let $A^{[z]}$ be the stabilizer of the class of $z$. Then $A^{[z]}$ also stabilizes the character of $\pi_0(Z(\hat{\bar G})^+)$ determined by $[z]$. Let $S_\varphi=\tx{Cent}(\varphi,\hat G)$ and $\tilde S_\varphi^{[z]}=\tx{Cent}(\varphi,\hat G \rtimes A^{[z]})$. 

The first main result of this paper is the construction of an injection
\begin{equation} \label{eq:is}
\Pi_\varphi(\tilde G_z) \to \tx{Irr}(\pi_0(\tilde S_\varphi^{+,[z]})),
\end{equation}
whose image consists of representations on which $\pi_0(Z(\hat{\bar G})^+)$ acts via $[z]$. This injection will depend on $\mf{w}$. Its construction will be given presently.

To that end, we recall the analogous injection
\begin{equation} \label{eq:is1}
\Pi_\varphi(G_z) \to \tx{Irr}(\pi_0(S_\varphi^+))  
\end{equation}
discussed in \cite[\S4]{DPR}. To $\varphi$ there is canonically associated a pair $(S,\tau,J_z)$ consisting of an $\R$-torus $S$, a stable class $J_z$ of embeddings $S \to G_z$ whose image is an elliptic maximal torus, and a genuine character $\tau$ of $S(\R)_G$ whose differential $d\tau \in i \cdot \tx{Lie}^*(S)(\R)$ is regular; we note here that the stable class $J_z$ endows $X^*(S)$ with a subset $X$ via which one obtains the double cover $S(\R)_G$, as well as with a set $R(S,G) \subset X^*(S)$ which gives meaning to the word ``regular''. Set $f=d\tau/i$. For each $j \in J_z$ we can transport $(f,\tau)$ via $j$ and obtain a representation $\pi_{f,\tau}$ of $G(\R)$, which we denote by $\pi_j$. Then
\[ \Pi_\varphi(G_z) = \{\pi_j|j \in J_z/G_z(\R)\}, \]
where $J_z/G_z(\R)$ denotes the set of $G_z(\R)$-conjugacy classes of elements in $J_z$. There is a unique $\mf{w}$-generic member $\pi_\mf{w} \in \Pi_\varphi(G_1)$, and it corresponds to an embedding $j_\mf{w} \in J_1$; we have written here $G=G_1$ to emphasize that we are speaking of the quasi-split form. We have the injective map (see \cite[\S2.3]{DPR} for the definition of $\tx{inv}(j_\mf{w},j)$)
\[ J_z/G_z(\R) \to H^1(u \to \mc{E}^\tx{rig},Z(G) \to S),\qquad j \mapsto \tx{inv}(j_\mf{w},j). \]
On the other hand, there is a tautological identification $S_\varphi=\hat S^\Gamma$, and hence $\pi_0(S_\varphi^+)=\pi_0([\hat{\bar S}]^+)$. The duality pairing of \cite[Corollary 5.4]{KalRI} provides the desired injection \eqref{eq:is1}.

We now turn to the construction of \eqref{eq:is}. We have the exact sequence
\[ 1 \to G_z(\R) \to \tilde G_z(\R) \to A^{[z]} \to 1. \]
The irreducible representations of $G_z(\R)$ that are contained in the restriction of an irreducible representation $\tilde\pi$ of $\tilde G_z(\R)$ form a single orbit under the action of $A^{[z]}$. The disjointness of discrete series $L$-packets of $G_z(\R)$ implies that if $\pi \in \Pi_\varphi(G_z)$ and $a\pi \in \Pi_\varphi(G_z)$ for some $a \in A^{[z]}$, then $a$ preserves the $L$-packet $\Pi_\varphi(G_z)$, hence also its parameter (seen as a $\hat G$-conjugacy class). Let $A^{[\varphi]}$ be the stabilizer of the $\hat G$-conjugacy class of $\varphi$, and let $A^{[z],[\varphi]}=A^{[z]} \cap A^{[\varphi]}$. From this discussion we obtain the disjoint union decomposition
\begin{equation} \label{eq:subpack}
\Pi_\varphi(\tilde G_z) = \coprod_{\pi \in \Pi_\varphi(G_z)/A^{[z],[\varphi]}} \Pi_\varphi(\tilde G_z,\pi),
\end{equation}
where $\Pi_\varphi(\tilde G_z,\pi)$ is the subset of $\Pi_\varphi(\tilde G_z)$ consisting of those representations $\tilde\pi$ of $\tilde G_z(\R)$ whose restriction to $G_z(\R)$ contains $\pi$. 

We apply the analogous discussion to the exact sequence
\begin{equation} \label{eq:subpack1}
1 \to \pi_0(S_\varphi^+) \to \pi_0(\tilde S_\varphi^{+,[z]}) \to A^{[z],[\varphi]} \to 1,
\end{equation}
and obtain the disjoint union decomposition
\[ \tx{Irr}(\pi_0(\tilde S_\varphi^{+,[z]})) = \coprod_{\eta \in \tx{Irr}(\pi_0(S_\varphi^+))/A^{[z],[\varphi]}} \tx{Irr}(\pi_0(\tilde S_\varphi^{+,[z]}),\eta), \]
where $\tx{Irr}(\pi_0(\tilde S_\varphi^{+,[z]}),\eta)$ is the subset of those irreducible representations of the group $\pi_0(\tilde S_\varphi^{+,[z]})$ whose restriction to $\pi_0(S_\varphi^+)$ contains the character $\eta$.

Given the injection \eqref{eq:is1} and its equivariance properties with respect to automorphisms (cf. \cite[Conjecture A.1]{KalLLCD}, it would be enough to construct a bijection
\begin{equation} \label{eq:is2}
\Pi_\varphi(\tilde G_z,\pi) \to \tx{Irr}(\pi_0(\tilde S_\varphi^{+,[z]}),\eta),
\end{equation}
for each $\pi \in \Pi_\varphi(G_z)$ and $\eta \in \tx{Irr}(\pi_0(S_\varphi^+))$ corresponding under \eqref{eq:is1}. 

\begin{lem} \label{lem:duflo2}
Let $\pi \in \Pi_\varphi(G_z)$ and let $(f,\tau)$ be a pair in the $G_z(\R)$-conjugacy class associated to $\pi$. Let $\tilde S=Z_{\tilde G_z}(f)$ and $S=\tilde S \cap G_z$. Duflo's construction produces a bijection $\tilde\tau \mapsto \pi_{f,\tilde\tau}$ between the set of irreducible $\tilde\tau \in X(f)$ such that $\tilde\tau|_{S(\R)_G}$ contains $\tau$, and the set $\Pi_\varphi(\tilde G_z,\pi)$.
\end{lem}
\begin{proof}
Lemma \ref{lem:duflo1}(3,4) asserts that the members of $\Pi_\varphi(\tilde G_z,\pi)$ are obtained from Duflo's construction and thus correspond to certain $\tilde G_z(\R)$-orbits of pairs $(f',\tilde\tau)$, while Lemma \ref{lem:duflo1}(2) applied with $\tilde G_1=G_z$ asserts that a pair $(f',\tilde\tau)$ produces a member of $\Pi_\varphi(\tilde G_z,\pi)$ if and only if the $\tilde G_z(\R)$-orbit of $(f',\tilde\tau|_{S(\R)_G})$ meets the $G_z(\R)$-orbit of $(f,\tau)$. Therefore, up to $\tilde G_z(\R)$-conjugacy we have $f'=f$. This condition determines the pair $(f',\tilde\tau)$ in its $\tilde G_z(\R)$-conjugacy class up to conjugation by the centralizer of $f$ in $\tilde G_z(\R)$, which is $\tilde S(\R)$. But $\tilde S(\R)$ acts trivially on $\tilde\tau$, so the pair $(f',\tilde\tau)$ is uniquely determined. The restriction of $\tilde\tau$ to $S(\R)_G$ consists of characters whose differential equals to $f$, and one of which is $G_z(\R)$-conjugate to $\tau$. But the differential of $\tau$ is also $f$, and the centralizer of $f$ in $G_z(\R)$ is $S(\R)$, from which we conclude that $\tau$ is contained in the restriction of $\tilde\tau$ to $S(\R)_G$.
\end{proof}

Let us write $X(f,\tau)$ for the subset of $X(f)$ consisting of those representations whose restriction to $S(\R)_G$ contains $\tau$. Lemma \ref{lem:duflo2} provides a bijection $X^\tx{irr}(f,\tau) \to \Pi_\varphi(\tilde G_z,\pi)$, and therefore the construction of \eqref{eq:is2} reduces to constructing a bijection
\begin{equation} \label{eq:is3}
  X^\tx{irr}(f,\tau) \to \tx{Irr}(\pi_0(\tilde S_\varphi^{+,[z]}),\eta).
\end{equation}

In the following we will construct this bijection by identifying $\tilde S_\varphi$ with the centralizer of an $L$-parameter $\varphi_\mf{w}$ for a disconnected torus $\tilde S_\mf{w}$, and identifying $\tilde S$ with a rigid inner form of $\tilde S_\mf{w}$ in such a way that $\tau$ is the character with parameter $\varphi_\mf{w}$ (up to twisting away the genuine behavior, which is a step that is conceptually unnecessary but dictated by the current state of the literature). The desired bijection will then be a consequence of the local Langlands correspondence for disconnected tori established in \cite[\S8]{KalLLCD} and reviewed in \S\ref{sub:luo}.

Let $\pi_\mf{w}$ be the unique member of $\Pi_\varphi(G_1)$ that is $\mf{w}$-generic. According to Proposition \ref{pro:adgen} there exists an $A$-stable representative $(S_\mf{w},\tau_\mf{w})$ of the Harish-Chandra parameter of $\pi_\mf{w}$. The  $\C$-Borel subgroup $C_\mf{w}$ determined by $R_{\tau_\mf{w}}^+$ (notation from \S\ref{sub:genchar}) is then $A$-stable.

\begin{lem} \label{lem:semi2}
Fix a $\Gamma$-stable and $A$-stable pinning $(\hat T,\hat B,\{\hat X_\alpha\})$ in $\hat G$. Let $\hat\jmath : \hat S_\mf{w} \to \hat T$ be the unique admissible isomorphism that matches the Weyl chamber in $X_*(\hat S_\mf{w})$ determined by $C_\mf{w}$ to the Weyl chamber in $X_*(\hat T)$ determined by $\hat B$. Let $^L\jmath : \hat S_\mf{w} \rtimes W_\R \to {^LG}$ be the $L$-embedding extending $\hat\jmath$ and obtained from $C_\mf{w}$-based $\chi$-data. Then $^L\jmath$ is equivariant for the action of $A$ on both sides.
\end{lem}
Recall here that $C_\mf{w}$-based $\chi$-data means that for any absolute root $\alpha$ that is $C_\mf{w}$-positive, $\chi_\alpha(z):=\tx{sgn}_\C(z):=z/|z|$.
\begin{proof}
It is enough to prove equivariance for the restrictions of $^L\jmath$ to $\hat S$ and to $W_\R$. The restriction to $\hat S_\mf{w}$ is the admissible isomorphism $\hat\jmath$. For any $a \in A$, $a\circ\hat\jmath\circ a^{-1}$ is also an admissible isomorphism $\hat S_\mf{w} \to \hat T$. Since the actions of $a$ preserve the distinguished Weyl chambers in $X_*(\hat S_\mf{w})$ and $X_*(\hat T)$, and since matching these chambers uniquely determines $\hat\jmath$, we see that $a\circ\hat\jmath\circ a^{-1}=\hat\jmath$.

The value of $^L\jmath$ on elements of the Weil group is given explicitly in \cite[\S10]{SheTE1}. There the Weil group is presented as the twisted product $\C^\times \boxtimes \Gamma_{\C/\R}$. For $z \in \C^\times$ we have $^L\jmath(z)=(z/\bar z)^\rho$, where $\rho$ is half sum of the $\hat B$-positive coroots of $\hat T$. Since $A$ preserves $\hat B$, it preserves $\rho$, and we see that the restriction of $^L\jmath$ to $\C^\times \subset W_\R$ is $A$-equivariant. For $\sigma \in \Gamma_{\C/\R}$ the value of $^L\jmath(\sigma)$ is $n(\omega_\sigma)\rtimes\sigma$, where $\omega_\sigma$ is the element of the Weyl group such that the action of $\omega_\sigma \rtimes \sigma$ on $\hat T$ is the transport via $\hat\jmath$ of the action of $\sigma$ on $\hat S_\mf{w}$, and $n(\omega_\sigma)$ is the Tits lift of $\omega_\sigma$ relative to the chosen pinning. Since the actions of $\Gamma$ and $A$ on $\hat S_\mf{w}$ commute, and we have already established that $\hat\jmath$ is $A$-equivariant, we have $a(\omega_\sigma \rtimes \sigma)=\omega_\sigma \rtimes \sigma$. Since the actions of $A$ and $\Gamma$ on $\hat T$ commute, we conclude from this $a(\omega_\sigma)=\omega_\sigma$. Since the pinning is $A$-stable we have $a(n(\omega_\sigma))=n(a(\omega_\sigma))$, which we just argued equals $n(\omega_\sigma)$. Since the actions of $A$ and $\Gamma$ on $\hat G$ commute, we finally conclude
\[ a({^L\jmath}(\sigma)) = a(n(\omega_\sigma) \rtimes \sigma)=n(\omega_\sigma) \rtimes\sigma = {^L\jmath}(\sigma). \qedhere\] 
\end{proof}

We now consider $\tilde S_\mf{w} = S_\mf{w} \rtimes A$. It is a disconnected torus in the sense of \cite{KalLLCD}. Lemma \ref{lem:semi2} shows that $^L\jmath \rtimes \tx{id}_A$ provides an inclusion
\[ \hat S_\mf{w} \rtimes (W_\R \times A) \to \hat G \rtimes (W_\R \times A).\]
\begin{lem} \label{lem:is}
This inclusion induces an isomorphism
\[ \tx{Cent}(\varphi_\mf{w},\hat S_\mf{w} \rtimes A^{[\varphi]}) \to \tilde S_\varphi, \]
where $\varphi_\mf{w} : W_\R \to {^LS_\mf{w}}$ is the $L$-parameter satisfying $\varphi = {^L\jmath}\circ\varphi_\mf{w}$.
\end{lem}
\begin{proof}
It is clear that we obtain a map of this form, and that it is injective. It fits into the following exact sequences
\[ \xymatrix{
1\ar[r]&  \tx{Cent}(\varphi_\mf{w},\hat S_\mf{w})\ar[r]\ar[d]&\tx{Cent}(\varphi_\mf{w},\hat S_\mf{w} \rtimes A^{[\varphi]})\ar[r]\ar[d]&A^{[\varphi]}\ar[r]\ar@{=}[d]&1\\
1\ar[r]&S_\varphi\ar[r]&\tilde S_\varphi\ar[r]&A^{[\varphi]}\ar[r]&1.
}
\]
The regularity of $\varphi|_{\C^\times}$ implies that the first vertical map is bijective. The 5-lemma completes the proof.
\end{proof}

Recall that we have fixed a representation $\pi \in \Pi_\varphi(G)$. It corresponds to a $G(\R)$-conjugacy class of pairs $(S,\tau)$ consisting of an elliptic maximal torus $S \subset G$ and a genuine character $\tau$ of $S(\R)_G$. On the other hand, we have the $A$-stable representative $(S_\mf{w},\tau_\mf{w})$ of the Harish-Chandra parameter of $\pi_\mf{w}$. Let $\eta \in \pi_0([\hat{\bar S_\mf{w}}]^+)^*$ denote the element corresponding under the generalized Tate-Nakayama isomorphism to $\tx{inv}(d\tau_\mf{w},d\tau) \in H^1(u \to W,Z(G) \to S_\mf{w})$. Transport $\eta$ to a character of $\pi_0(S_\varphi^+)$ under the isomorphism of Lemma \ref{lem:is}.

\begin{cor} \label{cor:is}
There is a natural bijection between $\tx{Irr}(\tilde S_\varphi^{+,[z]},\eta)$ and $X^\tx{irr}(f,\tau)$.
\end{cor}
\begin{proof}
Let $A^{[\eta]}$ be the stabilizer in $A^{[\varphi]}$ of the character $\eta : \pi_0([\hat{\bar S_\mf{w}}]^+) \to \C^\times$, let $\tilde S_\varphi^{+,[\eta]}$ be the preimage of $A^{[\eta]}$ in $\tilde S_\varphi^+$, and let $\tx{Irr}(\tilde S_\varphi^{+,[\eta]},\eta)$ be the set of irreducible representations of $\tilde S_\varphi^{+,[\eta]}$ whose restriction to $[\hat{\bar S_\mf{w}}]^+$ is $\eta$-isotypic. Induction produces a bijection 
\[ \tx{Irr}(\tilde S_\varphi^{+,[\eta]},\eta) \to  \tx{Irr}(\tilde S_\varphi^{+,[z]},\eta). \]
Lemma \ref{lem:is} gives the bijection
\[ \tx{Irr}([\hat S_\mf{w} \rtimes A^{[\eta]}]^+_{\varphi_\mf{w}},\eta) \to \tx{Irr}(\tilde S_\varphi^{+,[\eta]},\eta). \]
The unique admissible isomorphism $S_\mf{w} \to S$ that transports $d\tau_\mf{w}$ to $d\tau$ extends to an inner twist $\tilde S_\mf{w} \to \tilde S$ which is rigidified by the element $\tx{inv}(d\tau_\mf{w},d\tau)$. Conjecture 5.12 of \cite{KalLLCD}, which in the case of tori is proved in \cite[\S8]{KalLLCD}, identifies the left hand side of the last bijection with the set of irreducible representations of $\tilde S(\R)$ whose restriction to $S(\R)$ contains the character $\tau_\mf{w}$ with $L$-parameter $\varphi_\mf{w}$. This character equals $\tau\otimes\rho_{\mf{u}_f}^{-1}$. Tensoring by $\rho_{\mf{u}_f}$ and recalling that this character extends naturally to $\tilde S(\R)_G$, we obtain a bijection 
\[ X^\tx{irr}(f,\tau) \to \tx{Irr}([\hat S_\mf{w} \rtimes A^{[\eta]}]^+_{\varphi_\mf{w}},\eta).\qedhere\]
\end{proof}

We have now obtained the following maps, for any fixed $\pi \in \Pi_\varphi(G)$ and pair $(f,\tau)$ representing the the Harish-Chandra parameter of $\pi$:
\begin{enumerate}
  \item the bijection $X^\tx{irr}(f,\tau) \to \Pi_\varphi(\tilde G,\pi)$ of Lemma \ref{lem:duflo2} and
  \item the bijection $X^\tx{irr}(f,\tau) \to \tx{Irr}(\tilde S_\varphi^{[z]},\eta)$ of Corollary \ref{cor:is}.
\end{enumerate}
The composition of these two is easily seen to be independent of the choice of $(f,\tau)$ in its $G(\R)$-conjugacy class. It is \emph{a} bijection of the form \eqref{eq:is2}, and splicing over all $\pi$ we obtain \emph{an} injection of the form \eqref{eq:is}. It will however turn out that this is \emph{not the} desired bijection. In fact, the desired bijection differs from the one just obtained by precomposing it with the operation $\otimes \epsilon_{\tilde G}$, where $\epsilon_{\tilde G}$ is the sign character constructed in \S\ref{sub:signchar}, with respect to the quasi-split group $\tilde G$. A first indication is given by the following lemma.

\begin{lem} \label{lem:whit}
  Under the composed map $\tx{Irr}(\tilde S_\varphi,1) \to X^\tx{irr}(f_\mf{w},\tau_\mf{w}) \to \tx{Irr}(\tilde G)$ of Lemma \ref{lem:duflo2} and Corollary \ref{cor:is}, the trivial representation is mapped to the representation $\tilde\pi_\mf{w} \otimes \epsilon_{\tilde G}$, where $\tilde\pi_\mf{w}$ is the Whittaker extension to $\tilde G(\R)$ of the representation $\pi_\mf{w}$.
\end{lem}
\begin{proof}
  The bijection $\tx{Irr}(\tilde S_\varphi,1) \to \tx{Irr}(\pi_0(\hat S_\mf{w} \rtimes A),1) \to X^\tx{irr}(f_\mf{w},\tau_\mf{w})$ sends the trivial representation to the extension of $\tau_\mf{w} : S_\mf{w}(\R) \to \C^\times$ to $S_\mf{w}(\R) \rtimes A$ that is trivial on the subgroup $1 \rtimes A$, according to the formula in Proposition \ref{pro:luo}. Let us call this extension $\tilde\tau_\mf{w}$. According to Theorem \ref{thm:rtci-sign}, the representation of $\tilde G(\R)$ that Duflo's construction assigns to $\tilde\tau_\mf{w}$ is the twist of $\tilde\pi_\mf{w}$ by $\epsilon_{\tilde G}$.
\end{proof}

Thus, we shall take the bijection \eqref{eq:is1}, and hence the injection \eqref{eq:is}, to be obtained by composing the maps from Lemma \ref{lem:duflo2} and Corollary \ref{cor:is}, \emph{and} then twisting by the character $\epsilon_{\tilde G}$. More precisely, if $\tilde\pi \in \Pi_\varphi(\tilde G_z)$, $\tilde\rho \in \tx{Irr}(\tilde S_\varphi^{[z]},[z])$, and $\tilde\tau$ are in correspondence $\tilde\rho \mapsfrom \tilde\tau \mapsto \tilde\pi$ under Corollary \ref{cor:is} and Lemma \ref{lem:duflo2}, respectively, then the injection \eqref{eq:is} maps $\tilde\pi \otimes \epsilon_{\tilde G}$ to $\tilde\rho$.

Besides Lemma \ref{lem:whit}, further and more robust evidence that this is the correct construction of the injection \eqref{eq:is} will be provided by Theorem \ref{thm:main}.

\begin{lem} \label{lem:rest}
  The injection \eqref{eq:is} constructed here satisfies \cite[Conjecture 7.1.1]{KalLLCD}.
\end{lem}
\begin{proof}
Consider a map of finite groups $B \to A$, which we may assume injective by \cite[Remark 7.1.4]{KalLLCD}, and write $G^A = G \rtimes A$ in place of $\tilde G$, as well as $G^B = G \rtimes B$. Let $\varphi : W_\R \to {^LG}$ be discrete, let $\rho^A \in \tx{Irr}(\pi_0(S_\varphi^{A,+,[z]}),{[z]})$ and $\rho^B \in \tx{Irr}(\pi_0(S_\varphi^{B,+,[z]}),{[z]})$, and let $\pi^A$ and $\pi^B$ be the corresponding representations of $G^A_z(\R)$ and $G^B_z(\R)$.

Since $\epsilon_{G_A}|_{G^B}=\epsilon_{G^B}$, it is enough to examine the naive versions of the injection \eqref{eq:is}, where we ignore the twist by $\epsilon_{\tilde G}$.

If the restrictions to $G_z(\R)$ of $\pi^A$ and $\pi^B$ are disjoint, then it follows from \eqref{eq:subpack}, \eqref{eq:subpack1}, and \eqref{eq:is2}, that the restrictions of $\rho^A$ and $\rho^B$ to $\pi_0(S_\varphi^+)$ are disjoint. In that case the multiplicity of $\pi^B$ in $\tx{Res}\,\pi^A$ and the multiplicity of $\rho^B$ in $\tx{Res}\,\rho^A$ are both zero, hence equal.

Assume now that the restrictions to $G_z(\R)$ of $\pi^A$ and $\pi^B$ are not disjoint and pick an irreducible representation $\pi$ of $G_z(\R)$ that is common to both restrictions. Then $\pi^A \in \Pi_\varphi(G^A_z,\pi)$ and $\pi^B \in \Pi_\varphi(G^B_z,\pi)$. According to Lemma \ref{lem:duflo2} there exist irreducible $\tau^A \in X^A(f)$ and $\tau^B \in X^B(f)$ such that $\pi^A=\pi_{f,\tau^A}$ and $\pi^B=\pi_{f,\tau^B}$. According to Lemma \ref{lem:duflo1}(2) we have $\tx{Res}\,\pi^A=\bigoplus_c \pi_{f^c,(\tau^A)^c}$, where the index $c$ runs over the set $S^A(\R) \lmod G^A_z(\R) / G^B_z(\R)$, where $S^A(\R)$ is the centralizer of $f$ in $G^A_z(\R)$. Note that $(\tau^A)^c$ need not be irreducible. Lemma \ref{lem:duflo1}(1,4) tells us that the multiplicity of $\pi^B$ in $\tx{Res}\,\pi^A$ equals the sum over those $c$ for which $f^c$ is $G^B_z(\R)$-conjugate to $f$ of the multiplicity of $\tau^B$ in $(\tau^A)^c$. But the condition that $f^c$ and $f$ be $G^B_z(\R)$-conjugate is equivalent to $c=1$, so we obtain the multiplicity of $\tau^B$ in $\tau^A|_{S^B(\R)_f}$.

Now $\tau^A$ and $\tau^B$ correspond to $\rho^A$ and $\rho^B$ under the bijection \eqref{eq:is3}. The construction of this bijection shows that the multiplicity of $\tau^B$ in $\tx{Res}\,\tau^A$ equals the multiplicity of $\rho^B$ in $\tx{Res}\,\rho^A$.
\end{proof}

\section{Character identities} \label{sec:charid}

\subsection{Statement of main result}

We continue with the set-up of \S\ref{sub:sqpac}. Thus, $\tilde G = G \rtimes A$ is a quasi-split disconnected reductive group, $z \in Z^1(u \to \mc{E}^\tx{rig},Z(G)^A \to G)$, $\tilde G_z$ is the $z$-twist of $\tilde G$, and $\varphi : W_\R \to {^LG}$ is a discrete $L$-parameter. 

Let $\tilde s \in \tilde S_\varphi^{+,[z]}$. Associated to $\tilde s$ is a refined twisted endoscopic datum $(H,\tilde s,\mc{H},\xi)$ as discussed in \cite[\S4.8]{KalLLCD}. We choose a $z$-pair $(H_1,\xi_1)$, where $\xi_1 : \mc{H} \to {^LH_1}$. Let $\varphi_1 : W_\R \to {^LH_1}$ be the composition of $\varphi$, which by construction takes image in $\mc{H}$, with $\xi_1$. Then $\varphi_1$ is a discrete $L$-parameter. We have the $L$-packet $\Pi_\varphi(\tilde G_z)$ constructed in \S\ref{sub:sqpac} and the $L$-packet $\Pi_{\varphi_1}(H_1)$ constructed classically by Langlands, as reviewed in \cite[\S4]{DPR}. We have the virtual character on $\tilde G_z(\R)$
\[ \Theta_{\varphi,z}^{\tilde s,\mf{w}} = e(G_z)\sum_{\tilde\pi \in \Pi_\varphi(\tilde G_z)} \tx{tr}(\rho_{\tilde\pi}(\tilde s)) \Theta_{\tilde\pi}, \]
where $\rho_{\tilde\pi} \in \tx{Irr}(\tilde S_\varphi^{+,[z]})$ is the image of $\tilde\pi$ under the injection \eqref{eq:is}, as well as the virtual character on $H_1(\R)$
\[ S\Theta_{\varphi_1} = \sum_{\pi_1 \in \Pi_{\varphi_1}(H_1)} \Theta_{\pi_1}. \]
The following is the second main result of this paper.
\begin{thm} \label{thm:main}
Let $\tilde f \in \mc{C}^\infty_c(\tilde G(\R))$ and $f_1 \in \mc{C}^\infty_c(H_1(\R))$ be matching functions as in \cite[Lemma 4.12.1]{KalLLCD}. Then 
\[ \Theta_{\varphi,z}^{\tilde s,\mf{w}}(\tilde f) = S\Theta_{\varphi_1}(f_1). \]
\end{thm}
This verifies \cite[Conjecture 5.6.1(2)]{KalLLCD}. In fact, as shown in \cite[\S7.2]{KalLLCD}, \cite[Conjecture 5.6.1(2)]{KalLLCD} follows from \cite[Conjecture 7.2.1]{KalLLCD}, which states that
\begin{equation} \label{eq:main}
e(G_z)\sum_{\substack{\pi \in \Pi_\phi(G_z)\\\pi\circ a\cong\pi}} \tx{tr}\,(\pi\boxtimes\rho^\vee)^\tx{can}(\tilde f,\tilde s^{-1}) = S\Theta_{\varphi_1}(f_1),
\end{equation}
where now $a \in A$ is the image of $\tilde s^{-1}$, $\tilde f \in \mc{C}^\infty_c([G \rtimes a]_z(\R))$ and $f_1 \in \mc{C}_c^\infty(H_1(\R))$ match in the sense of \cite[(7.2.1)]{KalLLCD}. In what follows we will prove \eqref{eq:main}.

\subsection{Admissible isomorphisms} \label{sub:admiso}

We fix a pinning $(\hat T,\hat B,\{Y_\alpha\})$ of $\hat G$ that is stable under $A$ and $\Gamma$. Conjugating $(\varphi,\tilde s)$ by $\hat G$ if necessary we may assume that $\tilde s \in \hat T \rtimes A$. In fact, we have the following stronger statement that will not be necessary, but may be worth recording.

\begin{lem}
  We may conjugate $\varphi$ under $\hat G$ to that $\tilde S_\varphi \subset \hat T \rtimes A$.
\end{lem}
\begin{proof}
Compose $\varphi$ with the projection $^LG \to {^LG}/Z(\hat G)$ to assume $G$ simply connected. Since $\varphi$ is discrete the centralizer of $\varphi(\C^\times)$ is a maximal torus of $\hat G$, so we may conjugate $\varphi$ to assume this maximal torus is $\hat T$. Then $\varphi(z)=(z/\bar z)^\lambda$ with $\lambda \in X_*(\hat T)_\Q$ regular. Conjugating $\varphi$ by $N_{\hat G}(\hat T)$ we may assume that $\lambda$ is $\hat B$-dominant. Thus $(\hat T,\hat B)$ are now determined by $\varphi$, so $\tilde S_\varphi$ lies in the stabilizer of $(\hat T,\hat B)$ in $\hat G \rtimes A$, but this stabilizer is $\hat T \rtimes A$.
\end{proof}

Pulling back $(\hat T,\hat B)$ under $\xi$ we obtain a pair $(\hat T^H,\hat B^H)$ of $\hat H$. Let $S^H \subset H$ be a maximal $\R$-torus. It is shown in \cite[Lemma 3.3.B]{KS99} that there exist a Borel $\C$-subgroup $C^H \subset H$ containing $S^H$, an $a$-stable maximal $\R$-torus $S' \subset G$, and $a$-stable Borel $\C$-subgroup $C' \subset G$ containing $S'$ (the existence of $C'$ makes $S'$ into an $a$-admissible maximal torus), such that the data $(S^H,C^H)$, $(\hat T^H,\hat B^H)$, $(\hat T,\hat B)$, and $(S',C')$, lead to an isomorphism $\eta : S^H \to (S')_a$ of $\R$-tori. Such an isomorphism is called \emph{admissible}. We let $s_\eta \in \hat S'_a$ denote the image in the group $\hat S'_a$ of $a$-coinvariants of $\hat S'$ of the transport of $s \in \hat T$ under the isomorphism $\hat T \to \hat S'$ determined by the pairs $(\hat T,\hat B)$ and $(S',C')$. It is shown in \cite[Lemma 4.2.A]{KS99} that $s_\eta$ depends only on $\eta$, and not on the remaining data.

\subsection{Conjugacy classes in $[G \rtimes a]_z^\tx{sr}(F)$} \label{sub:tc}

In this subsection we work over an arbitrary local field $F$ of characteristic zero (most of the discussion carries over to positive characteristic, with some additional care). It is well-known that the set of strongly regular semi-simple conjugacy classes in $G_z(F)$ can be described as the disjoint union 
\begin{equation} \label{eq:rc}
\coprod_S S_\tx{sr}(F)/\Omega_F(S,G_z),  
\end{equation}
where $S$ runs over the set (of representatives for) the $G_z(F)$-conjugacy classes of maximal tori $S \subset G_z$, $S_\tx{sr}$ denotes the set of strongly regular elements of $S$, and $\Omega_F(S,G_z)=N(S,G_z)(F)/S(F)$ is the rational Weyl group.

The above set has a natural topology and can be equipped with a Borel measure: the topology of $S(F)$ induces the subspace topology on $S_\tx{sr}(F)$ and the quotient topology on $S_\tx{sr}(F)/\Omega_F(S,G_z)$; choosing a Haar measure on each $S(F)$, restricting it to the subset $S_\tx{sr}(F)$ (whose complement has measure zero), and taking the quotient by the counting measure on $\Omega_F(S,G_z)$, induces a Borel measure.

There is a version of this discussion that applies to stable conjugacy, namely the set of stable conjugacy classes of strongly regular semi-simple elements is the disjoint union
\begin{equation} \label{eq:sc}
  \coprod_S S_\tx{sr}(F)/\Omega(S,G_z)(F),
\end{equation}
where now $S$ runs over the set (of representatives for) the stable classes of maximal tori $S \subset G_z$ and $\Omega(S,G_z)=N(S,G_z)/S$ is the absolute Weyl group, a finite algebraic group defined over $F$. One puts a topology and measure on this space in an analogous way.

The purpose of this section is to review a similar description for the set of $G_z(F)$-orbits of strongly regular semi-simple elements of $[G \rtimes a]_z(F)$, where $a \in A^{[z]}$ is fixed, as well as the set of stable orbits, following the discussion of \cite[\S5]{HL17}.

An element $\tilde\delta \in [G \rtimes a]_z^\tx{sr}$ gives rise to a maximal torus $S \subset G_z$, namely its ``double centralizer'' $S=Z(Z(\tilde\delta,G_z),G_z)$, as well as the ``Cartan subspace'' $S^\natural = S \cdot \tilde\delta$ of the twisted $G_z$-space $[G \rtimes a]_z$. Each element of $S^\natural$ acts on $S$ by conjugation, and all elements induce the same automorphism, which we shall call $\tau$. The set $S^\natural_\tx{sr} = S^\natural \cap [G \rtimes a]_z^\tx{sr}$ is Zariski open and dense, and for each $\tilde\delta \in S^\natural_\tx{sr}$ we have $Z(\tilde\delta,G_z)=S^\tau$. 

By definition, each element of $[G \rtimes a]_z^\tx{sr}$ belongs to a unique Cartan subspace $S^\natural$. Therefore, two elements of $S^\natural_\tx{sr}$ are $G_z$-conjugate if and only if they are $N(S^\natural,G_z)$-conjugate. The same statement works $F$-rationally: two elements of $S^\natural_\tx{sr}(F)$ are $G_z(F)$-conjugate if and only if they are $N(S^\natural,G_z)(F)$-conjugate. We have the inclusions $S \subset N(S^\natural,G_z) \subset N(S,G_z)$ and the action of $N(S^\natural,G_z)$ on $S^\natural$ by conjugation factors through a faithful action of $N(S^\natural,G_z)/S^\tau$. 

We define $\Omega(S^\natural,G_z)=N(S^\natural,G_z)/S^\tau$. This is an algebraic $F$-group, not necessarily finite. We define $\Omega_F(S^\natural,G_z)=N(S^\natural,G_z)(F)/S^\tau(F)$, and have the natural inclusion $\Omega_F(S^\natural,G_z) \subset \Omega(S^\natural,G_z)(F)$ that may be proper.

Given a Cartan subspace $S^\natural \subset [G \rtimes a]_z$, the map
\[ (G_z(F)/S^\tau(F)) \times S^\natural_\tx{sr}(F) \to [G \rtimes a]_z^\tx{sr}(F),\qquad (g,\tilde\delta) \mapsto g \tilde\delta g^{-1} \]
is a submersion whose fibers are torsors for the action of $\Omega_F(S^\natural,G_z)$ given by 
\[ n \cdot (g,\tilde\delta) = (gn^{-1},n\tilde\delta n^{-1}). \]
Therefore, the set of $G_z(F)$-conjugacy classes in $[G \rtimes a]_z^\tx{sr}(F)$ can be described as
\begin{equation} \label{eq:rtc}
   \coprod_{S^\natural} S^\natural_\tx{sr}(F)/\Omega_F(S^\natural,G_z),
\end{equation}
where $S^\natural$ runs over the set $G_z(F)$-conjugacy classes of Cartan subspaces of $[G \rtimes a]_z$. This set is topologized by the topology on each $S^\natural_\tx{sr}(F)$ obtained as a subspace of $S^\natural(F)$, the latter inheriting its topology from $S(F)$, under which it is a torsor.

Let us now describe the stable classes. Two elements $\tilde\delta_1,\tilde\delta_2 \in [G \rtimes a]_z^\tx{sr}(F)$ are called stably conjugate if there exists $g \in G(\bar F)$ such that $g\tilde\delta_1g^{-1}=\tilde\delta_2$. If that is the case then $g^{-1}\sigma(g) \in S_1^\tau$, where $S_1^\tau$ is the centralizer of $\tilde\delta_1$, and $S_1$ is the double centralizer of $\tilde\delta_1$, i.e. the centralizer of $S_1^\tau$. Therefore, $\tx{Ad}(g) : S_1 \to S_2$ and $\tx{Ad}(g) : S_1^\natural \to S_2^\natural$ are isomorphisms defined over $F$. We can now define two Cartan subspaces $S_1^\natural, S_2^\natural \subset [G \rtimes a]_z$ as stably conjugate if there exists $g \in G(\bar F)$ such that $gS_1^\natural g^{-1}=S_2^\natural$ and the isomorphism $\tx{Ad}(g) : S_1^\natural \to S_2^\natural$ is defined over $F$. Again this implies that $g^{-1}\sigma(g) \in S_1^\tau$, where $S_1^\tau$ is the centralizer of $S_1^\natural$, and that $\tx{Ad}(g) : S_1 \to S_2$ is an isomorphism defined over $F$, where $S_i$ is the double centralizer of $S_i^\natural$. Two elements $\tilde\delta_1,\tilde\delta_2$ lying in $S^\natural_\tx{sr}(F)$ are stably conjugate if and only if they are conjugate under $\Omega(S^\natural,G_z)(F)$. Therefore, the set of stable classes in $[G \rtimes a]_z^\tx{sr}(F)$ has the description
\begin{equation} \label{eq:stc}
 \coprod_{S^\natural} S^\natural_\tx{sr}(F)/\Omega(S^\natural,G_z)(F),
\end{equation}
where now $S^\natural$ runs over the set (of representatives of) stable conjugacy classes of Cartan subspaces of $[G \rtimes a]_z$.

The fact that $\Omega_F(S^\natural,G_z)$ is usually infinite can sometimes be inconvenient. This group contains the normal subgroup $S(F)/S^\tau(F)$ with finite index. The quotient embeds into $\Omega_F(S,G_z)=N(S,G_z)(F)/S(F)$ and one can check that the image of this embedding is precisely the subgroup of $\Omega_F(S,G_z)$ consisting of those elements that are fixed by the conjugation action on $\Omega_F(S,G_z)$ of one, hence any, element of $S^\natural$. If we write $\Omega_F(S,G_z)^\tau$ for this subgroup, we have the exact sequence
\[ 1 \to S(F)/S^\tau(F) \to \Omega_F(S^\natural,G_z) \to \Omega_F(S,G_z)^\tau \to 1. \]

We can now divide the action of $\Omega_F(S^\natural,G_z)$ on $S^\natural(F)$ in stages. We first divide the action of the subgroup $S(F)/S^\tau(F)$. For $t \in S(F)/S^\tau(F)$ and $\tilde\delta \in S^\natural(F)$ we have $t\tilde\delta t^{-1}=t\tau(t)^{-1}\tilde\delta$. Therefore, the quotient of $S^\natural(F)$ by this action, which we shall call $S^\natural(F)_\tau$, is a torsor under the group $S(F)_\tau=(1-\tau)S(F) \lmod S(F)$ of $\tau$-coinvariants in $S(F)$. We write $S^\natural_\tx{sr}(F)_\tau$ for the image of $S^\natural_\tx{sr}(F)$ in $S^\natural(F)_\tau$. The action of $\Omega_F(S^\natural,G_z)$ on $S^\natural_\tx{sr}(F)_\tau$ now factors through the finite quotient $\Omega_F(S,G_z)^\tau$. 

Dividing out the action of $S(F)/S^\tau(F)$ from the source of the above submersion leads to the submersion
\[ G(F)/S(F) \times S^\natural_\tx{sr}(F)_\tau \to [G \rtimes a]_z^\tx{sr}(F), \]
whose fibers are torsors for the finite group $\Omega_F(S,G_z)^\tau$. Therefore, we obtain the alternative description of the set of $G_z(F)$-conjugacy classes in $[G \rtimes a]_z^\tx{sr}(F)$ as 
\begin{equation} \label{eq:rtc1}
  \coprod_{S^\natural} S^\natural_\tx{sr}(F)_\tau/\Omega_F(S,G_z)^\tau,
\end{equation}
where $S^\natural$ runs over the set $G_z(F)$-conjugacy classes of Cartan subspaces of $[G \rtimes a]_z$. This set is topologized by the topology on each $S^\natural_\tx{sr}(F)_\tau$, which as a torsor under $S(F)_\tau$ receives its topology from that group, and in turn the topology on $S(F)_\tau$ is the quotient of the topology of $S(F)$.

Moreover, the above set is endowed with a measure, by equipping each $S^\natural_\tx{sr}(F)_\tau/\Omega_F(S,G_z)^\tau$ with the quotient measure of the counting measure on the Weyl group $\Omega_F(S,G_z)^\tau$ and the measure on $S^\natural_\tx{sr}(F)_\tau$ obtained from a Haar measure on $S(F)_\tau$ via the torsor structure. 

A similar discussion applies to stable classes. Here we have the normal subgroup $S/S^\tau \subset \Omega(S^\natural,G_z)$ of finite index, whose quotient embeds into $\Omega(S,G_z)$ as the subgroup of $\tau$-fixed points. This leads to the exact sequence 
\[ 1 \to S/S^\tau \to \Omega(S^\natural,G_z) \to \Omega(S,G_z)^\tau \to 1, \]
which, upon taking $F$-points, becomes
\[ 1 \to [S/S^\tau](F) \to \Omega(S^\natural,G_z)(F) \to \Omega(S,G_z)^\tau(F) \to H^1(F,S/S^\tau) \to \cdots. \]
Write $\Omega(S,G_z)^\tau(F)'$ for the image of $\Omega(S^\natural,G_z)(F)$ in $\Omega(S,G_z)^\tau(F)$.

We can now divide out from $S^\natural(F)$ the action of $[S/S^\tau](F)$ by conjugation, which again is the same as the action by left-multiplication by $t\tau(t)^{-1}$ for $t \in [S/S^\tau](F)$. The map $1-\tau : S/S^\tau \to S$ induces the exact sequence
\[ 1 \to S/S^\tau \to S \to S_\tau \to 1, \]
where $S_\tau$ is the torus of $\tau$-coinvariants of $S$, and taking $F$-points we arrive at 
\[ 1 \to [S/S^\tau](F)\to S(F) \to S_\tau(F) \to H^1(F,S/S^\tau) \to \cdots. \]
Letting $S^\natural_\tau$ be the quotient of $S^\natural$ by the conjugation action of $S$, equivalently the left-multiplication action by $(1-\tau)S \cong S/S^\tau$, we see that $[S/S^\tau](F) \lmod S^\natural(F)$ is a (usually proper) subspace of $S^\natural_\tau(F)$, which we shall call $S^\natural_\tau(F)'$. The action of $\Omega(S^\natural,G_z)(F)$ on $S^\natural_\tau(F)$ factors through the quotient $\Omega(S,G_z)(F)'$ and preserves the subspace $S^\natural_\tau(F)'$. Note that, while $\Omega(S,G_z)$ acts on $S^\natural_\tau$, the subgroup $\Omega(S,G_z)(F)$ generally doesn't preserve the subset $S^\natural_\tau(F)$.
The stable classes in $[G \rtimes a]_z^\tx{sr}(F)$ are then alternatively described as 
\begin{equation} \label{eq:stc1}
  \coprod_{S^\natural} S^\natural_{\tau,\tx{sr}}(F)'/\Omega(S,G_z)^\tau(F)',
\end{equation}
where now $S^\natural$ runs over the set of (representatives for) the stable classes of Cartan subspaces of $[G \rtimes a]_z$.

The descriptions of the rational and stable conjugacy classes in $[G \rtimes a]_z^\tx{sr}(F)$ are related to each other via the map
\begin{equation} \label{eq:snt}
  S^\natural(F)_\tau \to S^\natural_\tau(F)
\end{equation}
induced by the projection $S^\natural(F) \to S^\natural_\tau(F)$. It is useful to keep in mind that this map is neither surjective nor injective. Its fibers account for $G_z(F)$-conjugacy classes that meet $S^\natural(F)$ and become stably conjugate by an element of $S(\bar F)$. Its image, which we have denoted by $S^\natural_\tau(F)'$, accounts for those elements of $S^\natural_\tau(F)$ that come from stable classes of elements of $S^\natural(F)$. The following lemma provides some quantitative information.

\begin{lem} \label{lem:snt}
  \begin{enumerate}
    \item The kernel of \eqref{eq:snt} is a torsor under the finite abelian group $\tx{ker}(H^1(F,S^\tau) \to H^1(F,S))$.
    \item The image $S^\natural_\tau(F)'$ of \eqref{eq:snt} is open and closed in $S^\natural_\tau(F)$ and we have 
    \[ S^\natural_\tau(F) = \coprod_{\alpha} \dot\alpha \cdot S^\natural_\tau(F)', \]
    where $\alpha$ runs over $\tx{ker}(H^1(F,S/S^\tau) \stackrel{1-\tau}{\to} H^1(F,S))$ and $\dot\alpha \in S_\tau(F)$ is any lift of $\alpha$ under the connecting homomorphism $S_\tau(F) \to H^1(F,S/S^\tau)$.
    \item The inclusion $S^\natural_\tau(F)' \to S^\natural_\tau(F)$ induces an injection 
    \[ S^\natural_{\tau,\tx{sr}}(F)'/\Omega(S,G_z)^\tau(F)' \to S^\natural_{\tau,\tx{sr}}(F)/\Omega(S,G_z)^\tau(F). \]
  \end{enumerate}
\end{lem}
\begin{proof}
  The exact sequence
  \[ 1 \to S^\tau \to S \stackrel{1-\tau}{\to} S \to S_\tau \to 1 \]
  leads to the diagram
  \[ \scalebox{0.95}{\xymatrix{
    1\ar[r]&S(F)/S^\tau(F)\ar[r]\ar[d]&S(F)\ar[r]\ar@{=}[d]&S(F)_\tau\ar[r]\ar[d]&1\\
    1\ar[r]&[S/S^\tau](F)\ar[d]\ar[r]&S(F)\ar[r]&S_\tau(F)\ar[r]&H^1(F,S/S^\tau)\ar[r]&H^1(F,S)\\
    &H^1(F,S^\tau)\ar[d]\\
    &H^1(F,S)
  }}\]
  which shows that the image of $S(F)_\tau \to S_\tau(F)$, being the kernel of $S_\tau(F) \to H^1(F,S/S^\tau)$ is an open subgroup of finite index, and the cokernel of $S(F)_\tau \to S_\tau(F)$ equals the kernel of $H^1(F,S/S^\tau) \to H^1(F,S)$, hence (1).
  
  The kernel of $S(F)_\tau \to S_\tau(F)$ is isomorphic to the kernel of $H^1(F,S^\tau) \to H^1(F,S)$ by the snake lemma, hence (2).

  The claim in (3) can be reformulated as follows: Let  $\tilde\delta,\tilde\delta' \in S^\natural_\tx{sr}(F)$. If their images in $S^\natural_\tau(F)$ are in the same $\Omega(S,G_z)^\tau(F)$-orbit, then they are in the same $\Omega(S,G_z)^\tau(F)'$-orbit. Being in the same $\Omega(S,G_z)^\tau(F)$-orbit implies the existence of $n \in N(S^\natural,G_z)(\bar F)$ with $n\tilde\delta n^{-1}=\tilde\delta'$. Then $n^{-1}\sigma_z(n) \in S^\tau$, hence the image of $n$ in $\Omega(S^\natural,G_z)$ is an $F$-point.
\end{proof}

\subsection{Norms} \label{sub:norm}

We continue with an arbitrary local field $F$ of characteristic zero. A key role in the proof of Theorem \ref{thm:main} is played by the Kottwitz--Shelstad notion of a norm. In this subsection we will review this notion, following \cite[Definition 3.3.2]{KalLLCD}, and then relate it to the discussion of conjugacy classes and Cartan subspaces of \S\ref{sub:tc}.

Let $\tilde\delta = \delta \rtimes a \in [G \rtimes a]_z^\tx{sr}(F)$. An (abstract) norm of $\tilde\delta$ is a pair $(S',\gamma')$ consisting of an $a$-stable maximal $F$-torus $S' \subset G$ that is contained in an $a$-stable $\bar F$-Borel subgroup of $G$ (such $S'$ is called $a$-admissible), and an element $\gamma' \in S'_a(F)$, where $S'_a=S'/(1-a)S'$ is the quotient of $a$-coinvariants of $S'$, such that there exists $g \in G(\bar F)$ with the property that $\delta' \rtimes a := \tilde\delta' := g^{-1}\tilde\delta g$ lies in $S'(\bar F) \rtimes a$ and the image of $\delta'$ in $S'_a(\bar F)$ equals $\gamma'$. According to \cite[Lemma 3.3.4(1)]{KalLLCD} an (abstract) norm for $\tilde\delta$ exists and is unique up to stable conjugacy under $G^{a,\circ}$. Moreover, according to \cite[Lemma 3.3.4(2)]{KalLLCD}, $S=\tx{Cent}(\tx{Cent}(\tilde\delta,G_z),G_z)$ is a maximal $F$-torus of $G_z$ and $\tx{Ad}(g) : S' \to S$ is an isomorphism defined over $F$. We say that $\gamma^H \in H(F)$ is a \emph{norm} of $\tilde\delta$, or that $\gamma^H$ and $\tilde\delta$ are \emph{related} if there exist an admissible isomorphism $S^H \to S'_a$ that transports $\gamma^H$ to $\gamma'$, where $S^H$ is the centralizer of $\gamma^H$. 

We note here that the pair $(\tilde\delta,g)$ determines $(S',\gamma')$. Indeed, if we set $\tilde\delta'=g^{-1}\tilde\delta g$, then $\tilde\delta'$ is still strongly regular semi-simple and $S'$ is its double centralizer, while $\gamma'$ is the image of $\delta'$ in $S'_a$. In the following proposition we will describe the possible $g$ and reinterpret the concept of a norm in a way that is uniform for a given Cartan subspace. 

\begin{rem}
  If $S^\natural$ is a Cartan subspace of $[G \rtimes a]_z$ then there exists a $\bar F$-Borel pair of $G$ normalized by $S^\natural$. The maximal torus in any such Borel pair equals the double centralizer $S$ of $S^\natural$. Indeed, due to \cite[Lemma 3.11.2]{HL17} these claims hold for any one element $\tilde\delta$ of the non-empty (\cite[Proposition 3.13.1(2)]{HL17}) subset $S^\natural_\tx{sr}$, after which we use $S^\natural=S \cdot \tilde\delta$.
\end{rem}

\begin{pro} \label{pro:norm}
  Let $S^\natural$ be a Cartan subspace of $[G \rtimes a]_z$ defined over $F$ and let $S$ be its double centralizer. 
  \begin{enumerate}
    \item There exists an element $g \in G(\bar F)$ satisfying the following equations:
    \begin{enumerate}
      \item $t_\sigma := z_\sigma \sigma(g)g^{-1} \in S(\bar F)$,
      \item $\tilde\delta_g := g \cdot (1 \rtimes a) \cdot g^{-1} \in S^\natural(\bar F)$.
    \end{enumerate}
    \item The torus $S' := \tx{Ad}(g)^{-1}S$ is defined over $F$ and $a$-admissible. More precisely, $S'$ is defined over $F$, $a$-stable, and  if $C$ is any $\bar F$-Borel subgroup containing $S$ and normalized by $S^\natural$, then $C'=\tx{Ad}(g)^{-1}C$ is $a$-stable. 
    \item The isomorphism $\tx{Ad}(g) : S' \to S$ is defined over $F$ and translates the automorphism $\tau$ of $S$, by which any one element of $S^\natural$ acts via conjugation, to the automorphism $a$ of $S'$.
    \item The element $g$ is unique up to replacement by $sgx$ with $s \in S(\bar F)$ and $x \in G^{a,\circ}(\bar F)$ satisfying $g\sigma(x)x^{-1}g^{-1} \in S(\bar F)$.
    \item The image of $\tilde\delta_g$ in $S_\tau^\natural=(1-\tau)S \lmod S^\natural$ is independent of the choice of $g$ and lies in $S_\tau^\natural(F)$; we will call it $a_\tau$. 
    \item Let $\tilde\delta \in S^\natural_\tx{sr}(F)$ and let $\gamma = \tilde\delta_\tau \cdot a_\tau^{-1} \in S_\tau(F)$, where $\tilde\delta_\tau \in S_\tau^\natural(F)$ is the image of $\tilde\delta$. Let $(S',\gamma') := \tx{Ad}(g)^{-1}(S,\gamma)$. Then $(S',\gamma')$ is an abstract norm of $\tilde\delta$, and every abstract norm arises this way.
    \item In the notation of the above point, let $\delta := \tilde\delta \cdot \tilde\delta_g^{-1} \in S(\bar F)$. The tuple $(t_\sigma^{-1},\delta)$ lies in $Z^1(F,S \stackrel{1-\tau}{\to} S)$ and the image of its class under $\tx{Ad}(g)^{-1}$ is equal to 
    $\tx{inv}((S',\gamma'),\tilde\delta) \in H^1(F,S' \stackrel{1-a}{\to} S')$.
  \end{enumerate}
\end{pro}

\begin{proof}
(1) follows from the existence of abstract norms (\cite[Lemma 3.3.4(1)]{KalLLCD}) and part (4), proved independently below.

(2) Let $C$ be a $\bar F$-Borel subgroup containing $S$ and normalized by $S^\natural$. Then $(S,C)$ is normalized by $\tilde\delta_g$, hence the $\bar F$-Borel pair $(S',C'):=\tx{Ad}(g)^{-1}(S,C)$ is normalized by $1 \rtimes a$, i.e. it is $a$-stable.

By assumption $S$ is stable under $\tx{Ad}(z_\sigma)\circ\sigma$, hence $S'=\tx{Ad}(g)^{-1}S$ is stable under $\tx{Ad}(g)^{-1}\circ\tx{Ad}(z_\sigma)\circ\sigma\circ\tx{Ad}(g)=\tx{Ad}(g^{-1}z_\sigma\sigma(g))\circ\sigma$, but $g^{-1}z_\sigma\sigma(g)=g^{-1}t_\sigma g \in S'(\bar F)$ due to 1(a), hence $S'$ is stable under $\sigma$, i.e. it is defined over $F$. 

(3) Again by 1(a), the isomorphism $\tx{Ad}(g) : S' \to S$ is defined over $F$. By 1(b) it identifies the action of $S^\natural$ on $S$ with the action of $a$ on $S'$.

(4) We first check that if $g$ satisfies (a) and (b) of (1), then so does $sgx$. For (a) we see $z_\sigma\sigma(sgx)(sgx)^{-1}=z_\sigma\sigma(s)\sigma(g)\sigma(x)x^{-1}g^{-1}s^{-1}$. Since $\tx{Ad}(z_\sigma)\circ\sigma$ normalizes $S$ we have $z_\sigma\sigma(s)=s'z_\sigma$ for $s'=z_\sigma\sigma(s)z_\sigma^{-1} \in S(\bar F)$, and hence $z_\sigma\sigma(sgx)(sgx)^{-1}=s'(z_\sigma\sigma(g)g^{-1})(g\sigma(x)x^{-1}g^{-1})s^{-1}$, where all four terms in the product lie in $S(\bar F)$. For (b) we note that $sgx(1\rtimes a)(sgx)^{-1}=\tx{Ad}(s)(g (1 \rtimes a)g^{-1}) \in S^\natural$.

Now consider two elements $g_1,g_2$ satisfying (a) and (b). Choose a Borel $\bar F$-subgroup $C$ containing $S$ and stable under conjugation by $S^\natural$. According to (2), $(S_i',C_i')=\tx{Ad}(g_i)^{-1}(S,C)$ are $a$-stable $\bar F$-Borel pairs. Therefore they are conjugate under $G^{a,\circ}(\bar F)$. Taking $x \in G^{a,\circ}(\bar F)$ with $x^{-1}(S_1',C_1')x=(S_2',C_2')$ we see that $g_2$ and $g_1x$ both conjugate $(S_2',C_2')$ to $(S,C)$. We conclude that $s:=g_2x^{-1}g_1^{-1}$ normalizes $(S,C)$ and therefore lies in $S(\bar F)$. Thus $g_2=sg_1x$. To show that $g_1\sigma(x)x^{-1}g_1^{-1} \in S(\bar F)$, we apply condition (a) to $g_2=sg_1x$ and see
\begin{eqnarray*}
  S(\bar F)&\ni& z_\sigma\sigma(sg_1x)(sg_1x)^{-1}\\
  &=&z_\sigma\sigma(s)\sigma(g_1)\sigma(x)x^{-1}g_1^{-1}s^{-1}\\
  &=&s'(z_\sigma\sigma(g_1)g_1^{-1})(g_1\sigma(x)x^{-1}g_1^{-1})s^{-1}, 
\end{eqnarray*}
where $s'=z_\sigma\sigma(s)z_\sigma^{-1}$. Since $s \in S(\bar F)$ and $\tx{Ad}(z_\sigma)\circ\sigma$ normalizes $S$ we get $s' \in S(\bar F)$. Furthermore $z_\sigma\sigma(g_1)g_1^{-1} \in S(\bar F)$ by condition (a) for $g_1$. It follows that $g_1\sigma(x)x^{-1}g_1^{-1} \in S(\bar F)$, as desired.

(5) The element $\tilde\delta_g$ is replaced by $s \tilde\delta_g s^{-1}$ if we replace $g$ by $sgx$. Since $S^\natural_\tau$ is the set of $S$-conjugacy classes in $S^\natural$, the image of $\tilde\delta_g$ in $S^\natural_\tau$ is unchanged. Moreover, $z_\sigma\sigma(\tilde\delta_g )z_\sigma^{-1} = z_\sigma\sigma(g)(1 \rtimes a)\sigma(g^{-1})z_\sigma^{-1} = z_\sigma\sigma(g)g^{-1} \tilde\delta g \sigma(g^{-1})z_\sigma^{-1}$, which by 1(a) is $S$-conjugate to $\tilde\delta$.

(6) Let $\tilde\delta' = g^{-1}\tilde\delta g = \delta' \rtimes a$. If $C$ is any $\bar F$-Borel subgroup stable normalized by $S^\natural$, then the $\bar F$-Borel pair $(S,C)$ is normalized by $\tilde\delta$, hence $(S',C')=\tx{Ad}(g)^{-1}(S,C)$ is normalized by $\tilde\delta'$. But it also also $a$-stable by (2), so it is also normalized by $\delta'$, showing $\delta' \in S'(\bar F)$. The image $\gamma' \in S'_a(\bar F)$ of $\delta'$ equals $\tx{Ad}(g)^{-1}\gamma$ and thus lies in $S'_a(F)$. We conclude that $(S',\gamma')$ is an abstract norm of $\tilde\delta$.

Conversely, let $(S',\gamma')$ be an abstract norm for $\tilde\delta$ and let $g \in G(\bar F)$ be such that $\tilde\delta'=g^{-1}\tilde\delta g \in S'(\bar F) \rtimes a$ and if we write $\tilde\delta'=\delta' \rtimes a$ then the image of $\delta'$ in $S'_a(\bar F)$ equals $\gamma'$. Then $\tilde\delta_g = g \cdot (1 \rtimes a) \cdot g^{-1} = g(\delta')^{-1}g^{-1} \cdot g\tilde\delta'g^{-1} = g(\delta')^{-1}g^{-1} \cdot \tilde\delta$. Since $\tx{Ad}(g)S'=S$ we have $g(\delta')^{-1}g^{-1} \in S(\bar F)$, hence $\tilde\delta_g \in S^\natural$, confirming 1(b). Moreover, the isomorphism $\tx{Ad}(g) : S' \to S$ is defined over $F$, hence $g^{-1}z_\sigma\sigma(g)z_\sigma^{-1} \in S'(\bar F)$, hence $z_\sigma\sigma(g)g^{-1} \in S(\bar F)$, confirming 1(a).

(7) The image of $(t_\sigma^{-1},\delta)$ under $\tx{Ad}(g)^{-1}$ equals $((g^{-1}z_\sigma\sigma(g))^{-1},\delta')$, which is the definition of the invariant given in \cite[\S4.3,\S4.4]{KalLLCD}.
\end{proof}

\begin{rem} \label{rem:norm}
  \begin{enumerate}
    \item Condition (b) of (1) can be rewritten in a manner that is analogous to condition (a), as follows. For one, hence every, $\tilde\delta \in S^\natural$, if we write $\tilde\delta = \delta \rtimes a$, then $\delta a(g)g^{-1} \in S(\bar F)$.
    \item Part (1,b) of the proposition can be reinterpreted to say that $S^\natural = S \cdot \tilde\delta_g$, while part (5) says that $S_\tau^\natural = S_\tau \cdot a_\tau$. The latter decomposition is more robust, since $a_\tau$ depends only on the Cartan subspace $S^\natural$ (in particular, it does not depend on the choice of $g$) and is an $F$-point. In other words, the $S_\tau$-torsor $S_\tau^\natural$ and the $S_\tau(F)$-torsor $S_\tau^\natural(F)$ are both canonically trivialized. Note that $\tilde\delta_g$ is a quasi-central element in the sense of Digne--Michel, see \cite[\S3.8]{HL17}, and it has finite order equal to that of $a$.
    \item The assignment $\tilde\delta \mapsto \gamma := \tilde\delta_\tau \cdot a_\tau^{-1}$ used in part (6) of the proposition provides a continuous map $S^\natural(F) \to S_\tau(F)$ that is equivariant for left-multiplication by $S(F)$.
    \item The class of $(t_\sigma^{-1},\delta) \in Z^1(F,S \stackrel{1-\tau}{\to} S)$ in part (7) of the proposition does depend on the choice of $g$. If we replace $g$ by $sg$ with $s \in S(\bar F)$, the class is unchanged, but if we replace $g$ by $gx$ with $x \in G^{a,\circ}(\bar F)$ satisfying $g\sigma(x)x^{-1}g^{-1} \in S(\bar F)$, then the class multiplies by the image of the class of $g\sigma(x)x^{-1}g^{-1} \in H^1(F,S^\tau)$ under the map $H^1(F,S^\tau) \to H^1(F,S \stackrel{1-\tau}{\to} S)$ induced by the inclusion $S^\tau \to S = [S \to 1] \to [S \to S]$. This reflects the fact that now the abstract norm $(S',\gamma')$ that was associated to $g$ has been replaced by its $x^{-1}$-conjugate, which in turn is associated to $gx$.
  \end{enumerate}
\end{rem}

\begin{dfn} \label{dfn:relel}
  Let $S^H \subset H$ be a maximal torus and let $S^\natural \subset [G \rtimes a]_z$ be a Cartan subspace. Let $S$ be the double centralizer of $S^\natural$. An \emph{admissible relation} between $S^H$ and $S^\natural$ is an isomorphism $S^H \to S_\tau$ obtained as the composition of $\tx{Ad}(g) : S'_a \to S_\tau$, with $g \in G(\bar F)$ as in Proposition \ref{pro:norm}(1), and an admissible isomorphism  $S^H \to S'_a$. We say that $S^H$ and $S^\natural$ are \emph{related}, if there exists an admissible relation between them.
\end{dfn}

Note that $\gamma^H \in S^H(F)$ is related to $\tilde\delta \in S^\natural_\tx{sr}(F)$ if and only if there exists an admissible relation $S^H \to S_\tau$ sending $\gamma^H$ to the image $\gamma \in S_\tau(F)$ of $\tilde\delta$ under the map $S^\natural(F) \to S_\tau(F)$ of Remark \ref{rem:norm}(3).

\subsection{The space of related strongly regular conjugacy classes}

The concept of related elements of Definition \ref{dfn:relel} provides a correspondence between the set of stable classes of strongly $G$-regular semi-simple elements of $H(F)$ and the set of $G_z(F)$-classes of strongly regular semi-simple elements of $[G \rtimes a]_z(F)$, where an element of $H(F)$ is called $G$-strongly regular if it is a norm of a strongly regular element of $[G \rtimes a]_z(F)$. We can embody this correspondence in a topological space $\mf{X}$ defined as the fiber product
\[ \xymatrix{
  &\mf{X}\ar[dl]\ar[dr]&\\
  H_{G\tx{-sr}}(F)/\tx{st-conj}\ar[dr]&&[G \rtimes a]_z^\tx{sr}(F)/G_z(F)\tx{-conj}\ar[dl]\\
  &[G \rtimes a]_z^\tx{sr}(F)/\tx{st-conj}&
}
  \]
Explicitly, the underlying set of $\mf{X}$ consists of pairs $([[\gamma]],[\tilde\delta])$, where $[[\gamma]]$ is a stable class of strongly $G$-regular semi-simple elements of $H(F)$, and $[\tilde\delta]$ is a $G_z(F)$-conjugacy class of strongly regular semi-simple elements of $[G \rtimes a]_z(F)$, and $\gamma$ is a norm of $\tilde\delta$. Taking into account the discussion of \S\ref{sub:tc} and \S\ref{sub:norm}, we obtain the following picture
\begin{equation} \label{eq:diagx}
  \scalebox{0.85}{\xymatrix{
  &\mf{X}\ar[dl]\ar[dr]&\\
  H_{G\tx{-sr}}(F)/\tx{st-conj}\ar@{=}[d]\ar[dr]&&[G \rtimes a]_z^\tx{sr}(F)/G_z(F)\tx{-conj}\ar@{=}[d]\ar[dl]\\
  \coprod\limits_{S^H} S^H_{G\tx{-sr}}(F)/\Omega(S^H,H)(F)\ar[dr]&[G \rtimes a]_z^\tx{sr}(F)/\tx{st-conj}\ar@{^(->}[d]&\coprod\limits_{S^\natural} S^\natural_\tx{sr}(F)_\tau/\Omega_F(S,G_z)^\tau\ar[dl]\\
  &\coprod\limits_{S^\natural} S^\natural_{\tau,\tx{sr}}(F)/\Omega(S,G_z)(F)^\tau
  }}\end{equation}
where on the left $S^H$ runs over the set of stable classes of those maximal tori of $H$ which are related to Cartan subspaces of $[G \rtimes a]_z$, on the right $S^\natural$ runs over the set of $G_z(F)$-conjugacy classes of Cartan subspaces of $[G \rtimes a]_z$, and on the bottom $S^\natural$ runs over the stable classes of Cartan subspaces of $[G \rtimes a]_z$; the inclusion on the bottom comes from Lemma \ref{lem:snt}(3).

\begin{lem} \label{lem:x-top}
  All maps in Diagram \eqref{eq:diagx} are covering maps with open images and finite fibers. 
\end{lem}
\begin{proof}
  Since the diagram is Cartesian, the properties of the top two diagonal maps mirror the properties of the bottom two diagonal maps. Consider first the right bottom map. There are only finitely many $G_z(F)$-conjugacy classes of Cartan subspaces of $[G \rtimes a]_z$ in a fixed stable class. For any Cartan subspace $S^\natural$ we claim that the map $S^\natural_\tx{sr}(F)_\tau/\Omega_F(S,G_z)^\tau \to S^\natural_{\tau,\tx{sr}}(F)/\Omega(S,G_z)^\tau(F)$ is a covering map with open image and finite fibers. It is enough to show that the same is true for the map $S^\natural(F)_\tau \to S^\natural_\tau(F)$, but this was discussed in Lemma \ref{lem:snt}.
  
  Consider next the bottom left map. Since there are only finitely many stable classes of maximal tori $S^H$ in $H$, we may focus on a single one. It is related to a unique stable class of Cartan subspaces $S^\natural$ of $[G \rtimes a]_z$. To show that $S^H_{G\tx{-sr}}(F)/\Omega(S^H,H)(F) \to S^\natural_{\tau,\tx{sr}}(F)/\Omega(S,G_z)(F)^\tau$ is a finite cover with open image, it is enough to show the same statement for the map $S^H(F) \to S^\natural_\tau(F)$. Recall that this map is obtained as the composition of an admissible relation $S^H(F) \to S_\tau(F)$ and the map $S_\tau(F) \to S^\natural_\tau(F)$ given by the canonical base point $a_\tau \in S^\natural_\tau(F)$. It follows that $S^H(F) \to S^\natural_\tau(F)$ is a homeomorphism.
\end{proof}

\begin{rem}
  Let $f : X \to Y$ be covering map of topological spaces with finite fibers and open image. Given a Borel measure on $Y$ we can form the ``pull-back measure'' on $X$ to be the unique Borel measure whose quotient by the counting measures on the fibers equals the given measure on $Y$. In other words, $f : X \to \C$ is integrable if and only if $\bar f(y)=\sum_{x \mapsto y}f(x)$ is integrable on $Y$, and then $\int_X f(x)dx = \int_Y \bar f(y) dy$.
\end{rem}

We endow $\mf{X}$ with the pull-back under the map $\mf{X} \to [G \rtimes a]_z^\tx{sr}(F)/\tx{st-conj}$, which is a covering with open image and finite fibers according to Lemma \ref{lem:x-top}, of the measure on $[G \rtimes a]_z^\tx{sr}(F)/\tx{st-conj}$ discussed in \S\ref{sub:tc}. 

For the next lemma we will synchronize the measures on the spaces in Diagram \eqref{eq:diagx} as follows. In each of the three disjoint unions in this diagram, we place a measure by taking the sum of the quotient measures obtained from the counting measure on the relevant Weyl group, and a Haar measure on the relevant torus (transported to the appropriate torsor). What then remains is to synchronize the measures on the tori $S^H(F)$, $S_\tau(F)$, and $S(F)_\tau$, whenever $S^H$ and $S^\natural$ are related.

For this we note that specifying a Haar measure on any one of the three groups $S(F)_\tau$, $S(F)^\tau=S^\tau(F)$, and $S_\tau(F)$, specifies a measure on the other two. This follows from the two exact sequences
\[ 0 \to \tx{Lie}(S^\tau) \to \tx{Lie}(S) \to \tx{Lie}(S) \to \tx{Lie}(S_\tau) \to 0\]
and 
\[ 1 \to S^\tau(F) \to S(F) \to S(F) \to S(F)_\tau \to 1, \]
both induced by the map $1-\tau : S \to S$ sending $t$ to $t\tau(t)^{-1}$. The first sequence shows that the spaces of top forms on $S^\tau(F)$ and $S_\tau(F)$ are canonically isomorphic. Since specifying a Haar measure on either group is equivalent to specifying a pair $(\omega,dx)$ of a top form $\omega$ and a Haar measure $dx$ on $F$, subject to the equivalence $(a\omega,|a|_F^{-1}dx)$ for $a \in F^\times$, we see that the spaces of Haar measures of $S^\tau(F)$ and $S_\tau(F)$ are canonically isomorphic. Via the second sequence we can start with a Haar measure $ds$ on $S^\tau(F)$, choose a Haar measure $dt$ on $S(F)$, transport the quotient measure $dt/ds$ via $1-\tau$ to the image $(1-\tau)S(F)=S(F)/S^\tau(F)$, and again transport the quotient measure $dt/(dt/ds)$ via the isomorphism $S(F)/(1-\tau)S(F) \to S(F)_\tau$ to a measure on $S(F)_\tau$, which is independent of the choice of $dt$.

With this observation we can now synchronize the measures in Diagram \eqref{eq:diagx} as follows. Whenever $S^H$ and $S^\natural$ are related, use the admissible isomorphism $S^H(F) \to S^\tau(F)$ to transport a chosen measure on one of these groups to the other, and then use the previous paragraph to endow the groups $S_\tau(F)$ and $S(F)_\tau$ with the corresponding measures. 

Although not necessary at the moment, we can also synchronize the measures among the different groups $S^H(F)$, using the fact that all maximal tori in $H$ are conjugate under $H_\tx{sc}(\bar F)$, so the spaces of top forms on $\tx{Lie}(S^H)(\bar F)$ are canonically identified.

\begin{lem} \label{lem:x-maes}
  Assume that $S^H$ and $S^\natural$ are related, and that the measures are synchronized as just described. Via the maps $S^H \to S_\tau$ and $S(F)_\tau \to S_\tau(F)$, the measures on $S^H(F)/\Omega(S^H,H)(F)$ and $S^\natural(F)_\tau/\Omega_F(S,G_z)^\tau$ are pulled back from the measure on  $S^\natural_\tau(F)/\Omega(S,G_z)(F)^\tau$.
\end{lem}
\begin{proof}
  Consider first the map $S^H_{G\tx{-sr}}(F)/\Omega(S^H,H)(F) \to S^\natural_{\tau,\tx{sr}}(F)/\Omega(S,G_z)(F)^\tau$. The space $S^\natural_\tau(F)$ is a torsor under $S_\tau(F)$ and the map $S^H(F) \to S^\natural_\tau(F)$ is given by an admissible isomorphism $S^H \to S_\tau$ and the identification $S_\tau \to S^\natural_\tau$ via the canonical element $a_\tau \in S^\natural_\tau(F)$. Therefore we must show that the measures on $S^H(F)/\Omega(S^H,H)(F)$ and $S_\tau(F)/\Omega(S,G_z)(F)^\tau$ are compatible under pull-back. But both measures pull-back to the given measures on $S^H(F)$ and $S_\tau(F)$, which are compatible under the admissible isomorphism $S^H \to S_\tau$.

  Consider now the map $S^\natural(F)_\tau/\Omega_F(S,G_z)^\tau \to S^\natural_\tau(F)/\Omega(S,G_z)(F)^\tau$. Since this map is equivariant under $S(F)_\tau$, and $S^\natural(F)_\tau$ is a torsor for $S(F)_\tau$, we reduce to considering $S(F)_\tau/\Omega_F(S,G_z)^\tau \to S_\tau(F)/\Omega(S,G_z)(F)^\tau$. The measures on both spaces pull back to the chosen Haar measures on $S(F)_\tau$ and $S_\tau(F)$, respectively. It is thus enough to show that the measure $d\bar s$ on $S(F)_\tau$ is the pull-back of the measure $d \tilde s$ on $S_\tau(F)$, provided they are synchronized as described before the statement of the lemma.

  For this, let $f : S(F) \to \C$ be locally constant and compactly supported and let $\bar f(\bar s)=\int_{S(F)/S^\tau(F)}f((1-\tau)t \cdot \bar s)dt/ds$, where $dt$ and $ds$ are fixed Haar measures on $S(F)$ and $S^\tau(F)$, respectively.  The measure $d\bar s$ on $S(F)_\tau$ is defined so that 
  \[ \int_{S(F)_\tau}\bar f(\bar s)d\bar s = \int_{S(F)}f(t)dt. \]

  We shall now use the fact that, if $1 \to A \to B \to C \to 1$ is an exact sequence of tori, and the measures on these tori arise from a fixed Haar measure on $F$ and volume forms $\omega_A,\omega_B,\omega_C$ that satisfy $\omega_B = \omega_A \otimes \omega_C$ under the canonical isomorphism $\Wedge^\bullet\tx{Lie}(B)^* = \Wedge^\bullet\tx{Lie}(A)^* \otimes \Wedge^\bullet\tx{Lie}(C)^*$, then $dc|_{B(F)/A(F)}=db/da$.

  We apply this to the bottom horizontal exact sequence of the diagram in the proof of Lemma \ref{lem:x-top} and see that 
  \[ \int_{S(F)}f(t)dt = \int_{S_\tau(F)} \int_{[S/S^\tau](F)} f((1-\tau)a \cdot \tilde s) dad\tilde s, \]
  where $d\tilde s$ is the measure on $S_\tau(F)$ and $da$ is the measure on $[S/S^\tau](F)$, all coming from the top forms on $S^\tau$, $S$, and $S_\tau$, synchronized via the exact sequence of their Lie algebras coming from the map $1- \tau$. 

  We apply the same principle to the vertical exact sequence in the same diagram and obtain 
  \begin{eqnarray*}
    \int_{[S/S^\tau](F)} f((1-\tau)a \cdot \tilde s)da&=& \sum_{\gamma \in \frac{[S/S^\tau](F)}{[S(F)/S^\tau(F)]}} \int_{S(F)/S^\tau(F)} f((1-\tau)\bar t \cdot \gamma \cdot \tilde s)d\bar t\\
    &=&\sum_{\gamma \in \frac{[S/S^\tau](F)}{[S(F)/S^\tau(F)]}} \bar f(\gamma \cdot \tilde s),   
  \end{eqnarray*}
  where $d\bar t=dt/ds$. Plugging this into the previous formula we obtain
  \[ \int_{S(F)}f(t)dt = \int_{S_\tau(F)} \sum_{\gamma \in [S/S^\tau](F)/[S(F)/S^\tau(F)]} \bar f(\gamma \cdot \tilde s) d\tilde s. \]
  Recalling that, by the snake lemma, $[S/S^\tau](F)/[S(F)/S^\tau(F)]$ is exactly the kernel of $S(F)_\tau \to S_\tau(F)$, we see that 
  \[ \int_{S_\tau(F)} \sum_{\gamma \in [S/S^\tau](F)/[S(F)/S^\tau(F)]} \bar f(\gamma \cdot \tilde s) d\tilde s \]
  is the integral of $\bar f$ with respect to the pull-back to $S(F)_\tau$ of the measure $d\tilde s$ of $S_\tau(F)$.
\end{proof}

  \begin{cor} \label{cor:x-meas}
      Let $f : \mf{X} \to \C$ be integrable. Then $\int_\mf{X} f(x)dx$ is equal to either of the expressions
      \[ \sum_{S^H} \int_{S^H_\tx{sr}(F)/\Omega(S^H,H)(F)} \sum_{[\tilde\delta]} f([[\gamma]],[\tilde\delta])d\gamma,\]
      and
      \[ \sum_{S^\natural} \int_{S^\natural_\tx{sr}(F)_\tau/\Omega_F(S,G_z)^\tau} \sum_{[[\gamma]]} f([[\gamma]],[\tilde\delta])d\tilde\delta. \]
      In the first expression, $S^H$ runs over the set of stable classes of maximal tori of $H$ and $[\tilde\delta]$ runs over the set of $G_z(F)$-conjugacy classes of elements of $[G \rtimes a]_z^\tx{sr}(F)$ that are related to $\gamma$. 

      In the second expression, $S^\natural$ runs over the set of $G_z(F)$-conjugacy classes of Cartan subspaces of $[G \rtimes a]_z$ and $[[\gamma]]$ runs over the set of stable conjugacy classes of norms of $\tilde\delta$.
  \end{cor}
  \begin{proof}
    The topology of $\mf{X}$ is defined so that for each $S^H \subset H$ and $S^\natural \subset [G \rtimes a]_z$, the preimage $\mf{X}'$ of $S^H_\tx{sr}(F)/\Omega(S^H,H)(F) \times S^\natural_\tx{sr}(F)_\tau/\Omega_F(S,G_z)^\tau$ in $\mf{X}$ is open and closed. Therefore 
    $\int_\mf{X}f(x)dx = \sum_{\mf{X}'}\int_{\mf{X}'}f(x)dx$. The space $\mf{X}'$ sits in the Cartesian diagram of topological spaces 
    \[ 
  \xymatrix{
  &\mf{X}'\ar[dl]\ar[dr]&\\
  S^H_{G\tx{-sr}}(F)/\Omega(S^H,H)(F)\ar[dr]&&S^\natural_\tx{sr}(F)_\tau/\Omega_F(S,G_z)^\tau\ar[dl]\\
  &S^\natural_{\tau,\tx{sr}}(F)/\Omega(S,G_z)(F)^\tau
  }\] 
    and it is enough to show that $\int_{\mf{X}'}f(x)dx$ is equal to 
    \[ \int_{S^H_\tx{sr}(F)/\Omega(S^H,H)(F)} \sum_{[\tilde\delta]} f([[\gamma]],[\tilde\delta])d\gamma, \]
    where the sum now runs over those elements of $S^\natural_\tx{sr}(F)_\tau/\Omega_F(S,G_z)^\tau$ which have $\gamma$ as a norm, and is furthermore equal to 
    \[ \int_{S^\natural_\tx{sr}(F)_\tau/\Omega_F(S,G_z)^\tau} \sum_{[[\gamma]]} f([[\gamma]],[\tilde\delta])d\tilde\delta \]
    where the sum now runs over those elements of $S^H_\tx{sr}(F)/\Omega(S^H,H)(F)$ which  are norms of $\tilde\delta$. 

    These two expressions are the integrals of $f$ with respect to the two measures on $\mf{X}'$ obtained by pulling back the measures on $S^H_{G\tx{-sr}}(F)/\Omega(S^H,H)(F)$ and $S^\natural_\tx{sr}(F)_\tau/\Omega_F(S,G_z)^\tau$, respectively, so the claim reduces to Lemma \ref{lem:x-maes}.
  \end{proof}

\subsection{Twisted endoscopic lifting of distributions}

We continue with an arbitrary local field $F$ of characteristic zero. The goal of this subsection is to prove a twisted version of \cite[Lemma 2.11.3]{DPR}.

Let $\Delta_{KS}''$ denote the transfer factor $\Delta_{KS}$ of \cite[\S4.11,\S5.5]{KalLLCD} but with the contribution of $\Delta_{IV}$ inverted.

\begin{lem} \label{lem:twistlift}
  Let $d_\phi$ be a genuine stable distribution on $H_1(F)$ that is represented by the locally integrable genuine function $\phi$ on $H_1(F)$ and a fixed Haar measure $d\gamma$ on $H(F)$. The distribution $\tilde f \mapsto d_\phi(f_1)$ on $[G \rtimes a]_z(F)$ is represented by the locally integrable function 
  \[ \tilde\delta \mapsto \sum_\gamma \Delta_{KS}''(\gamma_1,\tilde\delta)\phi(\gamma_1) \]
  and a Haar measure $d\tilde\delta$ on $[G \rtimes a]_z(F)$. Here $d\gamma$ and $d\tilde\delta$ are used in the transfer map $\tilde f \mapsto f_1$, the sum runs over the set of strongly regular semi-simple elements $\gamma \in H(F)$ up to stable conjugacy, and $\gamma_1 \in H_1(F)$ is an arbitrary lift of $\gamma$.
\end{lem}

\begin{cor} \label{cor:main1}
  Equation \eqref{eq:main}, and hence Theorem \ref{thm:main}, are equivalent to the identity 
  \begin{equation} \label{eq:main1}
    e(G_z)\sum_{\substack{\pi \in \Pi_\phi(G_z)\\\pi\circ a\cong\pi}}  \tx{tr}(\pi \boxtimes \rho^\vee)^\tx{can}(\tilde\delta,\tilde s^{-1}) = \sum_{\gamma \in H(\R)^\tx{sr}/\tx{st}} \Delta_\tx{KS}''(\gamma_1,\tilde\delta) S\Theta_{\varphi_1}(\gamma_1)
  \end{equation}
  of functions of $\tilde\delta \in [G \rtimes a]_z(F)_\tx{sr}$.
\end{cor}

\begin{proof}[Proof of Lemma \ref{lem:twistlift}]
  The proof of this lemma follows the same argument as \cite[Lemma 2.11.3]{DPR}; unlike loc. cit. we are not normalizing orbital integrals and characters by the square root of the Weyl discriminant, so the formulas here a slightly different. By assumption we have $d_\phi(f_1)=\int_{H(F)}\phi(\gamma_1)f_1(\gamma_1)d\gamma$. According to the stable Weyl integration formula \cite[Remark 2.10.3]{DPR} we have
  \[ d_\phi(f_1)=\sum_{S^H} \int_{S^H_\tx{sr}(F)/\Omega(S^H,H)(F)} |D_{H/S^H}(\gamma)|SO_\gamma(\phi \cdot f_1)d\gamma, \]
  where $D_{H/S^H}(\gamma)=\det(1-\tx{Ad}(\gamma)|\tx{Lie}(H)/\tx{Lie}(S^H))$.

  The stable invariance of $\phi$ implies $SO_\gamma(\phi \cdot f_1)=\phi(\gamma_1)SO_{\gamma_1}(f_1)$ for an arbitrary lift $\gamma_1$ of $\gamma$. Since $f_1$ and $\tilde f$ are matching, we have 
  \[ SO_{\gamma_1}(f_1)=\sum_{\tilde\delta} \Delta_{KS}(\gamma_1,\tilde\delta)O_{\tilde\delta}(\tilde f), \]
  where $\tilde\delta$ runs over the set of $G_z(F)$-conjugacy classes of strongly regular elements of $[G \rtimes a]_z(F)$. In this identity we have chosen arbitrarily Haar measures $dh$ and $dg$ on $H(F)$ and $G_z(F)$ and have synchronized the measures on the centralizers of $\gamma$ and $\tilde\delta$ as explained before Lemma \ref{lem:x-maes}.
  
  Putting things  together we arrive at 
  \[ d_\phi(f_1)=\sum_{S^H} \int_{S^H_\tx{sr}(F)/\Omega(S^H,H)(F)} |D_{H/S^H}(\gamma)| \phi(\gamma_1)\sum_{\tilde\delta} \Delta_{KS}(\gamma_1,\tilde\delta)O_{\tilde\delta}(\tilde f)d\gamma. \]
  The next step is to reindex the sums and integral. The support of the integrand is the image of the map $\mf{X} \to H(F)_\tx{sr}/\tx{st}$ of Diagram \eqref{eq:diagx}, so Corollary \ref{cor:x-meas} allows us to rewrite this as 
  \[ d_\phi(f_1)=\sum_{S^\natural} \int_{S^\natural(F)_\tau/\Omega_F(S,G_z)^\tau} O_{\tilde\delta}(\tilde f)\sum_\gamma \phi(\gamma_1)\Delta_{KS}(\gamma_1,\tilde\delta)|D_{H/S^H}(\gamma)|d\tilde\delta,\]
  where now $S^\natural$ runs over the set of $G_z(F)$-conjugacy classes of Cartan subspaces of $[G \rtimes a]_z$ and $\gamma$ runs over the set of stable classes of strongly regular semi-simple elements of $H(F)$ that are related to the integration variable $\tilde\delta$ (we can drop this condition, because for those $\gamma$ who fail it the factor $\Delta_{KS}(\gamma_1,\tilde\delta)$ will be zero). Applying the twisted Weyl integration formula \cite[Corollaire 7.3.7]{HL17} we arrive at 
  \[ d_\phi(f_1)= \int_{[G \rtimes a]_z(F)} \tilde f(\tilde\delta) \sum_\gamma \phi(\gamma_1)\Delta_{KS}(\gamma_1,\tilde\delta)\frac{|D_{H/S^H}(\gamma)|}{|D_{[G \rtimes a]_z/S}(\tilde\delta)|}d\tilde\delta,  \]
  where
  \[ D_{[G \rtimes a]_z/S}(\tilde\delta) = \det(1-\tx{Ad}(\tilde\delta)|\tx{Lie}(G)/\tx{Lie}(S)).\]
\end{proof}

\subsection{Elliptic norms over $\R$} 

We now consider return to $F=\R$ and consider further the special case of elliptic elements.

\begin{lem} \label{lem:ellnorm}
  Let $\tilde\delta \in [G \rtimes a]_z(\R)$ be strongly regular and elliptic, in the sense of Definition \ref{dfn:ell}. 
  \begin{enumerate}
    \item There exists a norm $\gamma \in H(\R)$ of $\tilde\delta$.
    \item Any norm $\gamma \in H(\R)$ of $\tilde\delta$ is elliptic, i.e. $S^H/Z(H)$ is an anisotropic $\R$-torus.
    \item If $Z(G)(\R)$ is compact, $\tilde\delta'$ can be chosen to lie in $S'(\R) \rtimes a$.
  \end{enumerate}
\end{lem}
\begin{proof}
  (1) Since $\varphi_1$ is a discrete parameter for $H_1$ the group $H_1$ has elliptic maximal tori, and then so does $H$. Let $S^H \subset H$ be one such. According to \cite[Lemma 3.3.B]{KS99} it transfers to the quasi-split group $G$, i.e. there exists an $a$-admissible maximal torus $S_0 \subset G$ and an admissible isomorphism $S^H \to (S_0)_a$. 
  
  In \cite[\S5.1]{KS99} it is shown that the natural inclusion $Z_G \to S_0$ induces an injection $(Z_G)_a \to (S_0)_a$, whose composition with the inverse of the above isomorphism takes image in $Z_H$.

  Since $H$ is an elliptic endoscopic group of $G$, the map $(Z_G)_a \to Z_H$ is an isogeny. The ellipticity of $S^H$ now implies that $(S_0/Z_G)_a$ is anisotropic, from which we infer that $(S_0/Z_G)^a$ is anisotropic, and hence that $S_0^{a,\circ}$ is an elliptic maximal torus of $G^{a,\circ}=G^1$. 
  
  On the other hand, there exists an abstract norm $(S',\gamma')$ of $\tilde\delta$. Since the double centralizer $S$ of $\tilde\delta$ is elliptic due to Lemma \ref{lem:ell} and there is an $\R$-isomorphism $S' \to S$, we conclude that $S'$ is elliptic. Corollary \ref{cor:aeonj} now shows that $S_0$ and $S'$ are conjugate under $G^1(\R)$. We have thus exhibited an admissible isomorphism $S^H \to (S')_a$. Pulling back $\gamma'$ under this isomorphism we obtain the desired norm $\gamma \in H(\R)$ of $\gamma'$.

  (2) Lemma \ref{lem:ell} implies that $S$ is elliptic. Since $\tx{Ad}(g) : S' \to S$ is an isomorphism over $\R$ we see that $S'$ is also elliptic. Hence $(S'/Z(G))_a$ is anisotropic, and therefore so is $S^H/Z(H)$. 
  
  (3) Since $Z(G)(\R)$ is compact, $S'$ is anisotropic, and so is $S'_a$. Lemma \ref{lem:elem} implies that $S'(\R) \to S'_a(\R)$ is surjective. Since the preimage of $\gamma$ under the map $S'(\C) \to S'_a(\C)$ is the set of $S'(\C)$-conjugates of $\tilde\delta'$, we see that we may modify $g$ by right multiplication by $S'(\C)$ to ensure $\tilde\delta' \in S'(\R) \rtimes a$. 
\end{proof}

\subsection{Reduction to elliptic elements}

 Corollary \ref{cor:main1} reinterpreted Equation \eqref{eq:main}, which is the identity claimed in Theorem \ref{thm:main}, as Equation \eqref{eq:main1}. Our next task is to show that it is enough to prove this identity for elliptic elements $\tilde\delta$.

 Recalling Equation \eqref{eq:main}, we
define the distribution $\Theta$ by
\begin{equation}
  \label{thetadist}
\Theta(\tilde{f})  = e(G_z)\sum_{\substack{\pi \in \Pi_\varphi(G_z)\\\pi\circ a\cong\pi}} 
\mathrm{tr}(\pi \boxtimes \rho^{\vee})^{\mathrm{can}} (\tilde{f},
\tilde{s}^{-1}) - S\Theta_{\varphi_{1}}(f_{1})
\end{equation}
for test functions $\tilde{f} \in
C_{c}^{\infty}( [G \rtimes a ]_{z}(\mathbb{R}))$.  The goal of
this section  is to
prove that $\Theta = 0$ under the assumption that
$\Theta(\tilde{f}) = 0$ for all test functions $\tilde{f}$ supported on the
strongly regular elliptic subset of $[G \rtimes a
]_{z}(\mathbb{R)}$. 
Our strategy towards this goal is to show that the 
restriction of $\Theta$ to any connected component of $ [G \rtimes
  a ]_{z}(\mathbb{R})$ (in the manifold topology) satisfies this
property (\emph{cf.}
\cite[\S 6.4]{Mezo13}). Since int this section we are working with the single automorphism $a \in A^{[z]}$, we set for convenience $A=\<a\>$.

We wish to relate the connected components of $[G \rtimes
  a ]_{z}(\mathbb{R})$ to the representations in the definition of $\Theta$.
To this end, fix a summand $\pi_{0} \in \Pi_{\varphi}(G_{z})$ of
(\ref{thetadist}).   The condition $\pi_{0} \circ a \cong \pi_{0}$
means that   $\pi_{0}$ extends to an irreducible representation
$\tilde{\pi}_{0}$ of $\tilde{G}_{z}(\mathbb{R})$.  According to Lemma \ref{lem:duflo1}(3), there exists regular semi-simple elliptic $f_{0} \in
\mathfrak{g}_{z}^{*}(\mathbb{R})$ and  $\tilde{\tau}_{0} \in X(f_{0})$
such that $\tilde{\pi}_{0} = \pi_{f_{0},\tilde{\tau}_{0}}$.  In the
following lemma, we regard the character $\Theta_{\tilde{\pi}_{0}}$ as
a function defined on the regular subset of
$\tilde{G}_{z}(\mathbb{R})$ (\cite[\S 3.1]{Bouaziz87}).
\begin{pro}
\label{goodcomp}
Suppose $\tilde{\delta} \in [G \rtimes a]_{z}(\mathbb{R})$ is a
regular semi-simple element such that
$\Theta_{\tilde{\pi}_{0}}(\tilde{\delta}) \neq 0$.  Then $\tilde{\delta}$
belongs to a connected component of $[G \rtimes a]_{z}(\mathbb{R})$ which
contains a strongly regular elliptic element fixing a
$G_{z}(\mathbb{R})$-conjugate of $f_{0}$.
\end{pro}
\begin{proof}
Let $\tilde{S}$ be the centralizer of $f_{0}$ in $\tilde{G}$ and $S =
\tilde{S} \cap G$. The group $S(\mathbb{R})$ is an elliptic maximal
torus in $G_{z}(\mathbb{R})$ and $f_{0} \in
\mathfrak{s}^{*}(\mathbb{R})$.  The assumption
  $\pi_0\circ a \cong \pi_0$, which we recall means that there exists $\tilde\delta_0 \in \tilde G_z(\R)$ mapping to $a$ such that $\pi_0\circ\tx{Ad}(\tilde\delta_0) \cong \pi_0$, implies that we can choose this $\tilde\delta_0$ such that $\tx{Ad}(\tilde\delta_0)f_{0} = f_0$. Let $B_0$ be the $\C$-Borel subgroup corresponding to the positive system $R_{f_0}^+$. Since 
$\tilde{\delta}_{0}$ fixes $f_{0}$, the pair is preserved by
$\tilde{\delta}_{0}$, implying that $\tilde\delta_0$ is
semi-simple. Thus, $\tx{Ad}(\tilde{\delta}_0)$ acts semisimply on the
derived subgroup of $G_{z}(\mathbb{R})$. An application of 
\cite[Lemma 1.4.1]{Bouaziz87} to the setting of the derived subgroup
produces  a regular element of
$S(\R)\tilde\delta_{0} \subset 
\tilde{S}(\R)$. Replacing $\tilde\delta_0$ by such an element we
obtain a 
regular semi-simple element $\tilde\delta_0 \in \tilde S(\R)$.  

Recall from Section \ref{sub:duflo} that
$$\tilde{\pi}_{0} = \pi_{f_{0}, \tilde{\tau}_{0}} =
\mathrm{ind}_{G_{z}(\mathbb{R})^{\circ} \,
  \tilde{S}(\mathbb{R})}^{\tilde{G}_{z}(\mathbb{R})}\, (\tilde{\tau}_{0}
\otimes \tilde{\pi}_{f}).$$
The finite quotient
$$\tilde{G}_{z}(\mathbb{R})/ ( G_{z}(\mathbb{R})^{\circ} \,
  \tilde{S}(\mathbb{R}))  \cong G_{z}(\mathbb{R})/ \left(
  G_{z}(\mathbb{R})^{\circ} \,   S(\mathbb{R}) \right) \cong 
G_{z}(\mathbb{R})/ \left( G_{z}(\mathbb{R})^{\circ} \,
Z_{G_{z}}(\mathbb{R}) \right) $$ 
(\cite[page 134]{Lan89}) has representatives in
$G_{z}(\mathbb{R})$, and  the induced character 
$\Theta_{\tilde{\pi}_{0}}$ has support in the conjugates of
$G_{z}(\mathbb{R})^{\circ} \, \tilde{S}(\mathbb{R})$ under these representatives.
In particular, since $\Theta_{\tilde{\pi}_{0}}(\tilde{\delta}) \neq
0$, there exists $h \in G_{z}(\mathbb{R})$ such that 
$$\tilde{\delta} \in h  ( G_{z}(\mathbb{R})^{\circ} \,
  \tilde{S}(\mathbb{R})) h^{-1}.$$
Now,
\begin{align*}
G_{z}(\mathbb{R})^{\circ}\, \tilde{S}(\mathbb{R})
&= G_{z}(\mathbb{R})^{\circ}\, S(\mathbb{R}) \langle \tilde{\delta}_{0}
\rangle \\
&= G_{z}(\mathbb{R})^{\circ} \, Z_{G_{z}}(\mathbb{R})  \langle
\tilde{\delta}_{0}  \rangle\\
& = \coprod_{j} G_{z}(\mathbb{R})^{\circ} \, g_{j}  \langle
\tilde{\delta}_{0}  \rangle
\end{align*}
where $g_{j}$ runs over a set of representatives of the finite group
$Z_{G_{z}}(\mathbb{R})/ (Z_{G_{z}}(\mathbb{R}) \cap G_z(\R)^{\circ})$.
It follows that
$$\tilde{\delta}  \in h (  G_{z}(\mathbb{R})^{\circ} \, g_{j} 
\tilde{\delta}_{0} )h^{-1}$$
for some $g_{j}$.  This set is a connected component of $[G \rtimes
  a]_{z}(\mathbb{R}) = G_{z}(\mathbb{R}) \tilde{\delta}_{0}$
which contains the strongly regular elliptic element $h g_{j}
\tilde{\delta}_{0} h^{-1}$ fixing  $\mathrm{Ad}^{*}(h)f_{0}$, as desired.
\end{proof}

\begin{lem}
  \label{srell}
  Any regular semi-simple elliptic element of $[G \rtimes
    a]_{z}(\mathbb{R})$ fixes a stable conjugate
  of $f_{0}$.
  \end{lem}
  \begin{proof}
    Let $\tilde\delta$ be a regular semi-simple elliptic element of $[G \rtimes a]_z(\R)$. Let $S_1$ be the connected double centralizer of $\tilde\delta$. Then $\tilde\delta$ normalizes $S_1$, as well as a $\C$-Borel subgroup $B_1$ containing $S_1$.

Define $\tilde S_1$ as the product in $\tilde G_z$ of $S_1$ and the cyclic group generated by $\tilde\delta$. Then $\tilde S_1$ is a subgroup of $\tilde G_z$ (since $\tilde\delta$ normalizes $S_1$), defined over $\R$, and whose intersection with $G_z$ equals $S_1$ (by definition $S_1$ lies in that intersection, so enough to see the converse. But note that all elements of $S_1$, as well as the element $\tilde\delta$, normalize the pair $(S_1,B_1)$, so the group generated by them does, so its intersection with $G_z$ must lie in $S_1$). By construction $\tilde\delta$ lies in $\tilde S_1(\R)$. 

Let $\tilde S_0$ be the centralizer of $f_0$ in $\tilde G_z$. The intersections of $\tilde S_0$ and $\tilde S_1$ with $G_z$ are elliptic maximal tori, hence $G_z(\R)$-conjugate, and we conjugate $f_0$ by $G_z(\R)$ to arrange that they are equal, and call them $S$. We conjugate $f_0$ by an element of the $\C$-Weyl group of $S$ to ensure that the Borel $\C$-subgroup determined by the regular element $f_0$ is equal to $B_1$, and hence stabilized by $\tilde\delta$. Now let $\tilde\delta_0$ be any element of $\tilde S_0(\R)$ that projects onto a; it exists by our assumption. Then the difference $\tilde\delta \cdot (\tilde\delta_0)^{-1}$ is an element of $G_z(\R)$ that normalizes $S$ and $B_1$, hence lies in $S(\R)$. But this means that $\tilde\delta$ fixes $f_0$. 
  \end{proof}

  \begin{cor}
    \label{goodcomp2}
    Suppose  $\tilde{\delta}$  is a
    regular semi-simple element which belongs to a connected component of
    $[G \rtimes a]_{z}(\mathbb{R})$ which does not contain  regular
    elliptic elements.  Then
    $\Theta_{\tilde{\pi}}(\tilde{\delta}) = 0$ for all extensions
    $\tilde{\pi}$ of representations $\pi \in \Pi_{\varphi}(G_{z})$
    occurring in (\ref{thetadist}).
    \end{cor}
    \begin{proof}
    Let $if \in \mf{s}^*(\R)$ be the infinitesimal character of $\pi$. Proposition \ref{goodcomp} applied to $\pi$ in place of $\pi_0$ shows that if $\Theta_{\tilde\pi}(\tilde\delta) \neq 0$, then $\tilde\delta$ lies in a connected component of $[G \rtimes a]_z(\R)$ that contains an elliptic element $\tilde\delta_0$ which fixes a $G_z(\R)$-conjugate of $if$. As $\pi_{0}$ and $\pi$ belong to the same L-packet, the local Langlands correspondence dictates that $if$ is stably conjugate to $if_0$, therefore $\tilde\delta_0$ fixes a stable conjugate of $if_0$.
\end{proof}

Let us return to the strategy of showing that $\Theta$ vanishes by
considering its restriction to connected components of $[G \rtimes
a]_{z}(\mathbb{R})$.  Every connected component of $[G \rtimes
a]_{z}(\mathbb{R})$ may be written as $G_{z}(\mathbb{R})^{\circ}
\tilde{\delta}$ where $\tilde{\delta}$ is a strongly regular semi-simple element
in $[G \rtimes a]_{z}(\mathbb{R})$. In view of Corollary \ref{goodcomp2} we separate the connected
components into two kinds: 
those which contain strongly regular elliptic elements, and those which do not. 
Until Theorem \ref{thetazero}, we fix a connected component
$G_{z}(\mathbb{R})^{\circ} \tilde{\delta}$ containing strongly regular
elliptic elements. Until then, Lemma \ref{srell} allows us to
assume that $\tilde{\delta}$ is strongly 
regular elliptic and that there exists $f \in
\mathfrak{s}^{*}(\mathbb{R})$ in the stable conjugacy class of
$f_{0}$ such that 
\begin{equation}
\label{assumef}
\mathrm{Ad}^{*}(\tilde{\delta})f = f.
\end{equation}
The element $f$ is an elliptic and regular
infinitesimal character of some representation in
$\Pi_{\varphi}(G_{z})$.  As in Section \ref{sec:sqlp}, we denote the centralizer of 
$f$ in $G$ by $S$.
The maximal torus $S$ is preserved by
$\mathrm{Ad}(\tilde{\delta})$, and is equal to the double centralizer of  $\tilde{\delta}$.
By Lemma \ref{lem:ellnorm}, the element $\tilde{\delta}$ has a norm in
$H(\mathbb{R})$. We describe the norm in the notation of \S\ref{sub:norm}, i.e.  $g^{-1} \tilde{\delta} g = \tilde{\delta}' = \delta'
\rtimes a$, $\delta' \in S'$, $\gamma' = \delta' (1-a)S'$, $\gamma'
= \eta(\gamma_{1})$.

The proof of $\Theta$ vanishing on $G_z(\R)^\circ\tilde\delta_0$ relies on a twisted
version of Harish-Chandra's Uniqueness Theorem, due to Renard
(\cite[Theorem 15.1]{Ren97}). We will show that the restriction
$$\Theta_{|} = \Theta_{|G_{z}(\mathbb{R})^{\circ} \,\tilde{\delta}}$$
of $\Theta$ to $ G_{z}(\mathbb{R})^{\circ} \,\tilde{\delta}$ satisfies the
hypotheses of this theorem. Work of Bouaziz will
reduce this task to proving that $\Theta_{|}$ is an \emph{invariant
eigendistribution}.

By ``invariant'' we mean invariant under
conjugation by $G_{z}(\mathbb{R})^{\circ}$.  This is to say
$$\Theta_{|}(\tilde{f}^{y}) = \Theta_{|}(\tilde{f}), \quad \tilde{f} \in
C_{c}^{\infty}(G_{z}(\mathbb{R})^{\circ} \, \tilde{\delta}).$$
for all $y \in G_{z}(\mathbb{R})^{\circ}$, where
$$\tilde{f}^{y}(x) = \tilde{f}(y^{-1}xy).$$
The component $G_{z}(\mathbb{R})^{\circ} \,
\tilde{\delta}$ is preserved under conjugation by
$G_{z}(\mathbb{R})^{\circ}$ since $\tilde{\delta}$ normalizes $G_{z}(\mathbb{R})^{\circ}$.  
The invariance of $\Theta_{|}$ is easily verified by looking back to
the terms in 
(\ref{thetadist}).  The portion of $\mathrm{tr}(\pi \boxtimes
\rho^{\vee})^{\mathrm{can}} (\tilde{f}, \tilde{s}^{-1})$ which depends
on $\tilde{f}$ is the character value
$\mathrm{tr}(\tilde{\pi}(\tilde{f}))$, where 
$\tilde{\pi}$ is an extension of $\pi$ to $\tilde{G}_{z}(\mathbb{R})$.  Since characters are invariant, so too is
$\mathrm{tr}(\pi \boxtimes \rho^{\vee})^{\mathrm{can}}$.  The
invariance of the term $S\Theta_{\varphi_{1}}(f_{1})$ stems from the
defining property of $f_{1} \in C_{c}^{\infty}(H(\mathbb{R}))$.
Indeed, the function $f_{1}$ is defined by matching 
orbital integrals of $\tilde{f}$ (\cite[Lemma 4.12.1]{KalLLCD}).
Since orbital integrals are invariant, the function $f_{1}$ also matches
the orbital integrals of $\tilde{f}^{y}$, and  
$S\Theta_{\varphi_{1}}(f_{1})$ is insensitive to conjugation by $y$.

Having dispensed with the invariance of $\Theta_{|}$, we move to showing that
$\Theta_{|}$ is an eigendistribution.  Let $\mathcal{Z}(\mathfrak{g})$ be the
centre of the universal enveloping algebra of $\mathfrak{g}$.  The
algebra $\mathcal{Z}(\mathfrak{g})$ acts on
$C_{c}^{\infty}(G_{z}(\mathbb{R})^{\circ} \, \tilde{\delta})$ by
differential operators.  The algebra acts on $\Theta_{|}$ through the
transpose, that is
$$X\Theta_{|}(\tilde{f}) = \Theta_{|}(X^{\mathrm{tr}} \tilde{f}),
\quad \tilde{f} \in C_{c}^{\infty}(G_{z}(\mathbb{R})^{\circ} \,
\tilde{\delta}), \, X \in \mathcal{Z}(\mathfrak{g}),$$
where $X^{\mathrm{tr}}$ is defined \cite[X.4]{KnappSS}.
One says that $\Theta_{|}$ is an
eigendistribution with infinitesimal character $\Lambda \in
\mathfrak{s}^{*}$ if
$$X \Theta_{|} = \Lambda(\beta_{\mathfrak{g}}(X))\ \Theta_{|}, \quad X \in
\mathcal{Z}(\mathfrak{g}),$$ 
where $\beta_{\mathfrak{g}}: \mathcal{Z}(\mathfrak{g}) \rightarrow
\mathcal{S}(\mathfrak{s})^{\Omega(S,G)}$ is the Harish-Chandra
isomorphism onto the Weyl group invariants of the  symmetric algebra
of $\mathfrak{s}$.
We wish to prove that $\Theta_{|}$ is an eigendistribution with
infinitesimal character $if \in \mathfrak{s}^{*}$.  Weaving these
definitions into (\ref{thetadist}), it becomes apparent
that we need an understanding of the function in $C_{c}^{\infty}(H(\mathbb{R}))$
that matches $X\tilde{f}$ for $X \in \mathcal{Z}(\mathfrak{g})$.  As it
turns out, this matching function is of the form $X_{1}f_{1}$, where
$f_{1} \in C_{c}^{\infty}(H(\mathbb{R}))$ matches $\tilde{f}$, and
$X_{1} \in \mathcal{Z}(\mathfrak{h})$.

The definition of the map $X \mapsto X_{1}$ requires a digression.
It is a composition of six  algebra
homomorphisms,
\begin{align}
  \label{zedone}
  \nonumber
\mathcal{Z}(\mathfrak{g})
& \stackrel{\beta_{\mathfrak{g}}}{\longrightarrow}
\mathcal{S}(\mathfrak{s})^{\Omega(S,G)}
\stackrel{\mathrm{Ad}(g)^{-1}}{\longrightarrow}
\mathcal{S}(\mathfrak{s}')^{\Omega(S',G)} 
\stackrel{\phi}{\longrightarrow}
\mathcal{S}((\mathfrak{s}')^{a})^{\Omega(S',G)^{a}}\\
&\stackrel{\eta^{-1}}{\longrightarrow}
\mathcal{S}(\mathfrak{s}^{H})^{\Omega(S^{H},H)}
\stackrel{I_{-\mu^{*}}}{\longrightarrow}
\mathcal{S}(\mathfrak{s}^{H})^{\Omega(S^{H},H)} 
\stackrel{\beta_{\mathfrak{h}}^{-1}}{\longrightarrow}
\mathcal{Z}(\mathfrak{h}).
\end{align}
The first and last maps in this sequence are Harish-Chandra
isomorphisms.  The second map is induced from the $\mathbb{R}$-isomorphism
$\mathrm{Ad}(g)^{-1}: S \rightarrow S'$ of maximal tori
in the definition of the norm of $\tilde{\delta}$ in \S\ref{sub:norm}.  The
third map $\phi$ is induced from the projections of $\mathfrak{s}'$
onto $(\mathfrak{s}')^{a}$ and is defined in \cite[\S
  2.4]{Bouaziz87}.  The fourth map is induced from the admissible
isomorphism $\eta$ of \S\ref{sub:admiso} and the isomorphisms
$$\Omega(S',G)^{a} \cong \Omega(\hat{S}', \hat{G})^{a} \cong
\Omega(((\hat{S}')^{a})^{\circ},  (\hat{G}^{a})^{\circ}) \cong
\Omega(\hat{S}^{H},\hat{H}) \cong \Omega(S^{H}, H)$$
(\cite[\emph{p.} 14]{KS99}).

The fifth and final map
$I_{-\mu^{*}}$ is more elaborate.  Here, $\mu^{*} \in
(\mathfrak{s}^{H})^{*}$ is defined from the endoscopic embeddings $\xi:
{^L}H \rightarrow {^L}G$ and $\xi_{1}: {^L}G^{1} \rightarrow {^L}G$ of
\cite[\S 4.1]{KalLLCD} in the following manner. According to
\cite[Proposition 3.1.2]{She81}, the restrictions of these two
embeddings to $\mathbb{C}^{\times} \subset W_{\mathbb{R}}$ have the
form
$$\xi(1 \times z) = t(z) \times z, \mbox{ and } \xi_{1}(1 \times z) =
t_{1}(z) \times z, \quad z \in \mathbb{C}^{\times},$$ 
where $t$ and $t_{1}$ are characters of $\mathbb{C}^{\times}$ taking
values in the centres of $\hat{H}$ and $\hat{G}^{1}$ respectively.  
The centrality of $t_{1}(z)$ is equivalent to  $\alpha(t_{1}(z)) = 1$ for
all roots $R(\hat{S}_{a}', \hat{G}^{1})$.  This root system
is identical to the system $R_{\mathrm{res}}(\hat{S}',\hat{G})$ of
reduced roots.  Under the 
isomorphism $\eta$, the root system $R(\hat{S}^{H},\hat{H})$ may be
regarded as a subsystem of $R_{\mathrm{res}}(\hat{S}', \hat{G})$
\cite[Theorem 1.1.A]{KS99}.  In this way,
$\alpha(t_{1}(z)) = 1$ for all $\alpha \in R(\hat{S}^{H},\hat{H})$,
and this implies that $t_{1}(z) \in Z(\hat{H})$.  Since $Z(\hat{H})
\subset \hat{S}^{H}$ we may view the product $t t_{1}^{-1}$ as a
character with values in $\hat{S}^{H}$.  The linear form $\mu^{*} \in
X_{*}(\hat{S}^{H}) \otimes \mathbb{C} = 
X^{*}(S^{H}) \otimes \mathbb{C} = (\mathfrak{s}^{H})^{*}$ is
defined by the equation
\begin{equation}
\label{mustar}
tt_{1}^{-1}(z) = z^{\mu^{*}} \bar{z}^{\sigma_{S_{H}} \mu^{*}},
\quad z \in \mathbb{C}^{\times},
\end{equation}
where $\sigma_{S^{H}}$ is the action of $\Gamma$ on $\hat{S}^{H}$ in ${^L}S^{H}$.
As $tt_{1}^{-1}$ takes values in the centre of $\hat{H}$, the
functional $\mu^{*}$ is fixed by $\Omega(S^{H},H)$.  

We are now able to define $I_{-\mu^{*}}$. The linear map
 $\mathfrak{s}^{H} \rightarrow
\mathcal{S}(\mathfrak{s}^{H})$ defined by
$$Y \mapsto Y - \mu^{*}(Y)$$
extends to an algebra automorphism of $\mathcal{S}(\mathfrak{s}^{H})$
(\cite[Proposition 3.1]{KnappSS}). As  $\mu^{*}$ is
fixed by $\Omega(S^{H},H)$,  the automorphism of
$\mathcal{S}(\mathfrak{s}^{H})$ restricts to an automorphism
$I_{-\mu^{*}}$ of $\mathcal{S}(\mathfrak{s}^{H})^{\Omega(S^{H},H)}$.
This completes the definition of the homomorphism $X \mapsto X_{1}$
given in (\ref{zedone}). 

Although (\ref{zedone}) is  defined relative to the tori $S$, $S'$ and
$S^{H}$ determined by the strongly regular elliptic element
$\tilde{\delta}$ and its norm, any other strongly regular semi-simple (not
necessarily elliptic) element of
$[G \rtimes a]_{z}(\mathbb{R})$ and norm leads to the same map.
This can be seen from the facts that maximal tori are conjugate over the algebraic closure,  that
the central algebras $\mathcal{Z}(\mathfrak{g})$ and 
$\mathcal{Z}(\mathfrak{h})$ are pointwise fixed under conjugation over the algebraic closure, and that
the character $tt_{1}^{-1}$ takes values in any maximal torus of $\hat{H}$.  

The purpose of the map $X \mapsto X_{1}$ is to show
that $X\tilde{f}$ matches $X_{1}f_{1}$ when $\tilde{f}$ matches
$f_{1}$.  The geometric transfer identity \cite[Lemma
  4.12.1]{KalLLCD} which defines the matching  involves the
transfer factors $\Delta_{KS}$ of \cite[\S 4.10]{KalLLCD}. Therefore
one would expect the map $X \mapsto X_{1}$ to interact with these transfer
factors.  One such interaction will come from the linear functional
$\mu^{*}$ appearing in (\ref{zedone}).  This linear functional
appears in the  term $\Delta_{III}^{\mathrm{new}}$ of $\Delta_{KS}$ in the following
way.  The number $\Delta_{III}^{\mathrm{new}}(\gamma_{1},
\tilde{\delta})$ is defined 
through a pairing, and one of the terms in this pairing is a cocycle
$a_{S'} \in Z^{1}(W_{\mathbb{R}}, \hat{S}')$  (\cite[Corollary
  4.10.4]{KalLLCD}).  The the cocyle $a_{S'}$ 
determines an L-parameter $W_{\mathbb{R}} \rightarrow {^L}S'$ of the
elliptic tours $S'$.  Through the local Langlands correspondence the
L-parameter determines a character of $S'(\mathbb{R})$. The
value of the inverse of this character on an element $x^{*} \in
S'(\mathbb{R})$ is denoted by  $\langle x^{*}, [a_{S'}]^{-1} \rangle$
in \cite[Remark 4.10.1]{KalLLCD}.  The following lemma makes the
connection between this character and $\mu^{*}$.  Before stating the
lemma, we remark that $\Delta_{III}^{\mathrm{new}}$ depends on a
choice of $\chi$-data, but $\Delta_{KS}$ does not.  We may therefore
use based $\chi$-data (see the note after the statement of Lemma \ref{lem:semi2} for definition) in the proof of the lemma.
\begin{lem}
\label{neartildedelta}
Suppose $\tilde{\delta} \in [G \rtimes a]_{z}(\mathbb{R})$ is strongly
regular elliptic with norm $\gamma_{1} \in
H(\mathbb{R})$.  
Suppose further that $x \in S^{\tilde{\delta}}(\mathbb{R})$ such that
$x\tilde{\delta}$ is strongly regular elliptic.  Then
\begin{enumerate}
\item The strongly regular element $x \tilde{\delta}$ has norm
  $\gamma_{x} \gamma' \in 
  S'_{a}(\mathbb{R})$, where $\gamma_{x} \in S'_{a}(\mathbb{R})$ is the
  coset of $g^{-1}xg \in S'(\mathbb{R})$. 

\item Let $x^{*} = g^{-1} x g \in S'(\mathbb{R})$.  Then
  $$\Delta_{III}^{\mathrm{new}}(\eta^{-1}(\gamma_{x}) \gamma_{1}, 
  x\tilde{\delta}) = \langle x^{*}, [a_{S'}]^{-1} \rangle
  \  \Delta_{III}^{\mathrm{new}}(\gamma_{1},   \tilde{\delta}).$$
\item For based $\chi$-data of $R_{\mathrm{res}}(S', G)$, the
  differential of the character 
  $\langle \cdot, [a_{S'}]^{-1} 
  \rangle$ on $S'(\mathbb{R})$ is equal to $\rho_{G_\mathrm{res}} -
  \rho_{H}   - \mu^{*}$,  
where $\rho_{H}$ is the half-sum of the positive roots of $R(S^{H},
\hat{H})$ (determined by $\eta$) and $\rho_{G_{\mathrm{res}}}$ is the
half-sum of the positive roots of $R_{\mathrm{res}}(\hat{S}',\hat{G}) =
R(\hat{S}'_{a}, \hat{G}^{1})$. 
\end{enumerate}
\end{lem}
\begin{proof}
For the first assertion we note that
$$g^{-1} (x\tilde{\delta}) g= (g^{-1}xg)
( g^{-1} \tilde{\delta}g) =  (g^{-1}xg) (\delta'
\rtimes a).$$
The element $\delta'$ lies $S'$, and $g^{-1}xg$
lies in $S'(\mathbb{R})$, since $\mathrm{Ad}(g)^{-1}: S \rightarrow S'$ is an
$\mathbb{R}$-isomorphism.  The first assertion now follows from the
definition of norm (\cite[Definition 3.3.2]{KalLLCD}).

For the second assertion, we begin with the pairing
\begin{equation}
\label{del3def}
\Delta_{III}^{\mathrm{new}}(\eta^{-1}(\gamma_{x}) \gamma_{1},
x\tilde{\delta}) = \langle \mathrm{inv}(\gamma_{x} \gamma', (z,x
\delta)), A_{0}  
\rangle
\end{equation}
displayed in \cite[(4.4)]{KalLLCD}.  The term
$\mathrm{inv}(\gamma_{x} \gamma', (z, x\delta)) \in H^{1}(\Gamma, S'
\stackrel{1-a}{\rightarrow} S')$ on the right is the class of a
$1$-cocycle $((z')^{-1},
x' \delta')$, where $(z^{*}, x' \delta') = g^{-1} \cdot
(z,x\delta)$ (\cite[Lemma 4.3.1, (4.1)]{KalLLCD}).   To say that $((z')^{-1},
x'\delta')$ is a $1$-cocyle is to say that
$$(1-a)(z'(\sigma)^{-1}) = (x' \delta')^{-1} \ \sigma( x'
\delta' ), \quad \sigma \in \Gamma.$$
From the first assertion we see that $x' \in S'(\mathbb{R})$ and so
$$(1-a)(\mathbf{1}(\sigma)^{-1}) = 1 = (x')^{-1}
\ \sigma(x'),$$
where $\mathbf{1}$ is the trivial cocycle in $Z^{1}(\Gamma, S')$.
In consequence, $(\mathbf{1},x')$ is a $1$-cocycle and 
$$\mathrm{inv}(\gamma_{x} \gamma_{1}, ((z')^{-1},
x' \delta') ) = \mathrm{inv}(\gamma_{x},(\mathbf{1},x'))  \
\mathrm{inv}(\gamma_{1},((z')^{-1}, \delta'))$$
in  $H^{1}(\Gamma, S'\stackrel{1-a}{\rightarrow} S')$.
Looking back to (\ref{del3def}), we find that
\begin{align*}
\Delta_{III}^{\mathrm{new}}(\eta^{-1}(\gamma_{x}) \gamma_{1},
x\tilde{\delta}) &= \langle
\mathrm{inv}(\gamma_{x},(\mathbf{1},x^{*})) ,A_{0}\rangle \  \langle
\mathrm{inv}(\gamma_{1},((z^{*})^{-1}, \delta^{*})) ,A_{0}\rangle\\
&=  \langle
\mathrm{inv}(\gamma_{x},(\mathbf{1},x^{*})) ,A_{0}\rangle
\ \Delta_{III}^{\mathrm{new}}(\gamma_{1}, \tilde{\delta})
\end{align*}
The second assertion follows from \cite[Remarks 4.3.2 and
  4.10.1]{KalLLCD}. 

For the final assertion, we begin by recalling the definition of the cocycle
$a_{S'} \in Z^{1}(\Gamma, S')$ as seen in the proof of \cite[Corollary
  4.10.4]{KalLLCD}.  In fact we only need the definition of
restriction of $a_{S'}$ to $\mathbb{C}^{\times} \subset
W_{\mathbb{R}}$. The cocyle $a_{S'}$ arises through the comparison 
of two L-embeddings.   Fixing based $\chi$-data for the system of reduced
roots $R_{\mathrm{res}}(S', G)$ (\cite[\S 1.3]{KS99}) and
regarding $R(\hat{S}^{H}, \hat{H})$ as a subset of
$R_{\mathrm{res}}(\hat{S}', \hat{G})$, one obtains two L-embeddings
$$\xi^{1}_{S'}: {^L}S'_{a} \rightarrow ((\hat{G}^{a})^{\circ} \rtimes
  W_{\mathbb{R}}) = {^L}G^{1}$$
and
$$\xi^{H} :{^L}S^{H} \rightarrow {^L}H.$$
The first L-embedding may be composed with L-embedding $\xi_{G^{1}}: {^L}G^{1}
\rightarrow {^L}G$ to obtain
\begin{equation}
    \label{xione}
\xi_{G^{1}} \circ \xi^{1}_{S'}: {^L}S'_{a} \rightarrow  {^L}G.
\end{equation}
The second L-embedding may be composed with $\xi: {^L}H \rightarrow
{^L}G$ and ${^L}\eta: {^L} S'_{a} \rightarrow {^L}S^{H}$ to obtain
another L-embedding
\begin{equation}
\label{xiH}
    \xi \circ \xi^{H} \circ {^L}\eta:  {^L}S'_{a} \rightarrow  {^L}G.
\end{equation}
By \cite[Lemma 4.10.2]{KalLLCD}, embedding (\ref{xione}) extends to an
L-embedding $\xi_{S'}: {^L}S' \rightarrow {^L}G$, and embedding (\ref{xiH})
extends to an L-embedding $\xi'_{S'}:  {^L}S' \rightarrow {^L}G$
The restriction of the cocycle $a_{S'}$ to $\mathbb{C}^{\times}\subset
W_{\mathbb{R}}$ is given by the equation
\begin{equation}
\label{cocycledef}
\xi_{S'}^{-1} \circ \xi'_{S'}(1\rtimes z) = a_{S'}(z) \rtimes z, \quad
    z \in \mathbb{C}^{\times}.
\end{equation}
As $\mathbb{C}^{\times} \subset W_{\mathbb{R}} \subset {^L}S'_{a}$,
the element $a_{S'}(z)$ may be computed directly from this equation 
using the L-embeddings (\ref{xione}) and (\ref{xiH}), as follows
\begin{align*}
& \xi'_{S'}(1\rtimes z) = \xi_{S'}( a_{S'}(z) \rtimes z)\\
& \Rightarrow \xi'_{S'}(1\rtimes z) = a_{S'}(z) \ \xi_{S'}(1 \rtimes z)\\
& \Rightarrow \xi \circ \xi_{S^{H}} \circ {^L}\eta (1 \rtimes z) =
  a_{S'}(z) \ \xi_{G^{1}} \circ \xi_{S'}^{1}(1 \rtimes z)\\
& \Rightarrow \xi( (z/\bar{z})^{\rho_{H}} \rtimes z)
  =  a_{S'}(z) \  \xi_{G^{1}} (  (z / \bar{z})^{\rho_{G_{\mathrm{res}}}} \rtimes z)\\
& \Rightarrow a_{S'}(z) = (z/\bar{z})^{ \rho_{H}-  \rho_{G_{\mathrm{res}}}} \ \xi(1
  \rtimes z) \  \xi_{G^{1}}(1 \rtimes z)^{-1} \\
& \Rightarrow a_{S'}(z) =  (z/\bar{z})^{ \rho_{H} - \rho_{G_{\mathrm{res}}}
   } \ tt_{1}^{-1}(z).
\end{align*}
We have used the assumption of based $\chi$-data and $S'$ being
elliptic in the fourth implication (\cite[\S 10]{SheTE1}).  The
final assertion now follows from (\ref{mustar}) and the local
Langlands correspondence for real tori applied to $a_{S'}$.  
\end{proof}

\begin{pro}
\label{zmatch}
Suppose $\tilde{f} \in C_{c}^{\infty}([G \rtimes a]_{z}(\mathbb{R}))$
and $f_{1} \in C_{c}^{\infty}(H(\mathbb{R}))$ are functions which
satisfy the transfer identity of \cite[Lemma 4.12.1]{KalLLCD}.  Then
for any $X \in \mathcal{Z}(\mathfrak{g})$ the functions  $X\tilde{f}
\in C_{c}^{\infty}([G \rtimes a]_{z}(\mathbb{R}))$ 
and $X_{1}f_{1} \in C_{c}^{\infty}(H(\mathbb{R}))$ also satisfy the
transfer identity.
\end{pro}
\begin{proof}
We are to prove
\begin{equation}
\label{geomatch}
SO_{\gamma_{1}}(X_{1}f_{1}) = \sum_{\dot{\delta}}
  \Delta_{KS}(\gamma_{1}, \dot{\delta}) \,
  O_{\dot{\delta}}(X \tilde{f}),
\end{equation}
where $X \in \mathcal{Z}(\mathfrak{g})$, $\gamma_{1} \in
H(\mathbb{R})$ is the norm of  some strongly 
regular semi-simple element  $\dot{\delta}_{0} \in [G \rtimes
  a]_{z}(\mathbb{R})$, and the 
sum on the right runs over the $\tilde{G}_{z}(\mathbb{R})$-conjugacy
classes in the stable conjugacy class of $\dot{\delta}_{0}$.
According to \cite[Lemma 1.6.3]{Bouaziz87}, there exist
strongly regular semi-simple elements $\dot{\delta}_{1}, \ldots ,
\dot{\delta}_{\ell} \in [G \rtimes a]_{z}(\mathbb{R})$ such that any
strongly regular semi-simple element is $G_{z}(\mathbb{R})$\-conjugate
to an element in $G_{z}^{\dot{\delta}_{j}}(\mathbb{R})^{\circ} \,
\dot{\delta}_{j}$ for some $1 \leq j \leq \ell$.  It therefore
suffices to prove (\ref{geomatch}) for $\dot{\delta} \in
G_{z}^{\dot{\delta}_{j}}(\mathbb{R})^{\circ} \, \dot{\delta}_{j}$.
Without loss of generality  $\dot{\delta}_{1} = \tilde{\delta}$ 
and we begin by proving this identity for $\dot{\delta}_{0} = 
x\tilde{\delta}$ and $\gamma_{1} = \eta^{-1}(\gamma_{x}) \gamma$ as in Lemma
\ref{neartildedelta}.  The only special property of $\tilde{\delta}$
in this setting is that it is elliptic.  Most of what follows does
not use the property of being elliptic.  We will highlight the
places where the elliptic property is used.  Once the proof for
$\dot{\delta}_{1} = \tilde{\delta}$ is complete, we will describe
the proof for arbitrary $\dot{\delta}_{j}$ by indicating how the
elliptic proof is to be modified.

According to \cite[Lemma 23]{Mezo13}, 
$$O_{x\tilde{\delta}} (X \tilde{f}) = D_{Ga}(x\delta)^{-1} \left(
\mathrm{Ad}(g) \circ \phi \circ \mathrm{Ad}(g)^{-1} \circ
\beta_{\mathfrak{g}}(X) ( D_{Ga}(x\delta) \,
O_{x\tilde{\delta}}(\tilde{f}) )
\right)$$
where $D_{Ga}$ is defined in \cite[\S 4.5]{KS99}.
This a twisted version of a result of Harish-Chandra which tells us
that 
$$O_{\eta^{-1}(\gamma_{x}) \gamma}(X_{1} f_{1}) =
  D_{H}(\eta^{-1}(\gamma_{x}) \gamma)^{-1} \left( \beta_{\mathfrak{h}}(X_{1})
  D_{H}(\eta^{-1}(\gamma_{x}) \gamma) O_{\gamma_{1}}(f_{1}) \right).$$
It is straightforward to show that the values $D_{H}(\eta^{-1}(\gamma_{x})
\gamma)$ and $D_{Ga}(x \delta)$ are not changed if their arguments are
replaced by stably conjugate elements.  
Therefore we may substitute the
previous two equations into (\ref{geomatch}) to see that we are
required to prove

\resizebox{0.95\textwidth}{!}{\begin{minipage}{\linewidth}
\begin{align*}
  &  \beta_{\mathfrak{h}}(X_{1}) \left(
  D_{H}(\eta^{-1}(\gamma_{x}) \gamma)\ SO_{\eta^{-1}(\gamma_{x})
    \gamma}(f_{1}) \right)=\\
&\frac{D_{H}(\eta^{-1}(\gamma_{x}) \gamma)}{D_{Ga}(x \delta)}
\sum_{\dot{\delta}} 
  \Delta_{KS}(\eta^{-1}(\gamma_{x}) \gamma, \dot{\delta})  \left(
\mathrm{Ad}(g) \circ \phi \circ \mathrm{Ad}(g)^{-1} \circ
\beta_{\mathfrak{g}}(X) ( D_{Ga}(x \delta) \,
O_{\dot{\delta}}(\tilde{f}) )\right)
\end{align*}
\end{minipage}
}

If we substitute the right-hand side of (\ref{geomatch}) into the
left-hand side of this equation and rearrange terms slightly, we see
that the equation to prove is

\resizebox{0.87\textwidth}{!}{\begin{minipage}{\linewidth}
\begin{align}
\label{geomatchsub}
&   \sum_{\dot{\delta}} \beta_{\mathfrak{h}}(X_{1}) \left(
  \frac{D_{H}(\eta^{-1}(\gamma_{x}) \gamma)}{D_{Ga}(x \delta)} 
    \Delta_{KS}( \eta^{-1}(\gamma_{x}) \gamma, \dot{\delta}) \,
  D_{Ga}(x \delta) \ O_{\dot{\delta}}( \tilde{f}) \right)\\
& =  \sum_{\dot{\delta}}
\frac{D_{H}(\eta^{-1}(\gamma_{x}) \gamma)}{D_{Ga}(x \delta)}
  \Delta_{KS}(\eta^{-1}(\gamma_{x}) \gamma, \dot{\delta})  \left(
\mathrm{Ad}(g) \circ \phi \circ \mathrm{Ad}(g)^{-1} \circ
\beta_{\mathfrak{g}}(X) ( D_{Ga}(x \delta) \,
O_{\dot{\delta}}(\tilde{f}) \right).
\end{align}
\end{minipage}
}

This equation holds if
\begin{align*}
\beta_{\mathfrak{h}}(X_{1}) \
&  \frac{D_{H}(\eta^{-1}(\gamma_{x}) \gamma)}{D_{Ga}(x \delta)} 
    \Delta_{KS}( \eta^{-1}(\gamma_{x}) \gamma, \dot{\delta})\\
& =   \frac{D_{H}(\eta^{-1}(\gamma_{x}) \gamma)}{D_{Ga}(x \delta)} 
  \Delta_{KS}(\eta^{-1}(\gamma_{x}) \gamma, \dot{\delta})  \
\left(\mathrm{Ad}(g) \circ \phi \circ \mathrm{Ad}(g)^{-1} \circ
\beta_{\mathfrak{g}}(X) \right),
\end{align*}
Under stable conjugacy, the only terms that change above are the
transfer factors $\Delta_{KS}$, and they differ by a constant.  We may
therefore restrict ourselves to proving the identity for $\dot{\delta} =
\dot{\delta}_{0}$, \emph{i.e.} proving 
 \begin{align}
\label{diffeqid}
\beta_{\mathfrak{h}}(X_{1}) \
&  \frac{D_{H}(\eta^{-1}(\gamma_{x}) \gamma)}{D_{Ga}(x \delta)} 
    \Delta_{KS}( \eta^{-1}(\gamma_{x}) \gamma, x \tilde{\delta})\\
\nonumber & =   \frac{D_{H}(\eta^{-1}(\gamma_{x}) \gamma)}{D_{Ga}(x \delta)} 
  \Delta_{KS}(\eta^{-1}(\gamma_{x}) \gamma, x \tilde{\delta})  \
\left(\mathrm{Ad}(g) \circ \phi \circ \mathrm{Ad}(g)^{-1} \circ
\beta_{\mathfrak{g}}(X) \right)
\end{align}
We have reduced the proof of (\ref{geomatch}) to the proof of
(\ref{diffeqid}). Equation (\ref{diffeqid}) is to be interpreted as
an identity of linear maps on a space of smooth functions.
On the right-hand side $\mathrm{Ad}(g) \circ \phi \circ
\mathrm{Ad}(g)^{-1} \circ \beta_{\mathfrak{g}}(X)$ is a differential
operator determined by  an element in
$\mathcal{S}(\mathfrak{s}^{\tilde{\delta}})$.  It is a differential
operator of smooth functions defined on $S^{\tilde{\delta}}(\mathbb{R})$,
\emph{i.e.}~functions in the variable $x$.
 As a differential
operator, it depends only on a small open neighbourhood of a point $x$
and through the exponential map this open neighbourhood is isomorphic
to an open subset of the Lie algebra
$\mathfrak{s}^{\tilde{\delta}}(\mathbb{R})$.  This differential operator
is then multiplied by the function
$$\frac{D_{H}(\eta^{-1}(\gamma_{x}) \gamma)}{D_{Ga}(x \delta)} 
  \Delta_{KS}(\eta^{-1}(\gamma_{x}) \gamma, x \tilde{\delta})$$
which is a function on $S^{\tilde{\delta}}(\mathbb{R})$.  We
view this function as a locally defined function on
$\mathfrak{s}^{\tilde{\delta}}(\mathbb{R})$.  

On the left-hand side of (\ref{diffeqid}), one first multiplies by the
above function and then one applies the differential operator 
$\beta_{\mathfrak{h}}(X_{1})$.  It is a differential operator determined by
an element in $\mathcal{S}(\mathfrak{s}^{H})$. Consequently,  it is
a differential operator of smooth functions defined on $S^{H}(\mathbb{R})$, or
equivalently, as a locally defined smooth functions on
$\mathfrak{s}^{H}(\mathbb{R})$.  The 
discrepancy of spaces with the right-hand side appears because of the
substitution made in (\ref{geomatchsub}).  In that substitution we are
identifying the element $x \in S^{\tilde{\delta}}(\mathbb{R})$ with the
element $\eta^{-1}(\gamma_{x}) \in S^{H}(\mathbb{R})$ through the
progression
$$\xymatrix@1{
x \ar@{|->}[r]  & x'  \ar@{|->}[r] & \gamma_{x} \ar@{|->}[r]&
\eta^{-1}(\gamma_{x})\\
S^{\tilde{\delta}}(\mathbb{R}) \ar[r]^{\mathrm{Ad}(g)^{-1}} &
  (S')^{a}(\mathbb{R}) \ar[r] & S'_{a}(\mathbb{R})  \ar[r]^{\eta^{-1}} &
  S^{H}(\mathbb{R})  }$$
As we only need to prove an identity of locally defined functions, we
can pass to the isomorphisms
\begin{equation}
\label{samespace}
\mathfrak{s}^{\tilde{\delta}}(\mathbb{R})
\stackrel{\mathrm{Ad}(g)^{-1}}{\rightarrow} (\mathfrak{s}')^{a}(\mathbb{R})
\stackrel{\eta^{-1}}{\rightarrow} \mathfrak{s}^{H}(\mathbb{R}).
\end{equation}
The right-hand side of  (\ref{diffeqid}) applies to locally defined
functions on $S^{\tilde{\delta}}(\mathbb{R})$ and the left-hand side
applies to locally defined functions on $S^{H}(\mathbb{R})$.  The two
are identified via the isomorphisms of (\ref{samespace}).  It is
most convenient to identify both sides with functions on the middle space
$(\mathfrak{s}')^{a}(\mathbb{R})$, for then Equation (\ref{diffeqid})
becomes
\begin{align}
\label{diffeqid1}
& \left(  I_{-\mu^{*}} \circ \phi \circ \mathrm{Ad}(g)^{-1} \circ 
  \beta_{\mathfrak{g}}(X) \right)  \frac{D_{H}(\eta^{-1}(\gamma_{x})
    \gamma)}{D_{Ga}(x \delta)}  
  \Delta_{KS}(\eta^{-1}(\gamma_{x}) \gamma, x \tilde{\delta}) \\
\nonumber & = \frac{D_{H}(\eta^{-1}(\gamma_{x}) \gamma)}{D_{Ga}(x \delta)}  
  \Delta_{KS}(\eta^{-1}(\gamma_{x}) \gamma, x \tilde{\delta})\  (\phi \circ
\mathrm{Ad}(g)^{-1} \circ \beta_{\mathfrak{g}}(X)).
\end{align}
In order to simplify this hypothetical identity further, we unpack
the definition of 
$$\Delta_{KS} = e([G \rtimes a]_{z})\, \epsilon_{L}(V, \psi) \,
(\Delta_{I}^{\mathrm{new}})^{-1} \Delta_{II} 
(\Delta_{III}^{\mathrm{new}})^{-1} \Delta_{IV}$$
(\cite[\S 4.9]{KalLLCD}).
It is immediate from 
\cite[\S 4.5]{KS99} that $\Delta_{IV}(\eta^{-1}(\gamma_{x})
\gamma, x \tilde{\delta})$ cancels with $
\frac{D_{H}(\eta^{-1}(\gamma_{x}) \gamma)}{D_{Ga}(x \delta)}$.  In
addition, the three terms,  $e([G \rtimes a]_{z})$, $\epsilon_{L}(V,
\psi)$, and  $\Delta_{I}^{\mathrm{new}}(\eta^{-1}(\gamma_{x}) \gamma, x
\tilde{\delta}) = \Delta_{I}^{\mathrm{new}}( \gamma,
\tilde{\delta})$, are all constants independent of $x$ (\cite[\S
  3.7]{KalLLCD}, \cite[\S
  5.3 and 4.2]{KS99}).
As a result, Equation (\ref{diffeqid1}) reduces to 
\begin{align*}
&    \left( 
  I_{-\mu^{*}} \circ \phi \circ \mathrm{Ad}(g)^{-1} \circ 
  \beta_{\mathfrak{g}}(X) \right)\  \Delta_{II}(\eta^{-1}(\gamma_{x})
  \gamma, x \tilde{\delta}) 
(\Delta_{III}^{\mathrm{new}}(\eta^{-1}(\gamma_{x}) \gamma, x
\tilde{\delta}))^{-1} \\
\nonumber & =  \Delta_{II}(\eta^{-1}(\gamma_{x}) \gamma, x \tilde{\delta})
(\Delta_{III}^{\mathrm{new}}(\eta^{-1}(\gamma_{x}) \gamma, x
\tilde{\delta}))^{-1}
  (\phi \circ \mathrm{Ad}(g)^{-1} \circ \beta_{\mathfrak{g}}(X)).
\end{align*}
By Lemma \ref{neartildedelta}.2, we have a further reduction to 
\begin{align}
\label{diffeqid2}
&    \left( 
  I_{-\mu^{*}} \circ \phi \circ \mathrm{Ad}(g)^{-1} \circ 
  \beta_{\mathfrak{g}}(X) \right) \ \Delta_{II}(\eta^{-1}(\gamma_{x})
  \gamma, x \tilde{\delta}) \, 
\langle x^{*}, [a_{S'}] \rangle  \\
\nonumber & =  \Delta_{II}(\eta^{-1}(\gamma_{x}) \gamma, x \tilde{\delta})\, 
\langle x^{*}, [a_{S'}] \rangle
 (\phi \circ \mathrm{Ad}(g)^{-1} \circ \beta_{\mathfrak{g}}(X)).
\end{align}
Recall that we are regarding this identity as locally defined on
$(\mathfrak{s}')^{a}(\mathbb{R})$.  Regarded in this manner,  Lemma
\ref{neartildedelta}.3 tells us that the character $\langle \cdot
, [a_{S'}] \rangle$ is equal to $e^{   \rho_{H} -\rho_{G_\mathrm{res}}
  + \mu^{*}}$.  This result depends on $\tilde{\delta}$ being elliptic
and this is the first time in the argument that we are using the
assumption of being elliptic.
The argument on \cite[page 66]{Mezo13} shows that
$\Delta_{II}(\eta^{-1}(\gamma_{x}) \gamma, x \tilde{\delta})$,
regarded in this manner,  is a constant multiple of $e^{
  \rho_{G_\mathrm{res}} - \rho_{H}}$.  This also uses the fact that
$\tilde{\delta}$ is elliptic, as it affects the form of the based
$\chi$-data in the definition of $\Delta_{II}$.  Equation \ref{diffeqid2}
now simplifies to 
$$\left( 
  I_{-\mu^{*}} \circ \phi \circ \mathrm{Ad}(g)^{-1} \circ 
  \beta_{\mathfrak{g}}(X) \right) \ e^{\mu^{*}}  
 =  e^{\mu^{*}} \ (\phi \circ
\mathrm{Ad}(g)^{-1} \circ \beta_{\mathfrak{g}}(X)).$$
This equation may be seen to hold by reviewing the definition of
$I_{-\mu^{*}}$ and applying the product rule to the left-hand side.
This concludes the proof of (\ref{geomatch}) for $\tilde{\delta} =
\dot{\delta}_{1}$.  

The proof of  (\ref{geomatch}) for arbitrary $\dot{\delta}_{j}$ is the
same up to the point near the end where we invoke Lemma
\ref{neartildedelta}.3 to obtain 
$$\langle \cdot, [a_{S'}] \rangle = e^{\rho_{H}-\rho_{G_{\mathrm{res}}} + \mu^{*}}.$$
If $\dot{\delta}_{j}$ is not elliptic the choice of based $\chi$-data
alters this equation by replacing $\rho_{H}$ with the half-sum of the
positive \emph{imaginary} roots of $R(S^{H},H)$, and replacing
$\rho_{G_{\mathrm{res}}}$ with the half-sum of the of the positive
\emph{imaginary} roots of $R_{\mathrm{res}}(S',G)$.  The choice of based
$\chi$-data alters $\Delta_{II}(\eta^{-1}(\gamma_{x}) \gamma, x
\dot{\delta}_{j})$ in the same way so that the product 
of the two terms cancel as before (\emph{cf.}~\cite[Lemma 3.3.D and
  Lemma 3.5.A]{LS87}).  This is the only change to the
proof required for $\dot{\delta}_{j}$ and so the proposition is proved.
\end{proof}

\begin{pro}
\label{eigendist}
The distribution $\Theta_{|}$ is an eigendistribution with
infinitesimal character $if \in \mathfrak{s}^{*}$.
\end{pro}
\begin{proof}
Let $X \in \mathcal{Z}(\mathfrak{g})$ and $\tilde{f} \in
C_{c}^{\infty}(G_{z}(\mathbb{R})^{\circ} \, \tilde{\delta})$.  From
Section 5.4 and Proposition \ref{zmatch}
\begin{align*}
X \Theta_{|}(\tilde{f}) & = \Theta(X^{\mathrm{tr}} \tilde{f})\\
& =  \sum_{\substack{\pi \in \Pi_\varphi(G_z)\\\pi\circ a\cong\pi}} 
\mathrm{tr}(\pi \boxtimes \rho^{\vee})^{\mathrm{can}} (X^{\mathrm{tr}}
\tilde{f},
\tilde{s}^{-1}) - S\Theta_{\varphi_{1}}(X^{\mathrm{tr}}_{1} f_{1})\\
& =  \sum_{\substack{\pi \in \Pi_\varphi(G_z)\\\pi\circ a\cong\pi}} 
  \tilde{\rho}(\tilde{s}) \ \mathrm{tr}\,\tilde{\pi} (X^{\mathrm{tr}}\tilde{f})
 - \sum_{\pi_{1} \in \Pi_{\varphi_{1}}(H)} \mathrm{tr}\, \pi_{1}((X^{\mathrm{tr}})_{1} f_{1})
\end{align*}
The infinitesimal character of each representation $\tilde{\pi}$ in
the first sum is stably conjugate to the infinitesimal character of $\pi \in
\Pi_{\varphi}(G_{z})$.  By the local Langlands 
correspondence, the infinitesimal characters of these representations
are stably conjugate to
to $\mathrm{Ad}(g)^{-1}if \in (\mathfrak{s}')^{*}$ 
where 
$$\varphi(1 \rtimes z) = z^{\mathrm{Ad}(g)^{-1} if} \,
\bar{z}^{\sigma_{S'}(\mathrm{Ad}(g)^{-1} if)}, \quad z \in
\mathbb{C}^{\times} \subset W_{\mathbb{R}}$$
for $\varphi$ chosen so that its image lies in ${^L}S'$.  Thus, 
$$ \sum_{\substack{\pi \in \Pi_\varphi(G_z)\\\pi\circ a\cong\pi}} 
  \tilde{\rho}(\tilde{s}) \ \mathrm{tr}\,\tilde{\pi}
  (X^{\mathrm{tr}}\tilde{f}) = if (\beta_{\mathfrak{g}}(X))
  \sum_{\substack{\pi \in \Pi_\varphi(G_z)\\\pi\circ a\cong\pi}}  
  \tilde{\rho}(\tilde{s}) \ \mathrm{tr}\,\tilde{\pi} (\tilde{f}),$$
and the first sum in $\Theta_{|}$ is an eigendistribution of
infinitesimal character $if$. 

To deal with the second sum appearing in $X \Theta_{|}$, we consider
the infinitesimal character of each representation $\pi_{1}
\in \Pi_{\varphi_{1}}(H)$. In this case the infinitesimal characters
$\mu_{1} \in (\mathfrak{s}^{H})^{*}$ are
given by 
$$\varphi_{1}(1 \rtimes z ) = z^{\mu_{1}}
\bar{z}^{\sigma_{S^{H}}(\mu_{1})}, \quad z \in \mathbb{C}^{\times}
\subset W_{\mathbb{R}}.$$ 
and $\mu_{1}$ is determined by the equation $\varphi = \xi \circ
\varphi_{1}$ as follows.  Fixing 
based $\chi$-data for $R_{\mathrm{res}}(S', G)$ as in Lemma
\ref{neartildedelta} and using the L-embeddings appearing in the proof
of that lemma, we see from \cite[\S 7.b]{SheTE2} that
$\varphi =  \xi_{S'} \circ \varphi_{S'}$, where $\varphi_{S'}:
W_{\mathbb{R}} \rightarrow {^L}S'$ is an L-parameter whose
infinitesimal character equals $\mathrm{Ad}(g)^{-1} if -
\rho_{G_{\mathrm{res}}}$.
Similarly, $\varphi_{1} =  \xi^{H} \circ  {^L}\eta \circ
\varphi_{S^{H}}$ where $\varphi_{S^{H}}: 
W_{\mathbb{R}} \rightarrow {^L}S'_{a}$ is an L-parameter whose
infinitesimal character equals $\mu_{1} - \rho_{H}$.
Substituting these equations into  $\varphi = \xi \circ \varphi_{1}$,
we find that
$$\xi_{S'}(\varphi_{S'}(z)) = \xi \circ \xi^{H} \circ {^L}\eta \circ
\varphi_{S^{H}}(z) = \xi'_{S'}(\varphi_{S^{H}}(z)) =
\xi_{S'}(a_{S'}(z) \, \varphi_{S^{H}}(z)), \quad z \in \mathbb{C}^{\times}.$$ 
As $\xi_{S'}$ is injective, we have in turn that
$$\varphi_{S'}(z) = a_{S'}(z) \, \varphi_{S^{H}}(z), \quad z \in
\mathbb{C}^{\times}$$
and 
$$\mathrm{Ad}(g)^{-1} if - \rho_{G_{\mathrm{res}}} = (\rho_{H}-
\rho_{G_{\mathrm{res}}}   + \mu^{*}) + (\mu_{1} -\rho_{H})$$
(Lemma \ref{neartildedelta}.3). 
Consequently, the infinitesimal character of $\varphi_{1}$ equals
$$\mu_{1} = \mathrm{Ad}(g)^{-1}if - \mu^{*}$$
under the identifications of $(\mathfrak{s}')^{a} \cong
\mathfrak{s}^{H}$.  It is tempting to think that the infinitesimal
character of the second sum of $\Theta_{|}$ is equal to $\mu_{1}$ and
not $\mathrm{Ad}(g)^{-1}if$.  However, 
the shift by $\mu^{*}$ is undone by the discrepancy between
$(X^{\mathrm{tr}})_{1}$ and $(X_{1})^{\mathrm{tr}}$.  
Indeed, from the characterization of the transpose map $X \mapsto
X^{\mathrm{tr}}$ given in \cite[\S X.4]{KnappSS} and the definition
of $X \mapsto X_{1}$ given in (\ref{zedone}), one may deduce that
$(X^{\mathrm{tr}})_{1}$ differs from $(X_{1})^{\mathrm{tr}}$ by
replacing $I_{-\mu^{*}}$ with $I_{\mu^{*}}$.   More precisely,
$$(X^{\mathrm{tr}})_{1} = (\beta^{-1}_{\mathfrak{h}} \circ I_{\mu^{*}}
\circ \eta^{-1} \circ \phi \circ 
\mathrm{Ad}(g)^{-1} \circ \beta_{\mathfrak{g}}(X))^{\mathrm{tr}}.$$
We substitute this into the second sum in $X\Theta_{|}$ and compute
\begin{align*}
 &\sum_{\pi_{1} \in \Pi_{\varphi_{1}}(H)} \mathrm{tr}\,
 \pi_{1}((X^{\mathrm{tr}})_{1} f_{1})\\
&= \mu_{1} \left( I_{\mu^{*}}
\circ \eta^{-1} \circ \phi \circ \mathrm{Ad}(g)^{-1} \circ
\beta_{\mathfrak{g}}(X) \right) \sum_{\pi_{1} \in 
  \Pi_{\varphi_{1}}(H)} \mathrm{tr}\,  \pi_{1}(f_{1}) \\
&= (\mu_{1}+ \mu^{*}) \left( \eta^{-1} \circ \phi \circ \mathrm{Ad}(g)^{-1} \circ
\beta_{\mathfrak{g}}(X) \right) \sum_{\pi_{1} \in 
  \Pi_{\varphi_{1}}(H)} \mathrm{tr}\,  \pi_{1}(f_{1}) \\
& = if (\beta_{\mathfrak{g}}(X)) \sum_{\pi_{1} \in 
  \Pi_{\varphi_{1}}(H)} \mathrm{tr}\,  \pi_{1}(f_{1}).
\end{align*}
This proves that the second sum in $\Theta_{|}$ is also an
eigendistribution of infinitesimal character $if$.
\end{proof}

\begin{thm}
\label{uniquethm}
Suppose $\Theta_{|}$ vanishes on the regular subset of
$S^{\tilde{\delta}}(\mathbb{R}) \tilde{\delta}$.  Then the
distribution $\Theta_{|}$ vanishes on 
$G_{z}(\mathbb{R})^{\circ}  \tilde{\delta}$, that is $\Theta_{|}  = 0$.
\end{thm}
\begin{proof}
This is a straightforward application of \cite[Theorem 15.1]{Ren97}.
In Renard's theorem we may take $\Omega = G_{z}(\mathbb{R})^{\circ}
\tilde{\delta}$ and $B = \exp(\mathfrak{s}^{\tilde{\delta}}(\mathbb{R}))
\tilde{\delta}$. 
There are two hypotheses to be verified.  

The first hypothesis
is that $\Theta_{|}$ is annihilated by an
ideal $\mathcal{B} \subset \mathcal{Z}(\mathfrak{g})$ of ``elliptic
type''.  Recall from assumption (\ref{assumef}) that the regular
element $if \in \mathfrak{g}$  belongs to
$i(\mathfrak{s}^{\tilde{\delta}})^{*}(\mathbb{R})$. From 
these properties of $if$ it follows by definition that the
kernel $\mathcal{B}$ of the homomorphism $\mathcal{Z}(\mathfrak{g}) \rightarrow
\mathbb{C}$   
$$X \mapsto if( \beta_{\mathfrak{g}}(X)), \quad X \in
  \mathcal{Z}(\mathfrak{g})$$
is of elliptic type.  Clearly, the ideal $\langle  X - if(\beta_{\mathfrak{g}}(X))
\rangle$ generated by  the elements
$$ X - if(\beta_{\mathfrak{g}}(X)), \quad X \in
\mathcal{Z}(\mathfrak{g})$$
is contained in $\mathcal{B}$.  On the other hand, it is easy to see
that
$$\mathcal{Z}(\mathfrak{g})/ \langle  X - if(\beta_{\mathfrak{g}}(X))
\rangle \cong \mathbb{C}.$$
Therefore  $\langle  X - if(\beta_{\mathfrak{g}}(X))
\rangle$ is a maximal ideal and is equal to $\mathcal{B}$.
Proposition \ref{eigendist} tells us that $\Theta_{|}$ is annihilated
by the elliptic ideal $\langle  X - if(\beta_{\mathfrak{g}}(X))
\rangle = \mathcal{B}$.

The remaining hypothesis to be verified is that there exists $C > 0$
such that 
$$D_{Ga}(x \delta) \, |\Theta(x)| \leq C$$
for all regular elements $x \in G_{z}(\mathbb{R})^{\circ}
\tilde{\delta}$. 
In this inequality, $\Theta$ is taken to be a locally integrable
analytic function.  This is possible by Proposition  \ref{eigendist} and
\cite[Theorem 2.1.1]{Bouaziz87}.  Since $\Theta_{|}$ is a \emph{tempered}
invariant eigendistribution (\cite{She12}), the above inequality
is given by \cite[Proposition 3.6.1]{Bouaziz87}.  
\end{proof}
\begin{thm}
\label{thetazero}
Suppose $\Theta$ vanishes on the strongly regular elliptic
subset of $[G \rtimes a]_{z}(\mathbb{R})$.  Then the  distribution
$\Theta$ vanishes on $[G \rtimes a]_{z}(\mathbb{R})$. 
\end{thm}
\begin{proof}
The strongly regular subset is dense and open in the regular subset
(\emph{cf.}~\cite[Lemma 1.5.1]{Bouaziz87}).   Let $\Theta_{|}$ be the
restriction to a connected component $G_{z}(\mathbb{R})^{\circ}
\tilde{\delta}$ which contains strongly regular elliptic elements.  If
$\Theta_{|}$  vanishes on the strongly  
regular subset of $S^{\tilde{\delta}}(\mathbb{R}) \tilde{\delta}$
then, as an analytic function (\cite[Theorem 2.1.1]{Bouaziz87}), it
vanishes on the regular subset of $S^{\tilde{\delta}}(\mathbb{R})
\tilde{\delta}$ as well.  Theorem \ref{uniquethm} therefore applies
and shows that $\Theta$ vanishes on $G_{z}(\mathbb{R})^{\circ}
\tilde{\delta}$.

Now let $\Theta_{|}$ be the
restriction  of $\Theta$ to a connected component $G_{z}(\mathbb{R})^{\circ}
\tilde{\delta}$ which does not contain any strongly regular elliptic
elements.   Corollary
\ref{goodcomp2} implies that 
$$ \sum_{\substack{\pi \in \Pi_\varphi(G_z)\\\pi\circ a\cong\pi}} 
\mathrm{tr}(\pi \boxtimes \rho^{\vee})^{\mathrm{can}} (\tilde{f},
\tilde{s}^{-1}) = 0, \quad \tilde{f} \in
C_{c}^{\infty}(G_{z}(\mathbb{R})^{\circ} \tilde{\delta}).$$
We must prove that $S\Theta_{\varphi_{1}}(f_{1}) = 0$ as well.
The function $f_{1} \in C_{c}^{\infty}(H(\mathbb{R}))$ has support in
the set of norms of elements in $G_{z}(\mathbb{R})^{\circ}
\tilde{\delta}$.  To complete the proof, we shall prove that
$\Theta_{\pi_{1}}(\gamma_{1}) = 0$ 
for every $\pi_{1} \in \Pi_{\varphi_{1}}(H)$ and norm $\gamma_{1} \in
H(\mathbb{R})$. 
To this end, let $\tilde{\delta}_{1}  \in G_{z}(\mathbb{R})^{\circ}
\tilde{\delta}$ be strongly regular semi-simple, and suppose it has
norm $\gamma_{1} \in H(\mathbb{R})$ given by maps of pairs
$$(\tilde{\delta}_{1}, S_{1}) \mapsto
(\tilde{\delta}_{1}', S_{1}') \mapsto (\gamma_{1}', S_{1,a}')
\mapsto (\gamma_{1}, S_{1}^{H})$$
as in the proof of Lemma \ref{srell}.  

Let us prove that $S_{1}^{H}$ is not elliptic.  Suppose by way of
contradiction that $S_{1}^{H}$ is elliptic.  
The proof of Lemma \ref{lem:ellnorm}(1) then shows that $(S_{1}')^{a, \circ}$ is
elliptic in $G^{a,\circ}$. 
Proposition \ref{pro:aa} shows that $S_{1}'$ is elliptic.  This forces
$S_{1}$ and $\tilde{\delta}_{1}$ to be elliptic as well, and 
contradicts our hypothesis on $G_{z}(\mathbb{R})^{\circ}
\tilde{\delta}$.  Thus, the maximal torus $S_{1}^{H} \subset H$ is not
elliptic. 

Our next and final step is to prove that $\gamma_{1}$ does not belong
to $Z_{H}(\mathbb{R}) H_{\mathrm{der}}(\mathbb{R})^{\circ}$, where
$H_{\mathrm{der}}$ is the derived subgroup of $H$.  This step
proves $\Theta_{\pi_{1}}(\gamma_{1}) = 0$, for the support of
$\Theta_{\pi_{1}}$ lies in $Z_{H}(\mathbb{R})
H_{\mathrm{der}}(\mathbb{R})^{\circ}$ for every $\pi_{1} \in
\Pi_{\varphi_{1}}(H)$ (\cite[page 134]{Lan89}).   

Suppose by way of
contradiction that $\gamma_{1} \in Z_{H}(\mathbb{R})
H_{\mathrm{der}}(\mathbb{R})^{\circ}$.  As the torus $S_{1}^{H}$ is
not elliptic, there are real roots in $R(S_{1}^{H}, H)$
\cite[Proposition 11.16]{KnappSS}. 
 For any real
root $\alpha \in R(S_{1}^{H},H)$ the value $\alpha(\gamma_{1})$ is
non-zero and real.  We wish to find a real root $\alpha \in
R(S_{1}^{H},H)$ such that $\alpha(\gamma_{1}) > 0$.  Given such a real root
we choose $v \in \mathfrak{s}_{1}^{H}(\mathbb{R})$ with
$$\alpha(\exp(v)) = \exp( d \alpha(v)) = \alpha(\gamma_{1})^{-1}.$$
We may further choose $v$ so that 
$\gamma_{2}  = \gamma_{1} \exp(v) \in
\gamma_{1}S_{1}^{H}(\mathbb{R})^{\circ}$ is 
semiregular in the sense that $\alpha(\gamma_{2}) = 1$ and
$\beta(\gamma_{2}) \neq 1$ for all $\beta \in R(S_{1}^{H}, H)$, with
$\beta \neq \pm \alpha$. To find such a real root we decompose
$\gamma_{1} = z \,\gamma_{\mathrm{der}}$ where $z \in Z_{H}(\mathbb{R})$
and $\gamma_{\mathrm{der}} \in H_{\mathrm{der}}(\mathbb{R})^{\circ}
\cap S_{1}^{H}(\mathbb{R})$.  According to \cite[Proposition
7.110]{KnappLie}, the element $\gamma_{\mathrm{der}}$ may be
decomposed further as
$$\gamma_{\mathrm{der}} = \gamma_{A} \, \gamma_{T} \, \gamma_{F},$$
where $\gamma_{A}$ lies in the split component of $S_{1}^{H}$,
$\gamma_{T}$ lies in the anisotropic component of $S_{1}^{H}$, and
$\gamma_{F}$ is a product of elements of the form
$$\exp (2 \pi i |\alpha'|^{-2} H_{\alpha'})$$
for real roots $\alpha' \in R(S_{1}^{H} \cap H_{\mathrm{der}},
H_{\mathrm{der}}) \cong R(S_{1}^{H}, H)$ and $\alpha'(H_{\alpha'}) =
|\alpha'|^{2}$.  If $\gamma_{F} = 1$ then we may take $\alpha$ to be
any real root in $R(S_{1}^{H}, H)$ as 
$$\alpha(\gamma_{1}) = \exp(d
\alpha(\log(\gamma_{A}))) > 0.$$
If $\gamma_{F} \neq 1$ we choose  $\alpha$ as in Corollary
\ref{trivalpha} below.  With the desired root $\alpha$ ensured, we have a 
semiregular element
$$\gamma_{2}  \in \gamma_{1} S_{1}^{H}(\mathbb{R})^{\circ} \subset
Z_{H}(\mathbb{R}) H_{\mathrm{der}}(\mathbb{R})^{0}.$$

We shall now apply an argument of Shelstad to arrive at an element
$\gamma_{3} \in Z_{H}(\mathbb{R})
H_{\mathrm{der}}(\mathbb{R})^{\circ}$ which is the norm of an element in
$\tilde{\delta}_{3} \in G_{z}(\mathbb{R})^{\circ} \tilde{\delta}$ and
lies in a maximal torus which is less split than $S_{1}^{H}$.  
First off, by \cite[Lemma 6.1-6.2]{She12} the semiregular element
$\gamma_{2}$ is a $S_{1}^{H}$\emph{-norm} of an element
$$\tilde{\delta}_{2} \in S_{1}^{\tilde{\delta}_{1}}(\mathbb{R})^{\circ}
\tilde{\delta}_{1} \subset G_{z}(\mathbb{R})^{0} \tilde{\delta}.$$
This notion of norm 
allows for non-regular elements and reduces to the usual notion of
norm for strongly regular elements (\cite[page 1940]{She12}).
Let $d_{\alpha}$ be the Cayley transform which takes the torus
$S_{1}^{H}$ to the less split maximal torus $d_{\alpha}S_{1}^{H}
\subset H$.  The proof of \cite[Lemma 6.6 (ii)]{She12} shows that
$\gamma_{2} = d_{\alpha} \gamma_{2} \in d_{\alpha}S_{1}^{H}$ is a
$d_{\alpha}S_{1}^{H}$-norm 
of $\tilde{\delta}_{2}$.  Parallel to $d_{\alpha}$, Shelstad defines a Cayley
transform $\mathrm{Ad}(t)$ which sends
$(S_{1}^{\tilde{\delta}_{1}})(\mathbb{R})^{\circ}$ to the identity 
component $(tS_{1}^{\tilde{\delta}}t^{-1})(\mathbb{R})^{\circ}$ of a
less split torus of $G_{z}$.  By \cite[Lemma 
6.1]{She12}, there is a strongly regular semi-simple element
$$\tilde{\delta}_{3} \in
(tS_{1}^{\tilde{\delta}}t^{-1})(\mathbb{R})^{\circ} \tilde{\delta}_{2}
\subset  G_{z}(\mathbb{R})^{\circ} \tilde{\delta}$$
with norm
$$\gamma_{3} \in  (d_{\alpha}S_{1}^{H})(\mathbb{R})^{\circ} \gamma_{2}
\subset Z_{H}(\mathbb{R}) H_{\mathrm{der}}(\mathbb{R})^{\circ}.$$
This process may be repeated until one reaches a norm
of an element in $G_{z}(\mathbb{R})^{\circ}
\tilde{\delta}$ contained in a torus with no real roots,
\emph{i.e.}~an elliptic torus of $H(\mathbb{R})$ \cite[Proposition
  11.16]{KnappSS}.   However, we have proved that
elliptic tori contain no norms of elements in
$G_{z}(\mathbb{R})^{\circ} \tilde{\delta}$.  This concludes the proof
by contradiction, so 
that $\gamma_{1} \notin Z_{H}(\mathbb{R}) H_{\mathrm{der}}(\mathbb{R})^{\circ}$.
\end{proof}

We will now supply a technical result used in the above proof, with  slightly different notation. 
Suppose $G'$ is a real connected semisimple Lie group with a
complexification $G'_{\mathbb{C}}$ \cite[Section VII.1]{KnappLie}.
Let  $T'$ be a Cartan subgroup,
$\mathfrak{g}'$ be the complex Lie algebra of $G'$ and
$\mathfrak{t}'$ is the complex Lie algebra of $T'$.  Let $(\cdot,
\cdot)$ be the Killing form on $\mathfrak{t}'$. 
For every $\beta \in R(\mathfrak{t}',\mathfrak{g}')$ choose $H_{\beta}
\in \mathfrak{t}'$ 
such that $\beta(t) = (t, H_{\beta})$ for all $t \in \mathfrak{t}'$.
If $\beta \in R(\mathfrak{t}', \mathfrak{g}')$ is real then 
$$\gamma_{\beta} = \exp(2 \pi i | \beta |^{-2} H_{\beta}) \in T'$$
is an element of order two \cite[Proposition 7.110]{KnappLie}.

\begin{lem}
\label{rootprop}
Suppose $\gamma = \gamma_{\beta_{1}} \cdots \gamma_{\beta_{n}}$ for
real roots 
$\beta_{1}, \ldots , \beta_{n} \in R(\mathfrak{t}',\mathfrak{g}')$.  
Then $\gamma =  \gamma_{\beta_{1}'} \cdots \gamma_{\beta_{m}'}$ for
real roots  $\beta_{1}', \ldots, \beta_{m}'  \in
R(\mathfrak{t}',\mathfrak{g}')$ with 
the property that 
$$(\beta_{j}', \beta_{\ell}') \neq 0  \mbox{ and } j \neq \ell
\Rightarrow |\beta_{j}'|^{2} = 2|\beta_{\ell}'|^{2} \mbox{ or }
|\beta_{\ell}'|^{2} = 2|\beta_{j}'|^{2}.$$ 
\end{lem}
\begin{proof}
The facts about pairs of roots that we use below can all be found in
\cite[Section  9.4]{Humphreys80}. We prove by induction on $n \geq 1$.
The lemma holds vacuously for $n =1$.  Suppose $n >1$.  If the
property of the lemma does not hold for  $\beta_{1}, \ldots ,
\beta_{n},$ then  
without loss of generality $(\beta_{1}, \beta_{2}) \neq 0$ with either
$|\beta_{1}|^{2} = |\beta_{2}|^{2}$ or $|\beta_{1}|^{2} =
3|\beta_{2}|^{2}$.    In addition, either
$\beta_{1} + \beta_{2}$ or  $\beta_{1} - \beta_{2}$ is a real root.
Suppose first that  
$|\beta_{1}|^{2} = |\beta_{2}|^{2}$  and $\beta_{1} + \beta_{2}$ is a
root. Then $|\beta_{1} + \beta_{2}| = |\beta_{1}|
= |\beta_{2}|$,  $H_{\beta_{1}} + 
H_{\beta_{2}} = H_{\beta_{1} + \beta_{2}}$ and 
$$ |H_{\beta_{1}}|^{-2} H_{\beta_{1}} + |H_{\beta_{2}}|^{-2}
H_{\beta_{2}} =  |H_{\beta_{1} + \beta_{2}}  |^{-2} H_{\beta_{1} +
  \beta_{2}}.$$
It follows that $\gamma_{\beta_{1}} \gamma_{\beta_{2}} =
\gamma_{\beta_{1} + \beta_{2}}$ and $\gamma = \gamma_{\beta_{1} +
  \beta_{2}} \gamma_{\beta_{3}} \cdots \gamma_{\beta_{n}}$.  The same
computations leads to $\gamma = \gamma_{\beta_{1} -
  \beta_{2}} \gamma_{\beta_{3}} \cdots \gamma_{\beta_{n}}$ if
$\beta_{1} - \beta_{2}$ is a root since $\gamma_{\beta_{2}} =
\gamma_{\beta_{2}}^{-1}$.  In both of these cases we are able to apply the
induction assumption.  
Now  suppose $|\beta_{1}|^{2} = 3 |\beta_{2}|^{2}$ and $\beta_{1} +
\beta_{2}$ is a root.  Then $|\beta_{1} + \beta_{2}| = |\beta_{2}|$
and
\begin{align*}
& |\beta_{1} + \beta_{2}|^{-2} H_{\beta_{1}+ \beta_{2}}
-|\beta_{1}|^{-2} H_{\beta_{1}} - |\beta_{2}|^{-2} H_{\beta_{2}} \\
& = \frac{3( H_{\beta_{1}} + H_{\beta_{2}})  - H_{\beta_{1}}  - 3
    H_{\beta_{2}} }{3 |\beta_{2}|^{2}} \\
& = 2 |\beta_{1}|^{- 2} H_{\beta_{1}}.
\end{align*}
It follows that
$$\gamma_{\beta_{1} + \beta_{2}} \gamma_{\beta_{1}}^{-1} \gamma_{\beta_{2}}^{-1}
= \exp (4 \pi |\beta_{1}|^{-2} H_{\beta_{1}}) = \gamma_{\beta_{1}}^{2}
= 1,$$
$\gamma_{\beta_{1} +
  \beta_{2}} = \gamma_{\beta_{1}} \gamma_{\beta_{2}}$, and $\gamma =
\gamma_{\beta_{1} + \beta_{2}} 
\gamma_{\beta_{3}} \cdots \gamma_{\beta_{n}}$.  Similar computations
lead to $\gamma = \gamma_{\beta_{1} - 
  \beta_{2}} \gamma_{\beta_{3}} \cdots \gamma_{\beta_{n}}$ if
$\beta_{1} - \beta_{2}$ is a root.  The induction assumption therefore
applies to  these remaining cases and the proof is complete.
\end{proof}

\begin{cor}
\label{trivalpha}
Suppose $\gamma = \gamma_{\beta_{1}} \cdots \gamma_{\beta_{n}}$ for
real roots  $\beta_{1}, \ldots , \beta_{n} \in
R(\mathfrak{t}',\mathfrak{g}')$.  Then there exists a real root $\alpha \in
R(T'_{\mathbb{C}}, G'_{\mathbb{C}})$ such that $\alpha(\gamma) = 1$. 
\end{cor}
\begin{proof}
Let $\gamma =  \gamma_{\beta_{1}'} \cdots \gamma_{\beta_{m}'}$ for
real roots satisfying the property of Lemma \ref{rootprop}.  If the
the roots $\beta_{1}', \ldots , \beta_{m}'$ are mutually orthogonal
then choose $\alpha$ so that its differential is $\beta_{1}'$.  We
compute that $\alpha(\gamma)$ equals
\begin{align*}
 &\exp(2\pi i
|\beta_{1}'|^{-2}\beta_{1'}(H_{\beta_{1}'})) \exp( 2 \pi i
|\beta_{2}'|^{-2} \beta_{1}'(H_{\beta_{2}'}) ) \cdots  \exp( 2 \pi i
|\beta_{m}'|^{-2} \beta_{1}'(H_{\beta_{m}'}) )\\
& =  \exp(2\pi i
|\beta_{1}'|^{-2}| | \beta_{1}'|^{2}) \exp( 2 \pi i
|\beta_{2}'|^{-2} 0 ) \cdots  \exp( 2 \pi i
|\beta_{m}'|^{-2} 0 ) = 1.
\end{align*}
Otherwise, without loss of generality,  $(\beta_{1}', \beta_{2}') \neq
0$ and $|\beta_{1}'|^{2} = 2|\beta_{2}'|^{2}$.  We choose
$\alpha$ so that its differential is the long root $\beta_{1}'$. The corollary
follows by expanding  $\alpha(\gamma)$ as above and using the fact
that  $2 |\beta_{j}'|^{-2} \,\beta_{1}'(H_{\beta_{j}'})$ is an even
integer for all $1 \leq j \leq m$ \cite[Section 9.4, Table 1]{Humphreys80}.
\end{proof}

\subsection{The left hand side: sliced version} \label{sub:lhs_sliced}

We now begin with the main part of the proof of Theorem \ref{thm:main}, which will consist of developing alternative formulas for both sides of \eqref{eq:main1} under the assumption that $\tilde\delta$ is strongly regular elliptic, and then showing that they agree. In this subsection we will focus on the left-hand side of \eqref{eq:main1}. For this, let $a \in A^{[\varphi],[z]}$. Since only the coset $G \rtimes a$ of $\tilde G$ will play a role for our computation (this is what we mean by the word ``sliced''), we replace $A$ by its subgroup $\<a\>$.

The coset $[G \rtimes a]_z(\R)$ is non-empty, because $a \in A^{[z]}$, and we consider an elliptic regular semi-simple element $\tilde\delta \in [G \rtimes a]_z(\R)$. Also $[\hat G \rtimes a^{-1}]_\varphi \subset \tilde S_\varphi$ is non-empty and we let $\tilde S_{\varphi,a}^+$ be its preimage in $\tilde S_\varphi^+$ and fix $\tilde s \in \tilde S_{\varphi,a}^+$. The left-hand side of \eqref{eq:main1} is
\begin{equation} \label{eq:lhs_sliced}
e(G_z)\sum \tx{tr}(\pi \boxtimes \rho^\vee)^\tx{can}(\tilde\delta,\tilde s^{-1}),
\end{equation}
where the sum runs over those $\pi \in \Pi_\varphi(G_z)$ whose isomorphism class is stable under conjugation by $\tilde G_z(\R)$, $\rho \in \tx{Irr}(S_\varphi^+)$ corresponds to $\pi$, and $(\pi\boxtimes\rho^\vee)^\tx{can}$ is the canonical extension of $\pi\boxtimes\rho$ to $\tilde G_z(\R) \times_A \tilde S_\varphi^+$.

In this subsection we will develop a formula for \eqref{eq:lhs_sliced}, whose final form will be given in \eqref{eq:lhs_sliced_final} below.
The first part of it will be the evaluation of a given summand $\tx{tr}(\pi \boxtimes \rho^\vee)^\tx{can}(\tilde\delta,\tilde s^{-1})$ in \eqref{eq:lhs_sliced}, whose formula will be given in \eqref{eq:el0} below. 

To that end, fix $\pi \in \Pi_\varphi(G_z)$ stable under conjugation by $\tilde G_z(\R)$ and let $\rho \in \tx{Irr}(\pi_0(S_\varphi^+))$ correspond to it. The stability of the isomorphism class of $\pi$ under $\tilde G_z(\R)$ is equivalent to the stability of the isomorphism class of $\rho$ under $\tilde S^+_\varphi$. Because $\pi_0(S_\varphi^+)$ is finite abelian, $\rho$ is a finite order character, and hence  has an extension to a character of $\pi_0(\tilde S^+_\varphi)$. To obtain such an extension $\tilde\rho$ it is enough to choose $\sqrt[n]{\rho(\tilde s^n)} \in \C^\times$, where $n=\tx{ord}(a)$, and then set $\tilde\rho(\tilde s)=\sqrt[n]{\rho(\tilde s^n)}$. In fact, there is a preferred choice of $\sqrt[n]{\rho(\tilde s^n)}$ once $i=\sqrt{-1} \in \C^\times$ has been fixed, because we can write $\rho(\tilde s^n)=e^{ix}$ for a unique $x \in [0,2\pi)$ and then take $\tilde\rho(\tilde s)=e^{ix/n}$. The particular choice will however be irrelevant for us.

Lemma \ref{lem:rest} implies that $\pi$ has an extension $\tilde\pi$ to a representation of $\tilde G_z(\R)$. Then $(\pi\boxtimes\rho^\vee)^\tx{can}$ is obtained as the restriction of $\tilde\pi \boxtimes \tilde\rho^\vee$ from $\tilde G_z(\R) \times \tilde S^+_\varphi$ to $\tilde G_z(\R) \times_A \tilde S^+_\varphi$. This restriction does not depend on the choice of extension $\tilde\rho$ of $\rho$; indeed if we replace $\tilde\rho$ by $\tilde\rho\otimes\chi$ for a character $\chi : A \to \C^\times$, then $\tilde\pi$ is replaced by $\tilde\pi\otimes\chi$, and $\tilde\pi\boxtimes\tilde\rho^\vee$ is replaced by $(\tilde\pi\boxtimes\tilde\rho^\vee) \otimes (\chi\boxtimes\chi^{-1})$, but the second factor is trivial on $A \times_A A$.

We now compute 
\[ \tx{tr}(\pi \boxtimes \rho^\vee)^\tx{can}(\tilde\delta,\tilde s^{-1}) = \tx{tr}(\tilde\pi(\tilde\delta)) \cdot \tx{tr}(\tilde\rho^\vee(\tilde s^{-1})).\]
We have $\tx{tr}(\tilde\rho^\vee(\tilde s^{-1}))=\tx{tr}(\tilde\rho(\tilde s))=\tilde\rho(\tilde s)$, the latter since $\tilde\rho$ is a character. For $\tx{tr}(\tilde\pi(\tilde\delta))$ we invoke Bouaziz's formula given in Theorem \ref{thm:bouaziz}. If no element in the elliptic regular conjugacy class in $\mf{g}_z^*(\R)$ corresponding to $\pi$ commutes with $\tilde\delta$ the character is zero. Otherwise let $f \in \mf{g}_z^*(\R)$ belong to this conjugacy class and commute with $\tilde\delta$ and set $\tilde S := \tx{Cent}(f,\tilde G_z) $ and $S = \tilde S \cap G_z$, so that $S$ is an elliptic maximal torus of $G_z$ and $\tilde\delta \in \tilde S(\R)$. The formula of Bouaziz then says
\[ \tx{tr}(\tilde\pi(\tilde\delta)) = (-1)^{q(G_z)}\sum_{\substack{u \in N_{\tilde G_z(\R)}(S)/\tilde S(\R) \\ u^{-1}\tilde\delta u \in \tilde S(\R)}}\frac{\tilde\tau_\lhd(u^{-1}\tilde\delta u)}{\det(1-u^{-1}\tilde\delta u|\mf{u}_{f})},\]
where $\tilde\tau$ is the genuine character of $\tilde S(\R)_G$ corresponding to $\tilde\pi$, $\tilde\tau_\lhd$ is as in \eqref{eq:ttaulhd}, and $f=d\tilde\tau/i=d\tau/i$. Note here that $\tilde\tau$ is indeed a character, i.e. a 1-dimensional representation, because $\tilde\pi$ is an extension of $\pi$, and hence $\tilde\tau$ is an extension of $\tau$, the latter being a genuine character of $S(\R)_G$.

Here as before $\tilde G_z=[G \rtimes A]_z$, but recall that we now have $A=A^{[\varphi],[z]}=\<a\>$. Since both $N_{\tilde G_z(\R)}(S)$ and $\tilde S(\R)$ surject onto $\<a\>$ the natural inclusion induces an isomorphism $\Omega_\R(G_z,S) := N_{G_z(\R)}(S)/S(\R) \to N_{\tilde G_z(\R)}(S)/\tilde S(\R)$. An element $u \in \Omega_\R(G_z,S)$ contributes to the above sum if and only if it satisfies $u^{-1}\tilde\delta u \in \tilde S(\R)$, which is equivalent to $u^{-1}\tilde\delta u \tilde\delta^{-1} \in \tilde S(\R) \cap G_z(\R) = S(\R)$. Note that $\tilde\delta \in \tilde S(\R)$ normalizes $S$, hence acts on $\Omega_\R(G_z,S)$, and the condition on $u$ is equivalent to $u=\tilde\delta u \tilde\delta^{-1}$ in $\Omega_\R(G_z,S)$. The character formula now becomes
\[ \tx{tr}(\tilde\pi(\tilde\delta)) = (-1)^{q(G_z)}\sum_{u \in \Omega_\R(G_z,S)^{\tilde\delta}}\frac{\tilde\tau_\lhd(u^{-1}\tilde\delta u)}{\det(1-u^{-1}\tilde\delta u|\mf{u}_{f})}.\]
Since the extension $\tilde\rho$ of $\rho$ is linked to the extension $\tilde\tau$ of $\tau$, we will write $\<\tilde\tau_\lhd,\tilde s\>$ in place of $\tilde\rho(\tilde s)$. This leads to 
\begin{equation} \label{eq:el0}
\tx{tr}(\pi \boxtimes \rho^\vee)^\tx{can}(\tilde\delta,\tilde s^{-1}) = (-1)^{q(G_z)}\<\tilde\tau_\lhd,\tilde s\>\sum_{u \in \Omega_\R(G_z,S)^{\tilde\delta}}\frac{\tilde\tau_\lhd(u^{-1}\tilde\delta u)}{\det(1-u^{-1}\tilde\delta u|\mf{u}_{f})}.
\end{equation}
As already remarked, the right hand side does not depend on the choice of extension $\tilde\tau$ of $\tau$.

Next we consider the sum in \eqref{eq:lhs_sliced}. It runs over the set of those $\pi' \in \Pi_\varphi(G_z)$ whose isomorphism class is stable under one, hence any, element of $[G \rtimes a]_z(\R)$. One such element is $\tilde\delta$. If $(S',\tau')$ represents the Harish-Chandra parameter of some $\pi' \in \Pi_\varphi(G_z)$, then $\pi'$ is stable under conjugation by $\tilde\delta$ if and only if the $G_z(\R)$\-conjugacy class of $(S',\tau')$ is. On the other hand, in order for the character of $\pi'$ to not vanish at $\tilde\delta$, the condition is that some $G_z(\R)$\-conjugate of $f'=d\tau'/i$ commutes with $\tilde\delta$. But if we choose $(S',\tau')$ in its $G_z(\R)$\-conjugacy class so that $f'=d\tau'/i$ commutes with $\tilde\delta$, then $S'=\tx{Cent}(f',G_z)$ is normalized by $\tilde\delta$ and moreover $\tau'\circ\tx{Ad}(\tilde\delta)$ is a genuine character of $S'(\R)_G$ that is both $\Omega_\R(S',G_z)$\-conjugate to $\tau'$ and has the same regular differential $if'$ as $\tau'$, hence is equal to $\tau'$. 

Thus, the sum in \eqref{eq:lhs_sliced} runs over those $\pi'$ whose Harish-Chandra parameter has a representative $(S',\tau')$ that is fixed by $\tx{Ad}(\tilde\delta)$, equivalently $f'=d\tau'/i$ is fixed by $\tx{Ad}(\tilde\delta)$. 

In the derivation of \eqref{eq:el0} we already fixed one such $\pi$ with Harish-Chandra parameter $(S,\tau)$ and set $f=d\tau/i$ and $\tilde S=\tx{Cent}(f,\tilde G_z)$. The set of of elements of $\mf{g}_z^*(\R)$ that are fixed by $\tx{Ad}(\tilde\delta)$ lies in $\mf{s}^*(\R)$ and equals $\mf{s}^*(\R)^{\tilde\delta}$. The Harish-Chandra parameter of any $\pi' \in \Pi_\varphi(G_z)$ can be represented by $(S,w\tau)$ for some $w \in \Omega(S,G_z)$. Then $wf \in \mf{s}^*(\R)^{\tilde\delta}$ if and only if $w \in \Omega(S,G_z)^{\tilde\delta}$ by the argument given above. We conclude that the set
\[ \{(S,w\tau)|w \in \Omega_\R(S,G_z)^{\tilde\delta} \lmod \Omega(S,G_z)^{\tilde\delta} \} \]
is a set of representatives for the Harish-Chandra parameters of those elements of $\Pi_\varphi(G_z)$ whose isomorphism class is stable under the coset $[G \rtimes a]_z(\R)$ and whose character doesn't vanish on $\tilde\delta$. 

Since $\tilde S$ is generated by $S$ and $\tilde\delta$, it is normalized by any $w \in \Omega(S,G_z)^{\tilde\delta}$. Therefore the formula \eqref{eq:el0} is valued for the representation with parameter $(S,w\tau)$. With this \eqref{eq:lhs_sliced} becomes
\[ (-1)^{q(G)} \sum_{w \in \Omega_\R(S,G_z)^{\tilde\delta} \lmod \Omega(S,G_z)^{\tilde\delta}} \<\tilde\tau_{w,\lhd},\tilde s\>\sum_{u \in \Omega_\R(G_z,S)^{\tilde\delta}}\frac{\tilde\tau_{w,\lhd}(u^{-1}\tilde\delta u)}{\det(1-u^{-1}\tilde\delta u|\mf{u}_{wf})}, \]
where $\tilde\tau_w$ is an arbitrary extension to $\tilde S(\R)_G$ of the genuine character $\tau_w=w\tau$ of $S(\R)_G$.
Then we can combine the two sums and arrive at the formula
\[ (-1)^{q(G)} \sum_{w \in \Omega(S,G_z)^{\tilde\delta}} \frac{\<\tilde\tau_{w,\lhd},\tilde s\> \cdot \tilde\tau_{w,\lhd}(\tilde\delta)}{\det(1-\tilde\delta |\mf{u}_{wf})}. \]

Our next task is to express the numerator in the sum in terms of an abstract norm of $\tilde\delta$. For this, let $(S_\mf{w},\tau_\mf{w})$ denote an $A$-stable representative of the Harish-Chandra parameter of the unique $\mf{w}$-generic representation in $\Pi_\varphi(G)$, obtained from Proposition \ref{pro:adgen}. Let $g \in G(\C)$ be such that $\tx{Ad}(g)(S_\mf{w},\tau_\mf{w})=(S,\tau)$. Set $\tilde\delta_\mf{w}=\tx{Ad}(g)^{-1}\tilde\delta \in G(\C) \rtimes a$. Since $\tilde\delta$ commutes with $f=d\tau/i$, so does $\tilde\delta_\mf{w}$ commute with $f_\mf{w}=d\tau_\mf{w}/i$. The $a$-stability of $(S_\mf{w},\tau_\mf{w})$ implies that $f_\mf{w}$ is $a$-fixed. Therefore, writing $\tilde\delta_\mf{w}=\delta_\mf{w} \rtimes a$ we conclude that $\delta_\mf{w}$ commutes with $f_\mf{w}$, and hence $\delta_\mf{w} \in S_\mf{w}(\C)$. Let $\gamma_\mf{w}$ be the image of $\delta_\mf{w}$ in the group of $a$-coinvariants of $S_\mf{w}(\C)$.

We claim that $(S_\mf{w},\gamma_\mf{w})$ is an abstract norm for $\tilde\delta$. We only need to check that $\gamma_\mf{w} \in (S_\mf{w})_a(\R)$. The equation $\tx{Ad}(g)f_\mf{w}=f$ and the rationality properties of $f_\mf{w}$ and $f$ imply that $z'_\sigma = g^{-1}z_\sigma\sigma(g) \in S_\mf{w}$. Since $(z_\sigma,\tilde\delta)$ commute, so do $(z'_\sigma,\tilde\delta_\mf{w})$, but this implies $\gamma_\mf{w} \in (S_\mf{w})_a(\R)$, as desired.

To express the numerator in the above sum in terms of $(S_\mf{w},\gamma_\mf{w})$ and the elements $g$ and $z'$, consider  $w \in \Omega(S,G_z)^{\tilde\delta}$ and let $\tilde\tau_w$ be an arbitrary extension to $\tilde S(\R)_G$ of the genuine character $\tau_w=w\tau$ of $S(\R)_G$. Corresponding to $\tilde\tau_w$ there is a representation $\tilde\pi_w \in \Pi_\varphi(\tilde G_z)$, as well as a character $\tilde\rho_w$ of $\tilde S_\varphi^{[z]}$. The term $\<\tilde\tau_{w,\lhd},\tilde s\>$ is the value at $\tilde s$ of $\tilde\rho_w$. From the discussion following Lemma \ref{lem:whit} we recall that the association  between $\tilde\tau_w$ and $\tilde\pi_w$ was made via Duflo's construction, but involves twisting by the character $\epsilon_{\tilde G}$; thus Duflo's construction maps $\tilde\tau_w$ to $\tilde\pi_w \otimes \epsilon_{\tilde G}$. The association between $\tilde\tau_w$ and $\tilde\rho_w$ was made using the following data: the isomorphism
\begin{equation} \label{eq:el1}
\tx{Cent}(\varphi_\mf{w},\hat S_\mf{w} \rtimes A) \to \tilde S_\varphi  
\end{equation}
of Lemma \ref{lem:is}, and the inner twist 
\[ \tx{Ad}(h) : \tilde S_\mf{w} \to \tilde S \]
given by any $h \in G(\C)$ with $\tx{Ad}(h)f_\mf{w}=wf$, as discussed in the proof of Corollary \ref{cor:is}. 

To make this a bit more explicit, set $\tilde S_\mf{w}=\tx{Cent}(f_\mf{w},\tilde G)$. The $a$-stable $f_\mf{w}$ determines an $a$-stable $\C$-Borel subgroup containing $S_\mf{w}$, and such that the corresponding simple roots are imaginary non-compact.  Via \eqref{eq:el1} $\tilde\rho_w$ is identified with a character of $\tx{Cent}(\varphi_\mf{w},\hat S_\mf{w} \rtimes A)$, which via the local Langlands correspondence for disconnected tori was associated to the character $\tilde\tau_{w,\lhd}$ of $\tilde S(\R)$, seen as an inner twist of $\tilde S_\mf{w}$ via $\tx{Ad}(h)$. 

We now use Proposition \ref{pro:luo} and, keeping in mind the twist by $\epsilon_{\tilde G}$, see that
\begin{equation} \label{eq:el2}
\<\tilde\tau_{w,\lhd},\tilde s\>\tilde\tau_{w,\lhd}(\tilde\delta) = (-1)^{q(G)-q(G^a)}\<(\varphi_0^{-1},s_0),(z_w^{-1},\delta_w)\>_\tx{TN}^{-1},  
\end{equation}
where the notation is as follows. We transport $\tilde s$ under \eqref{eq:el1} to an element of  $\tx{Cent}(\varphi_\mf{w},\hat S_\mf{w} \rtimes A)$ which we decompose as $s_0 \rtimes a^{-1} \in \hat S_\mf{w} \rtimes a^{-1}$. We transport $\tilde\delta \in \tilde S(\R)$ under $\tx{Ad}(h)^{-1}$ to an element of $\tilde S_\mf{w}(\C) = S_\mf{w} \rtimes a$ and decompose this element as $\delta_w \rtimes a$. We decompose $\varphi_\mf{w}(x) = \varphi_0(x) \rtimes x$, and have $z_w(\sigma) = h^{-1}z_\sigma\sigma(h)$.

We have $h=n'g$ for a representative $n'$ of $w \in \Omega(S,G_z)^{\tilde\delta}$. The isomorphism $\tx{Ad}(g) : \Omega(S_\mf{w},G) \to \Omega(S,G_z)$ translates the action of $\tilde\delta$ on its target to the action of $\tilde\delta_w$ on its source. But the action of $\tilde\delta_w$ coincides with the action of $a$. Therefore we can write $h=gn$ with $n \in N(S_\mf{w},G)(\C)$ lifting $\tx{Ad}(g)^{-1}w \in \Omega(S_\mf{w},G)^a$. According to \cite[Proposition 3.4.2(8), Proposition D.0.3(8)]{KalLLCD} we can choose $n \in N(S_\mf{w},G)(\C)^a$. This gives a finer choice of $h$, with respect to which we can express $z_w$ and $\delta_w$ more precisely, namely $\delta_w=w^{-1}\delta_\mf{w}w$ and $z_w(\sigma)=w^{-1}z
'(\sigma)w \cdot \sigma(n)^{-1}n$.

The final formula for \eqref{eq:lhs_sliced} becomes
\begin{equation} \label{eq:lhs_sliced_final}
(-1)^{q(G^a)} \sum_{w \in \Omega(S_\mf{w},G)^a} \frac{\<(\varphi_0^{-1},s_0),((n^{-1}z'_\sigma\sigma(n))^{-1},\delta_\mf{w}^w)\>_\tx{TN}^{-1}}{\det(1-\tilde\delta_\mf{w} |\mf{u}_{wf_\mf{w}})}.  
\end{equation}

\subsection{The right hand side: sliced version} \label{sub:rhs_sliced}

In this subsection we will develop an alternative formula for the right-hand side of \eqref{eq:main1} under the assumption that $\tilde\delta$ is strongly regular elliptic. This formula is \eqref{eq:rhs_sliced_final} below. 

We begin by recalling the right-hand side of \eqref{eq:main1}. Fixing a strongly regular elliptic element $\tilde\delta \in [G \rtimes a]_z(\R)$, it is given by
\begin{equation} \label{eq:rhs_sliced}
\sum_{\gamma \in H(\R)^\tx{sr}/\tx{st}} \Delta_\tx{KS}''(\gamma_1,\tilde\delta) S\Theta_{\varphi_1}(\gamma_1). 
\end{equation}
Here the sum runs over the set of stable classes of strongly regular semi-simple elements $\gamma$ in $H(\R)$ and $\gamma_1 \in H_1(\R)$ is an arbitrary lift of $\gamma$. The factor $\Delta''_\tx{KS}$ is the factor defined in \cite[\S4.9, equation (4.3)]{KalLLCD}, but modified so that the contribution of $\Delta_{IV}$ is inverted.

We now fix $f \in \mf{g}_z^*(\R)$ in the stable class corresponding to the $L$-packet on $G_z$ that commutes with $\tilde\delta$, which exists by Lemma \ref{srell}. As before we set $\tilde S=\tx{Cent}(f,\tilde G_z)$ and $S = \tilde S \cap G_z$, and have $\tilde\delta \in \tilde S(\R)$. We also choose an elliptic maximal torus $S^H \subset H$ and fix an element $f^H \in (\mf{s}^H)^*(\R)$ in the stable conjugacy class corresponding to the $L$-packet on $H$.

As in the previous subsection we use the $a$-admissible Whittaker datum $\mf{w}$ to obtain, via Proposition \ref{pro:adgen}, an $a$-stable representative $(S_\mf{w},\tau_\mf{w})$ of the Harish-Chandra parameter of the unique $\mf{w}$-generic member of $\Pi_\varphi(G)$, set $f_\mf{w}=d\tau_\mf{w}/i$, fix $g \in G(\C)$ such that $\tx{Ad}(g)f_\mf{w}=f$, set $\tilde\delta_\mf{w}=\tx{Ad}(g)^{-1}\tilde\delta$, and observe that $\tilde\delta_\mf{w}=\delta_\mf{w} \rtimes a \in S_\mf{w}(\C) \rtimes a$ and that the image $\gamma_\mf{w} \in (S_\mf{w})_a(\C)$ of $\delta_\mf{w}$ is an $\R$-point, so that $(S_\mf{w},\gamma_\mf{w})$ is an abstract norm for $\tilde\delta$. Choose finally an admissible isomorphism $\eta : S^H \to (S_\mf{w})_a$ sending $f^H$ to $f_\mf{w}$.

The next lemma will examine the summation index of \eqref{eq:rhs_sliced} and prepare for the evaluation of $\Delta_{KS}''$.

\begin{lem} \label{lem:allnorms}
     Let $\gamma=\eta^{-1}(\gamma_\mf{w})$. A set of representatives for the stable classes of elements of $H(\R)$ that are related to $\tilde\delta$ is given by $\{w\gamma|w \in \Omega(S^H,H) \lmod \Omega(S_\mf{w},G)^a\}$, where we are letting $\Omega(S_\mf{w},G)^a$ act on $S^H$ via $\eta$.
\end{lem}
\begin{proof}
  Since all elements of $H(\R)$ that are related to $\tilde\delta$ are elliptic and strongly regular, they are conjugate under $H(\R)$ to elements of $S^H(\R)$. Moreover, the uniqueness of the stable $G^{a,\circ}$-class of $(S_\mf{w},\gamma_\mf{w})$ (\cite[Lemma 3.3.1(2)]{KalLLCD}) implies that the set of all elements of $S^H(\R)$ that are related to $\tilde\delta$ is given by $\Omega(S_\mf{w},G)^a \cdot \gamma$, where the action of $\Omega(S_\mf{w},G)^a$ on $(S_\mf{w})_a$ is transported to an action on $S^H$ via the admissible isomorphism $\eta$. Finally, any two such elements are stably conjugate in $H$ if and only if they are in the same $\Omega(S^H,H)$-orbit.
\end{proof}

We fix $\gamma =\eta^{-1}(\gamma_\mf{w})$ as in the Lemma \ref{lem:allnorms}, and choose further a lift $\gamma_1 \in H_1(\R)$ of $\gamma$. The action of $\Omega(S_\mf{w},G)^a$ on $(S_\mf{w})_a$ transported to $S^H$ via $\eta$ lifts to an action on the preimage $S^{H_1} \subset H_1$ of $S^H$, so we can form $w\gamma_1w^{-1}$ for any $w \in \Omega(S',G)^a$. According to Lemma \ref{lem:allnorms}, \eqref{eq:rhs_sliced} becomes

\[ \sum_{w \in \Omega(S^H,H) \lmod \Omega(S_\mf{w},G)^a} \Delta_\tx{KS}''(w\gamma_1w^{-1},\tilde\delta) S\Theta_{\varphi_1}(w\gamma_1w^{-1}).  \]
Applying Harish-Chandra's formula for the stable character of a discrete series $L$-packet (see e.g. \cite[\S5.3]{DPR}), the above expression becomes
\[ (-1)^{q(H)}\!\!\!\!\!\sum_{w \in \Omega(S^H,H) \lmod \Omega(S_\mf{w},G)^a}\!\!\!\!\! \Delta_\tx{KS}''(w\gamma_1w^{-1},\tilde\delta) \sum_{u \in \Omega(S^H,H)} \frac{\tau_{H_1,\lhd} (uw\gamma_1w^{-1}u^{-1})}{\det(1-uw\gamma w^{-1}u^{-1} |\mf{u}_{f^H})}. \]
We are using here the genuine character $\tau_{H_1}$ of $S^{H_1}(\R)_G$ that gives rise to the $L$-packet on $H_1(\R)$. We combine the two sums and use the stable invariance of $\Delta_\tx{KS}''$ in the first argument to see that the above expression, and thus \eqref{eq:rhs_sliced}, equals
\begin{equation} \label{eq:rs1}
(-1)^{q(H)}\sum_{w \in \Omega(S_\mf{w},G)^a} \Delta_\tx{KS}''(w\gamma_1w^{-1},\tilde\delta)  \frac{\tau_{H_1,\lhd}(w\gamma_1w^{-1})}{\det(1-w\gamma w^{-1} |\mf{u}_{f^H})}.  
\end{equation}
To continue, we must evaluate the transfer factor 
\[ \Delta''_\tx{KS} = \epsilon \cdot \Delta_I^{-1} \cdot \Delta_{II} \cdot \Delta_{III}^{-1} \cdot \Delta_{IV}^{-1}.
\] 
It is a normalization of the Kottwitz-Shelstad twisted transfer factor of \cite{KS99}, \cite{KS12}.  The term $\Delta_I$ is defined in \cite[\S3.4]{KS12} and called $\Delta_I^\tx{new}$ there. Since we are working over $\R$, it happens to coincide with the term defined in \cite[\S4.2]{KS99}, as discussed in \cite[Proposition 3.5.2(2)]{KS12}. The terms $\Delta_{II}$ and $\Delta_{IV}$ are defined in \cite[\S4.3,\S4.5]{KS99}. The term $\Delta_{III}$ is defined in \cite[\S4.10,\S4.11]{KalLLCD} and denoted by $\Delta_{III}^\tx{new}$ there. It is a specific refinement of the relative term $\Delta_{III}$ defined in \cite[\S4.4]{KS99} that takes into account the inner form data.

The next lemmas will compute $\Delta_{KS}''(w\gamma_1w^{-1},\tilde\delta)$. Recall that the pieces $\Delta_I$, $\Delta_{II}$, and $\Delta_{IV}$ of the transfer factor depend only on $\gamma \in H(\R)$  and not on the lift $\gamma_1  \in H_1(\R)$. Only the piece $\Delta_{III}$ depends on $\gamma_1$. These various pieces depend on further data, which we now fix.

Choose $R_{f_\mf{w}}^+$-based $a$-data for $R(S_\mf{w},G)$: $a_\alpha:=i$ for any $\alpha \in R_{f_\mf{w}}^+$. These data are $a$-stable, since $f_\mf{w}$ is $a$-fixed. Choose further $R_{f_\mf{w}}^-$-based $\chi$-data for $R_\tx{res}(S_\mf{w},G)$: $\chi_\alpha(z):=\tx{sgn}_\C(z):=z/|z|$ for any $\alpha \in R_{f_\mf{w}}^-$.

In our set-up, $(S_\mf{w},\gamma_\mf{w})$ is an abstract norm of $\tilde\delta$, $\eta: S^H \to (S_\mf{w})_a$ maps  $\gamma$ to $\gamma_\mf{w}$ and $f^H$ to $f_\mf{w}=\tx{Ad}(g)^{-1}f$, $\delta_\mf{w} \in S_\mf{w}(\C)$ projects to $\gamma_\mf{w}$, and we have $g \in G(\C)$ such that $g \tilde\delta_\mf{w} g^{-1}=\tilde\delta$, where $\tilde\delta_\mf{w} = \delta_\mf{w} \rtimes a$. Note that $\tx{inv}(\tilde\delta_\mf{w},\tilde\delta)=\tx{inv}(f_\mf{w},f)$, both being represented by $z'_\sigma = g^{-1}z_\sigma\sigma(g)$. These data will be used in the computation of $\Delta''_{KS}(\gamma_1,\tilde\delta)$.

But we need to compute more generally $\Delta''_{KS}(w\gamma_1w^{-1},\tilde\delta)$ for $w \in \Omega(S',G)^a$. For this, it will be more convenient to replace the $\mf{w}$-subscript on $\tilde\delta_\mf{w}$ with a prime superscript, i.e. write $\tilde\delta':=\tilde\delta_\mf{w}$. Set $\gamma_w=w\gamma w^{-1}$, $\gamma'_w=w\gamma'w^{-1}$, and $\tilde\delta'_w=w\tilde\delta'w^{-1}=w\delta'w^{-1} \rtimes a$. We use the same admissible isomorphism $\eta : S^H \to (S')_a$, which still maps $f^H$ to $f'=\tx{Ad}(g)^{-1}f$, and now maps $\gamma_w$ to $\gamma'_w$. Also $\delta'_w$ projects to $\gamma'_w$. The element $g$ must now be replaced by $gn^{-1}$, where $n \in N_G(S')^a$ maps to $w$. Note that $\tx{inv}(\tilde\delta_w',\tilde\delta)\neq\tx{inv}(f',f)$, because $\tx{inv}(f',f)$ is still represented by $z'_\sigma$ as above, while $\tx{inv}(\tilde\delta_w',\tilde\delta)$ is represented by $nz'_\sigma\sigma(n)^{-1}=w(\tx{inv}(f',f)) \cdot \tx{inv}(\tilde\delta'_w,\tilde\delta')$.

\begin{lem} \label{lem:d24}
\[ \epsilon\frac{\Delta_{II}(\gamma_w,\tilde\delta)}{\Delta_{IV}(\gamma_w,\tilde\delta)} = (-1)^{q(G^a)-q(H)} \cdot \frac{\det(1-\gamma_w|\mf{u}_{f^H})}{\det(1-\tilde\delta_w'|\mf{u}_{f'})}. \]
\end{lem}
\begin{proof}
Let $r_G$ denote the number of roots of $G$, $r_H$ the number of roots of $H$, $r_{G,a}$ the number of those $a$-orbits in the set of roots in $G$ that are not of type R3, and $A^G$ is the maximal split torus in the quasi-split group $G$.

By construction $\Delta_{II}=\Delta_{II}^G/\Delta_{II}^H$ and $\Delta_{IV}=\Delta_{IV}^G/\Delta_{IV}^H$. We begin on the $H$-side, where 
\[ \Delta_{IV}^H(\gamma_w)=|\det(1-\gamma_w|\mf{h}/\mf{s}_H)|^{1/2}\quad\tx{and}\quad \Delta_{II}(\gamma_w)=\!\!\!\!\!\prod_{\alpha \in R(S^H,H)/\Gamma}\!\!\!\!\!\chi_\alpha\left(\frac{\alpha(\gamma_w)-1}{a_\alpha}\right). \]
The set $R_{f_H}^+$ is a set of representatives for $R(S^H,H)/\Gamma$ and the choices of $a$-data and $\chi$-data imply
\[ \Delta_{II}^H(\gamma_w)=\prod_{\alpha \in R_{f_H}^+}\tx{sgn}_\C\left(\frac{\alpha(\gamma_w)-1}{i}\right)^{-1} = i^{-r_H/2} \cdot \prod_{\alpha \in R_{f_H}^+}\tx{sgn}_\C(1-\alpha(\gamma_w))^{-1}. \]
Combining this with Corollary \ref{cor:pn_d4} applied to the connected case this leads to 
\[ \frac{\Delta_{IV}^H(\gamma_w)}{\Delta_{II}^H(\gamma_w)} = i^{r_H/2} \cdot \det(1-\gamma_w|\mf{u}_{f^H}). \]
Turning to the $G$-side, 

we have
\[ \Delta_{IV}^G(\tilde\delta) = |\det(1-\tilde\delta|\mf{g}/\mf{s})|^{1/2} = |\det(1-\tilde\delta_w'|\mf{g}/\mf{s'})|^{1/2}, \]
which Lemma \ref{lem:pn_detroot} and Corollary \ref{cor:pn_d4} equate with
\[ \prod |1-\alpha(\tilde\delta^{|\mc{O}|})| = \prod |1-\alpha'(\tilde\delta_w'^{|\mc{O'}|})|. \]
The first product runs over the set of $\tx{Ad}^*(\tilde\delta)$ orbits $\mc{O}$ in $R_f^+$, while the second runs over the set of $\tx{Ad}^*(\tilde\delta')$ orbits $\mc{O}'$ in $R_{f'}^+$, and $\alpha \in \mc{O}$ resp. $\alpha' \in \mc{O'}$ are arbitrary choices. We recall again that $\tilde\delta^{|\mc{O}|}$ is an element of $\tilde S_\alpha$ that may not lie in $S$, and hence $\alpha$ is also taken to be the natural extension of the root $\alpha$ to a homomorphism $\tilde S_\alpha \to \mb{G}_m$, discussed just before Lemma \ref{lem:pn_detroot}.

Using the element $\tilde\delta_w'$ we can describe this more explicitly as follows. By assumption $\tilde\delta_w'=\delta_w' \rtimes a \in S' \rtimes a$, where $S'$ is a maximal $\R$-torus of $G$ that is $a$-stable, and contained in an $a$-stable Borel $\C$-subgroup of $G$. We are assuming that the automorphism $a$ preserves a pinning of $G$. Since any two $a$-stable $\C$-Borel pairs of $G$ are conjugate under $G^{a,\circ}$, we conclude that there exists a set of simple root vectors in $R(S',G)$ that is permuted by $a$. This in turn implies that, if some power of $a$ fixes a root $\alpha$, it acts on the root line $\mf{g}_\alpha$ by multiplication by either $+1$ or $-1$, and the latter happens only if $\alpha=\alpha_1+\alpha_2$ for two roots $\alpha_1,\alpha_2$ that are switched by this power of $a$. In the terminology of \cite{KS99}, the root $\alpha$ is of ``type R3'' and the roots $\alpha_1,\alpha_2$ are of ``type R2''. We also recall that an orbit of ``type R1'' is one in which the sum of any two members is not a root, and where no member is fixed by a non-trivial power of $a$. 

Thus, if $\mc{O'}$ is an orbit of type R1 or R2, then $\tilde\delta'^{|\mc{O'}|}$ lies in $S$ and equals $\delta_w' \cdot a(\delta_w') \cdots a^{|\mc{O'}|-1}(\delta_w')$, and therefore $\alpha'(\tilde\delta_w'^{|\mc{O'}|})$ is the value of the root $\alpha'$ at this element; this is denoted by $N\alpha'(\delta_w')$ in \cite{KS99}. On the other hand, if $\mc{O'}$ is of ``type R3'', then $\tilde\delta_w'^{|\mc{O'}|}$ does not lie in $S$, and instead equals $\delta_w' \cdot a(\delta_w') \cdots a^{|\mc{O'}|/2-1}(\delta_w') \rtimes a^{|\mc{O'}|/2}$, and the value of $\alpha'$ on this element is $-\alpha'(\delta_w' \cdot a(\delta_w') \cdots a^{|\mc{O'}|/2-1}(\delta_w'))$, which is denoted by $-N\alpha'(\delta_w')$ in \cite{KS99}. In this notation
\[ \Delta_{IV}^G(\tilde\delta) = \prod |1-N\alpha'(\delta_w')| \cdot \prod |1+N\alpha'(\delta_w')|, \]
where the first product runs over the set of orbits of $a$ in $R_{f'}^+$ of type R1 or R2, and the second product runs over the set of orbits of type R3. On the other hand, 
\[ \Delta_{II}^G(\tilde\delta) = \prod \tx{sgn}_\C\left(\frac{N\alpha'(\delta_w')-1}{i}\right)^{-1} \cdot \prod \tx{sgn}_\C(N\alpha'(\delta_w')+1)^{-1},\]
where the products run over the same sets. Reversing the arguments above we conclude
\[ \frac{\Delta_{IV}^G(\tilde\delta)}{\Delta_{II}^G(\tilde\delta)} = i^{r_{G,a}/2} \cdot \det(1-\tilde\delta_w'|\mf{u}_{f'}), \]
where $r_{G,a}$ denotes the number of $a$-orbits of roots of $G$ that are not of type R3. 

In summary, we have shown
\[ \frac{\Delta_{II}(\gamma_w,\tilde\delta)}{\Delta_{IV}(\gamma_w,\tilde\delta)} = i^{r_H/2-r_{G,a}/2} \cdot \frac{\det(1-\gamma_w|\mf{u}_{f^H})}{\det(1-\tilde\delta_w'|\mf{u}_{f'})}. \]
Next consider the factor $\epsilon=\epsilon(X^*(T^G)_\C^a-X^*(T^H)_\C,\Lambda)$, where $T^G$ is a maximal torus that belongs to an $a$-stable $\R$-pinning of $G$, and $T^H$ is maximal a maximal torus that belongs to $\R$-pinning of $H$. We let $A^G$ and $S^G$ denote the maximal split and maximal anisotropic torus of $T^G$. Since $a$ acts rationally on $T^G$, both $A^G$ and $S^G$ are $a$-stable. We have the $a$-stable decomposition $X^*(T^G)_\C=X^*(A^G)_\C \oplus X^*(S^G)_\C$. The first summand is isotypic for the trivial Galois representation, and the second summand is isotypic for the sign character. Using $\epsilon(1/2,\mathbf{1},\Lambda)=1$ and $\epsilon(1/2,\tx{sgn},\Lambda)=i$ (cf. \cite[(3.2.4)]{TateCor}) we see
\[ \epsilon(X^*(T^G)_\C^a,\Lambda) = i^{d-\dim(A^G_a)}, \]
where $d=\dim(T^G_a)$. The same calculation applied to $H$ (and a trivial automorphism) results in
\[ \epsilon(X^*(T^H)_\C,\Lambda) = i^{d-\dim(A^H)}, \]
where we are using that $\dim(T^H)=\dim(T^G_a)$. Therefore
\[ \epsilon = i^{\dim(A^H)-\dim(A^G_a)}. \]
Putting all together we obtain
\[ \epsilon\frac{\Delta_{II}(\gamma_w,\tilde\delta)}{\Delta_{IV}(\gamma,\tilde\delta)} = i^{r_H/2-r_{G,a}/2}i^{\dim(A^H)-\dim(A^G_a)} \cdot \frac{\det(1-\gamma_w|\mf{u}_{f^H})}{\det(1-\tilde\delta_w'|\mf{u}_{f'})}. \]

The Iwasawa decomposition $\tx{Lie}(G)=\mf{a} \oplus \mf{n} \oplus \mf{k}$ shows that $2q(G)=\dim(\mf{a})+\dim(\mf{n})=\dim(A^G)+r_G/2$. We apply this argument applied to $H$ and $G^a$ to see $2q(H)=\dim(A^H)+r_H/2$ and $2q(H)=\dim(A^G_a)+r_{G,a}/2$, the latter also using Lemma \ref{lem:dimrga} below. Therefore
\[ (-1)^{q(G^a)-q(H)} = i^{\dim(A_a^G)-\dim(A^H)+r_{G,a}/2-r_H/2}.\]
\end{proof}

\begin{lem} \label{lem:dimrga}
  We have $\dim(\mf{n}_0^a) = r_{G,a}/2$.
  \end{lem}
  
  \begin{proof}
  We reduce immediately to the case where $G$ is absolutely simple. Then $a$ (assumed non-trivial) has order $2$ and the root system is of type $A_{2n-1}$, $A_{2n}$, $D_n$, or $E_6$, or has order $3$ and the root system is of type $D_3$. Let $n = \tx{ord}(a) \in \{2,3\}$.
  
  It is clear that $\mf{n}$ decomposes under the action of $a$ according to the orbits of $a$ on the set of positive absolute roots. Such an orbit is either of size either $n$ or $1$. An orbit of size $n$ always contributes to $\dim(\mf{n}_0^a)$ the summand $1$. An orbit of size $1$  contributes a summand of $1$ when the action of $a$ on the corresponding root line is trivial, and a summand of $0$ when that action is non-trivial, necessarily then by the scalar $-1$. It is therefore enough to examine the fixed elements in the set of roots and show that the action of $a$ on the corresponding root line is by $-1$ if and only if the root is of type R3. 
  
  We consider the above Dynkin types separately. In type $A$ a root is fixed if and only if it is a sum of a connected $a$-stable piece of the Dynkin diagram. We induct on the number of nodes in this piece. When the type is $A_{2n-1}$, the piece has an odd number of nodes and the base of induction is a single node. Then the root is simple, and since $a$ fixes the corresponding root vector, it acts by $+1$ on the root line. The inductive step follows from Lemma \ref{lem:jacobi}. Thus all fixed roots contribute $+1$. Also, there are no roots of type R3. When the type is $A_{2n}$, the piece has an even number of nodes and the base of the induction is a piece consisting of two linked nodes. The corresponding root is the sum of two simple roots exchanged by $a$, and since $a$ also swaps their root vectors, it acts by $-1$ on the root line of their sum. The inductive step again follows from Lemma \ref{lem:jacobi}. Thus all fixed roots contribute $0$. Also, they are all of type R3. 
  
  In type non-triality type $D_n$ one examines the Bourbaki plate and finds that a positive root $\alpha$ is $a$-fixed if and only if it is of the form $\alpha=\epsilon_i\pm\epsilon_j$ for $1 \leq i<j<l$. A root $\alpha=\epsilon_i-\epsilon_j$ is the sum of simple roots each of which is $a$, fixed, so $a$ acts on $\mf{g}_\alpha$ by $+1$. The special case $\alpha=\epsilon_i+\epsilon_{l-1}$ is the sum of the case $\epsilon_i-\epsilon_{l-1}$ just treated, and the two simple roots $\epsilon_{l-1}-\epsilon_l$ and $\epsilon_{l-1}+\epsilon_l$, and Lemma \ref{lem:jacobi} implies that again $a$ acts by $+1$ on $\mf{g}_\alpha$. Finally, the case $\alpha=\epsilon_i+\epsilon_j$ for $j<l-1$ is the sum of $\epsilon_i-\epsilon_{l-1}$ and $\epsilon_j+\epsilon_{l-1}$, both of which are $a$-fixed and on both of whose root lines $a$ acts by $+1$, hence $a$ acts by $+1$ on $\mf{g}_\alpha$. At the same time, type $D_n$ has no roots of type R3.
  
  In triality type $D_4$, again using the Bourbaki plate, if $\alpha,\beta,\gamma,\delta$ are simple roots such that $a$ fixes $\beta$ and permutes $\{\alpha,\gamma,\delta\}$ we see that the only $a$-fixed positive roots are $\beta$, $\alpha+\beta+\gamma+\delta$, and $\alpha+2\beta+\gamma+\delta$. By assumption $a$ acts by the scalar $+1$ on the root line for the simple root $\beta$. By Lemma \ref{lem:jacobi} it also acts by $+1$ on the root line for the root $\alpha+\beta+\gamma+\delta$. If $X,Y$ are non-zero vectors in the root lines for $\beta$ and $\alpha+\beta+\gamma+\delta$, then $[X,Y]$ is a non-zero vector in the root line for $\alpha+2\beta+\gamma+\delta$. Then $a(X)=X$ and $a(Y)=Y$ implies $a([X,Y])=[X,Y]$.

  In type $E_6$, one examines that an $a$-fixed root is either simple, or the sum of two $a$-fixed roots, or of the type in Lemma \ref{lem:jacobi}. In all of these cases, we see that $a$ acts by $+1$ on the root line. Again there are no roots of type R3.
  \end{proof}
  
  \begin{lem} \label{lem:jacobi}
    Assume the absolute root system of $G$ is simply laced.
    \begin{enumerate}
      \item Let $\alpha,\beta,\gamma$ be positive roots such that $\alpha+\beta,\beta+\gamma,\alpha+\beta+\gamma$ are roots, but $\alpha+\gamma$ is not. Assume that $a\alpha=\gamma$, $a\gamma=\alpha$, and $a\beta=\beta$. Let $\epsilon \in \{\pm1\}$ be the scalar by which $a$ acts on $\mf{g}_\beta$. Then $\epsilon$ is also the scalar by which $a$ acts on $\mf{g}_{\alpha+\beta+\gamma}$.
      \item Let $\alpha,\beta,\gamma,\delta$ be positive roots such that $\alpha+\beta+\gamma+\delta$ is a root, but $\alpha+\gamma,\alpha+\delta,\gamma+\delta$ are not roots. Assume that $a\alpha=\gamma$, $a\gamma=\delta$, $a\delta=\alpha$, and $\beta=\beta$. If $a$ acts trivially on the root lines for $\alpha,\beta,\gamma,\delta$, then it also acts trivially on the root line for $\alpha+\beta+\gamma+\delta$.
    \end{enumerate}
    \end{lem}
    \begin{proof}
      (1) Let $0 \neq X \in \mf{g}_\alpha$ and $0 \neq Y \in \mf{g}_\beta$. Then $0\neq a(X) \in \mf{g}_\gamma$ and $0\neq [X,[Y,a(X)]] \in \mf{g}_{\alpha+\beta+\gamma}$. Using the Jacobi identity and the fact that $[X,a(X)]=0$ we see that $a$ fixes the vector $[X,[Y,a(X)]]$.

      (2) Let $0 \neq X \in \mf{g}_\alpha$ and $0 \neq Y \in \mf{g}_\beta$. Then $0 \neq a(X) \in \mf{g}_\gamma$, $0 \neq a^2(X) \in \mf{g}_\delta$, and 
      \[ 0 \neq [[[X,Y],a(X)],a^2(X)] \in \mf{g}_{\alpha+\beta+\gamma+\delta}. \]
      Using the Jacobi identity and the fact that 
      \[ [X,a(x)]=[X,a^2(X)]=[a(X),a^2(X)]=0, \] 
      we see that $a$ fixes $[[[X,Y],a(X)],a^2(X)]$.
    \end{proof}

\begin{lem} \label{lem:d1}
  We have $\Delta_I(\gamma_w,\tilde\delta)=1$.
\end{lem}
\begin{proof}
By construction $\Delta_I(\gamma_w,\tilde\delta)=\<\lambda(S',f'),s_\eta\>$, where $\lambda(S',f') \in H^1(\R,(S')_\tx{sc}^a)$ is the twisted splitting invariant of $S'$ with respect to the $R_{f'}^+$-based $a$-data. But it is shown in \cite[Lemma 6.5.1]{DPR} that $\lambda(S',f')=1$.

Note that this term does not depend on $w$ -- it only depends on the admissible isomorphism and the chosen $a$-data, and since we have chosen to use the same admissible isomorphism $\eta$ and the same $a$-data for all $w$, the result is independent of $w$.

\end{proof}

We now come to the factor $\Delta_{III}$, the only factor that depends on the lift $\gamma_1 \in H_1(\R)$ of $\gamma \in H(\R)$.

\begin{lem} \label{lem:d3}
  \[ \Delta_{III}^\tx{new}(\gamma_{1,w},\tilde\delta) = \tau_{H_1,\lhd}(\gamma_{1,w}) \cdot \<(\varphi_0^{-1},s_0),((n^{-1}z'_\sigma\sigma(n))^{-1},\delta_\mf{w}^w)\>_\tx{TN}.
\]
\end{lem}
\begin{proof}
  The actual proof of this lemma is short and simple, but it requires recalling the construction of the left-hand side, which takes a bit of explanation. We follow the constructions of \cite[\S4.11]{KalLLCD}, according to which 
  \[ \Delta_{III}^\tx{new}(\gamma_{1,w},\tilde\delta) = \<\tx{inv}(\gamma_{1,w},(z,\tilde\delta)),A_0\>_{TN}, \]
  where the notation is as follows. Define $S'_1$ to be the fiber product of $S' \to S'_a \cong S^H \from S^{H_1}$ and consider the element $\delta'_{1,w} := (\delta'_w,\gamma_{1,w}) \in S'_1(\C)$. The automorphism $a \times \tx{id}$ of $S \times S^{H_1}$ induces an automorphism of $S'_1$ which fixes pointwise the kernel of $S'_1 \to S'$, so the endomorphism $(1-a)$ of $S'_1$ descends to a homomorphism $(1-a) : S' \to S'_1$. This homomorphism sends the element $n^{-1}z'_\sigma\sigma(n) \in Z^1(\Gamma,S')$ to the differential of $\delta'_{1,w}$, so the pair $((n^{-1}z'_\sigma\sigma(n))^{-1},\delta_{1,w})$ lies in $Z^1(\Gamma,(1-a) : S' \to S'_1)$. We write $\tx{inv}(\gamma_{1,w},(z,\tilde\delta))$ for its cohomology class.

  The construction of the class $A_0 \in H^1(W_F,(1-a^{-1}):\hat S_1' \to \hat S')$ involves the following diagram, which we have copied from \cite[\S4.11]{KalLLCD} and to which we have added arrows for the various Langlands parameters that are involved in the endoscopic transfer that is our issue at hand here, namely the discrete series parameter $\varphi$ for $G$ and its factorization $\varphi_\mf{w}$ through $S_\mf{w}$ that leads to the character $\tau_{\mf{w},\lhd}$.

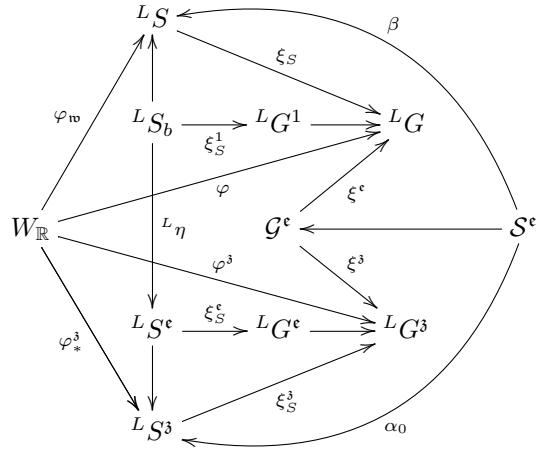
  \[ \xymatrix{
&^LS\ar[drr]^{\xi_S}\\&^LS_{b}\ar[r]_{\xi_S^1}\ar[dd]^{{^L\eta}}\ar[u]&{^LG^1}\ar[r]&{^LG}\\
W_\R\ar[uur]^{\varphi_\mf{w}}\ar[urrr]_\varphi\ar[drrr]^{\varphi^\mf{z}}\ar[ddr]_{\varphi^\mf{z}_*}\ar[ddr]&&\mc{G}^\mf{e}\ar[ur]_{\xi^\mf{e}}\ar[dr]^{\xi^\mf{z}}&&\mc{S}^\mf{e}\ar[ll]\ar@/_3pc/[uulll]_\beta\ar@/^3pc/[ddlll]^{\alpha_0}\\
&^LS^\mf{e}\ar[r]^{\xi_S^\mf{e}}\ar[d]&{^LG^\mf{e}}\ar[r]&{^LG^\mf{z}}\\
&^LS^\mf{z}\ar[urr]_{\xi_S^\mf{z}}
}
\] 
  The subgroup $\mc{S}^\mf{e}$ is described in \cite[\S4.11]{KalLLCD} as the preimage under $\xi^\mf{e}$ of the image of $\xi_S$. Since the image of $\xi^\mf{e}$ contains the image of $\varphi$ by definition, and $\varphi = \xi_S \circ \varphi_\mf{w}$, we see that $\varphi$ lies in the image of $\mc{S}^\mf{e}$, hence factors through an $L$-homomorphism $\varphi_* : W_\R \to \mc{S}^\mf{e}$, which has not been displayed in the diagram due to lack of space. Its composition with $\xi^\mf{z}$ is $\varphi^\mf{z}$, which factors through $\xi_S^\mf{z}$ to give rise to $\varphi_*^\mf{z}$.

  Having described some of the diagram, let us now recall the construction of $A_0$. The map $\beta \times \alpha_0$ gives an $L$-homomorphism $\mc{S}^\mf{e} \to {^L}(S \times S^\mf{z})$, whose composition with the inclusion $\hat S^\mf{e} \to \mc{S}^\mf{e}$ is the diagonal embedding $\hat S^\mf{e} \to \hat S \times \hat S^\mf{z}$. The kernel of the projection ${^L}(S \times S^\mf{z}) \to {^L}S_1^\mf{z}$ is the anti-diagonal embedding of $\hat S^\mf{e}$. Therefore, if we write $\alpha$ for the composition of $\alpha_0$ and the $L$-automorphism of $^LS_\mf{z}$ given by the inversion on $S^\mf{z}$ we obtain an $L$-homomorphism $\mc{S}^\mf{e} \to {^L}(S \times S^\mf{z})$ whose composition with the projection to ${^L}S_1^\mf{z}$ factors through the quotient $\mc{S}^\mf{e} \to W_\R$ and produces an $L$-homomorphism $\tilde a_S : W_\R \to {^L}S_1^\mf{z}$, i.e. an element $a_S \in Z^1(W_\R,\hat S_1^\mf{z})$. At the same time, $\tilde s^\mf{e} = \xi_S(s_0) \rtimes a^{-1}$, and one checks that $(a_S^{-1},s_0) \in Z^1(W_F,(1-a^{-1}) : \hat S_1^\mf{z} \to \hat S)$. The cohomology class of this hypercocycle is $A_0$.

  Having reviewed the construction of $\Delta_{III}^\tx{new}(\gamma_{1,w},\tilde\delta)$ we can now begin the proof of the lemma in earnest. The key point is that the $L$-homomorphism $(\beta \times \alpha) \circ \varphi_* : W_\R \to {^L}(S \times S^\mf{z})$, composed with the projection ${^L}(S \times S^\mf{z}) \to ^LS_1^\mf{z}$, equals $\tilde a_S$. This means that the 1-cocycle $a_S$ lifts to a 1-cocycle $W_\R \to \hat S \times \hat S^\mf{z}$, but tracing through the above diagram we see that the latter is given by the pair $(\varphi_0,(\varphi_0^\mf{z})^{-1})$, where $\varphi_\mf{w}(x)=\varphi_0(x) \rtimes x$ and $\varphi_*^\mf{z}(x)=\varphi_0^\mf{z}(x)\rtimes x$ for $x \in W_\R$. Using the functoriality of the Tate-Nakayama pairing we see that pairing $A_0$ with $\tx{inv}(\gamma_{1,w},(z,\tilde\delta))$ gives the same result as pairing the class of $((\varphi_0^{-1},\varphi_0^\mf{z}),s_S)$ with the class of $((n^{-1}z'_\sigma\sigma(n))^{-1},(\delta'_w,\gamma_{1,w}))$. The first hypercocycle is for the complex of tori are now $S' \to S' \times S^{H_1}$ given by the map $(1-a) : S' \to S'$ and the trivial map $S' \to S^{H_1}$, and the second hypercocycle is in the dual of that complex, namely $\hat S' \times \hat S^{H_1} \to \hat S'$, where the map is $(1-a^{-1})$ on $\hat S'$ and the trivial map on $\hat S^{H_1}$. This leads to the hypercohomology groups breaking up as
  \[ H^1(\Gamma,S' \to S' \times S^{H_1}) = H^1(\Gamma,(1-a): S' \to S') \times H^0(\Gamma,S^{H_1})\]
  and 
  \[ H^1(W_\R,\hat S' \times \hat S^{H_1} \to \hat S') = H^1(W_\R,(1-a^{-1}) :\hat S' \to \hat S') \times H^1(W_\R,\hat S^{H_1}). \]
  The compound pairing breaks up accordingly, as the product of Langlands duality between $H^0(\Gamma,S^{H_1})$ and $H^1(W_\R,\hat S^{H_1})$ and Tate-Nakayama duality between $H^1(\Gamma,S' \to S' \times S^{H_1})$ and $H^1(W_\R,\hat S' \times \hat S^{H_1} \to \hat S')$. Langlands duality pairs $\varphi_0^\mf{z}$ with $\gamma_{1,w}$ and the result is $\tau_{H_1,\lhd}(\gamma_{1,w})$. Tate-Nakayama duality pairs $(\varphi_0^{-1},s_0)$ with $((n^{-1}z_\sigma'\sigma(n))^{-1},\delta'_w)$.
\end{proof}

Applying Lemmas \ref{lem:d24}, \ref{lem:d1}, \ref{lem:d3} we see that \eqref{eq:rs1} equals
\begin{equation} \label{eq:rhs_sliced_final}
(-1)^{q(G^a)}  \sum_{w \in \Omega^a}\frac{\<(\varphi_0^{-1},s_0),((n^{-1}z'_\sigma\sigma(n))^{-1},\delta_\mf{w}^w)\>_\tx{TN}^{-1}}{\det(1-\tilde\delta_w'|\mf{u}_{f'})}.  
\end{equation}
But \eqref{eq:rs1} was just a different expression of \eqref{eq:rhs_sliced}. We conclude that the expressions \eqref{eq:rhs_sliced} and \eqref{eq:rhs_sliced_final} are equal. 

\subsection{Completion of proof}

In \S\ref{sub:lhs_sliced} we showed that the expressions \eqref{eq:lhs_sliced} and \eqref{eq:lhs_sliced_final} are equal for any strongly regular elliptic $\tilde\delta \in [G \rtimes a]_z(\R)$. In \S\ref{sub:rhs_sliced} we showed that the expressions \eqref{eq:rhs_sliced} and \eqref{eq:rhs_sliced_final} are equal for any strongly regular elliptic $\tilde\delta$. Thus, the expressions \eqref{eq:lhs_sliced} and \eqref{eq:rhs_sliced} are equal for such $\tilde\delta$. In other words, the identity \eqref{eq:main1} holds for such $\tilde\delta$. Theorem \ref{uniquethm} then shows that this identity holds for all strongly regular semi-simple $\tilde\delta$, not necessarily elliptic.  Corollary \ref{cor:main1} then implies that Equation \eqref{eq:main}, hence Theorem \ref{thm:main}, holds.

\appendix

  \section{Appendix by Dougal Davis: The core proof of Theorem \ref{thm:rtci-sign}} \label{app:a}

  In this appendix we prove Theorem \ref{thm:rtci-sign} under the assumption that $G$ is a semi-simple simply connected $\R$-group. This assumption is not strictly needed for the proof given here, but it simplifies some of the notation.

  \subsection{Statement}

  Recall from the statement of the problem that $G$ is a  quasi-split semi-simple simply connected group over $\mb{R}$. Changing notation slightly, we will denote by $A \colon G \to G$ an individual automorphism preserving a fixed $\mb{R}$-pinning $\mc{P}$ (rather than a whole group of such automorphisms). We have fixed an $A$-admissible essentially discrete series representation $\pi$ of $G(\mb{R})$, as well as:
\begin{enumerate}
\item An $A$-stable pair $(B, \psi)$ of a real Borel $B \subset G$ with unipotent radical $U$ and a non-degenerate unitary character $\psi \colon U(\mb{R}) \to \mb{C}^\times$, such that the space of Whittaker functionals
\[ \mrm{Wh}(\pi) := \hom_{U(\mb{R})}(\pi^\infty, \mb{C}_\psi) \]
is non-zero (and hence one-dimensional).
\item An $A$-fixed regular elliptic $f \in \mf{g}^*(\mb{R})$, whose $G(\mb{R})$-conjugacy class is the Harish-Chandra parameter of $\pi$.
\end{enumerate}

We let $S \subset G$ be the anisotropic maximal torus centralising $f$ and let $\mf{u}' \subset \mf{g}$ be the Lie algebra of the unipotent radical of the unique complex Borel $B'$ containing $S$ for which $if$ is dominant.

We will need to work with Harish-Chandra modules. To do so, let $\theta \colon G \to G$ be the unique Cartan involution such that $\theta(f) = f$. In particular, $\theta$ preserves $\mf{u}'$ and commutes with $A$. Set $K = G^\theta$, so that $K(\R)$ is a maximal compact subgroup of $G(\R)$. Let $\pi^{\mrm{HC}}$ be the corresponding Harish-Chandra $(\mf{g}, K)$-module of $K$-finite vectors in $\pi$. We recall from Lemmas 3 and 4 of \cite[III.4]{Duflo82} that
\[ H_j(\mf{u}', \pi^\infty)_{if + \rho_f} = H_j(\mf{u}', \pi^{\mrm{HC}})_{if + \rho_f} = 0 \quad \text{if $j \neq q(G)$}\]
and that
\[ H_{q(G)}(\mf{u}', \pi^\infty)_{if + \rho_f} = H_{q(G)}(\mf{u}', \pi^{\mrm{HC}})_{if + \rho_f}\]
is $1$-dimensional. We set $\mrm{H}(\pi) = H_{q(G)}(\mf{u}', \pi^{\mrm{HC}})_{if + \rho_f}$.

If we fix a lift $\tilde{\pi}$ of $\pi$ to a representation of the extended group $G(\mb{R}) \rtimes \langle A \rangle$ then $\langle A \rangle$ will act on the $1$-dimensional vector spaces $\mrm{Wh}(\tilde{\pi})$ and $\mrm{H}(\tilde{\pi})$. Our aim is to prove the following statement:

\begin{pro} \label{pro:signs}
We have
\[ \mrm{H}(\tilde{\pi}) \cong \mrm{Wh}(\tilde{\pi})^* \langle (-1)^{q(G) - q(G^A)}\rangle,\]
where $\langle a\rangle$ denotes a twist of the action of $A$ by $a \in \mb{C}$.
\end{pro}

By construction, the Duflo lift $\tilde{\pi}^D$ (resp.\ Whittaker lift $\tilde{\pi}^W$) is characterised by the triviality of $\mrm{H}(\pi)$ (resp.\ $\mrm{Wh}(\tilde{\pi}^W)$), so Proposition \ref{pro:signs} implies Theorem \ref{thm:rtci-sign} in the case where $G$ is semi-simple and simply connected.

\subsection{Localisation and associated cycles}

We will prove Proposition \ref{pro:signs} by comparing both sides to a third action coming from the theory of associated cycles. Suppose that $M$ is a finite length $(\mf{g}, K \rtimes \langle A \rangle)$-module. Fix a good filtration $F_\bullet M$ compatible with the PBW filtration $F_\bullet U(\mf{g})$ and invariant under the action of $K \rtimes \langle A \rangle$. Then the associated graded $\Gr^F M$ is a finitely generated $(S(\mf{g}), K \rtimes \langle A \rangle)$-module, or equivalently, a $K \rtimes \langle A \rangle$-equivariant coherent sheaf on $(\mf{g}/\mf{k})^*$. Moreover, since $M$ has finite length and is hence locally finite over $Z(U(\mf{g}))$, $\Gr^FM$ is supported set-theoretically on the $K$-nilpotent cone $\mc{N}_K^* \subset (\mf{g}/\mf{k})^*$. Thus, we get a class in the Grothendieck group
\[ [\Gr^F M] \in \mrm{K}_{K \rtimes \langle A \rangle}(\mc{N}_K^*).\]
It is easy to see that $[\Gr^F M]$ is independent of the choice of good filtration $F_\bullet M$; we will write $\mrm{AC}(M) = [\Gr^F M]$ and call it the \emph{associated cycle} of $M$. (A word of warning: in general, this is slightly different to the usual notion of associated cycle of a representation, which is roughly speaking valued in the equivariant Grothendieck group of the support of $\Gr^FM$.) This all works $K \rtimes \langle A \rangle$-equivariantly.

Let's compute the associated cycle for our essentially discrete series $\tilde{\pi}^{\mrm{HC}}$. It will be convenient to use the Beilinson-Bernstein localisation picture for this purpose.

Recall the abstract/universal Cartan group $H$ of $G$. We regard this simply as a complex torus. By definition, if we fix any complex Borel and maximal torus $B_\C \supset T_\C$, then we have a canonical identification $H \cong B_\mb{C}/[B_\mb{C}, B_\mb{C}] \cong T_\mb{C}$. Given $(B_\mb{C}, T_\mb{C})$, we equip $H$ with sets of positive roots and coroots
\[ R_{\mathit{abs}}^+ =  - R(T_\mb{C}, B_\mb{C}) \subset X^*(T_\mb{C}) \cong X^*(H)\]
and
\[ \check{R}_{\mathit{abs}}^+ = \{\check\alpha \in X_*(T_\mb{C}) \cong X_*(H)\mid \alpha \in R_{\mathit{abs}}^+ \}.\]
It is easy to see that $R_{\mathit{abs}}^+$ and $\check{R}_{\mathit{abs}}^+$ are independent of the choice of $B_\mb{C} \supset T_\mb{C}$. We emphasise that the negative sign is important in our convention; it ensures that dominant weights correspond to semi-ample line bundles on the flag variety.

Let $\mc{B}$ denote the complex flag variety of the complex reductive group underlying $G$ and let $p\colon \tilde{\mc{B}} \to \mc{B}$ be the projection from the base affine space, a tautological torsor under $H$. We set
\[\tilde{\mc{D}} = p_*(\mc{D}_{\tilde{\mc{B}}})^H,\]
where $\mc{D}_{\mc{\tilde B}}$ is the sheaf of differential operators on $\mc{\tilde B}$. This is a sheaf of rings on $\mc{B}$ locally isomorphic to $\mc{D}_{\mc{B}} \otimes S(\mf{h})$. For $\lambda \in \mf{h}^*$, we set
\[ \mc{D}_\lambda = \tilde{\mc{D}}\otimes_{S(\mf{h}), \lambda - \rho}\mb{C}_\lambda,\]
where $\rho \in \mf{h}_\mb{C}^* = X^*(H) \otimes \mb{C}$ is then defined as
\[ \rho = \frac{1}{2}\sum_{\alpha \in R_{\mathit{abs}}^+} \alpha.\]

Now, the $G$-action on $\tilde{\mc{B}}$ gives rise to a tautological algebra homomorphism $U(\mf{g}) \to \tilde{\mc{D}}$. The $K$-equivariant version of the Beilinson--Bernstein localisation theorem \cite{BB81} \cite[\S 3.3]{BB93} asserts that, when $\lambda$ is regular (i.e.\ $\langle \lambda, \check\alpha \rangle \neq 0$ for $\check{\alpha} \in \check{R}^+_{\mathit{abs}}$) and integrally dominant (i.e.\ $\langle \lambda, \check\alpha\rangle \not\in \mb{Z}_{<0}$ for $\check{\alpha} \in \check{R}^+_{\mathit{abs}}$), taking global sections defines an equivalence of categories
\[ \Gamma \colon \Mod(\mc{D}_\lambda, K) \to \Mod(\mf{g}, K)_{\chi_\lambda},\]
between the category of strongly $K$-equivariant quasi-coherent $\mc{D}_\lambda$-modules and the category of $(\mf{g}, K)$-modules with infinitesimal character $\chi_\lambda$ associated to $\lambda$ under the Harish-Chandra isomorphism. This also works with the $A$-action in place as long as $\lambda$ is $A$-fixed.

In the case of the discrete series $\pi$, this works out as follows. Recall that we have our regular elliptic element $f \in \mf{g}^*(\mb{R})$, with associated anisotropic maximal torus $S$ and complex Borel $B'$, for which $if$ is negative on the coroots attached to roots in $B'$. This choice of Borel defines an identification between $\mf{s}$ and $\mf{h}$ mapping $if$ to an element $\lambda \in \mf{h}^*$. According to our sign conventions for the abstract Cartan, the element $\lambda$ is regular dominant in the sense above. We then have that
\[ \pi^\mrm{HC} = \Gamma(\mc{M}),\]
where $\mc{M}$ is an irreducible $(\mc{D}_\lambda, K)$-module supported on the (closed) $K$-orbit $Q$ of the Borel subgroup $B'$.
\begin{rmk}
Since $\pi$ is a discrete series, $\lambda - \rho$ is actually the highest weight of a finite dimensional $G$-representation, given as global sections of a $G$-equivariant line bundle $\mc{O}(\lambda - \rho)$ on $\mc{B}$. Thus, $\mc{D}_\lambda = \mrm{Diff}(\mc{O}(\lambda - \rho))$ and we have an equivalence of categories
\[ \mc{O}(\lambda - \rho) \otimes - \colon \Mod(\mc{D}_{\mc{B}}, K) \to \Mod(\mc{D}_\lambda, K).\]
So one can forget about the twist and just think about ordinary $\mc{D}$-modules if desired.
\end{rmk}

Let's get back to computing the associated cycle of $\tilde{\pi}^\mrm{HC}$. By Kashiwara's equivalence, the $\mc{D}_\lambda$-module $\mc{M}$ is uniquely the direct image of a flat connection $\gamma$ on $Q$; because the component group of the $K$-stabiliser of a point on $Q$ is abelian, irreducibility of $\mc{M}$ implies that $\gamma$ has rank one. (In fact, $\gamma = \mc{O}_Q \otimes \mc{O}(\lambda - \rho)$.) So, by the general formalism of direct images for $\mc{D}$-modules along closed immersions, $\mc{M}$ has a canonical good filtration $F_\bullet \mc{M}$ so that
\[ \Gr^F \mc{M} = \mc{O}_{N_Q^*} \otimes \gamma \otimes \omega_{Q/\mc{B}},\]
as $K$-equivariant coherent sheaves on $T^*\mc{B}$. Here $N_Q^* \subset T^*\mc{B}$ is the conormal bundle to $Q$ and $\omega_{Q/\mc{B}}$ is the determinant of the normal bundle. In our setting, all objects come naturally equipped with an action of $A$, so the above equality upgrades to an equality of $K \rtimes \langle A \rangle$-equivariant coherent sheaves on $T^*\mc{B}$. Taking global sections, we get
\begin{equation} \label{eq:discrete associated cycle}
 \mrm{AC}(\tilde\pi^{\mrm{HC}}) = [\mrm{R}q_*(\mc{O}_{N_Q^*} \otimes \gamma \otimes \omega_{Q/\mc{B}})] := \sum_i (-1)^i [\mrm{R}^iq_*(\mc{O}_{N_Q^*}\otimes \gamma \otimes \omega_{Q/\mc{B}})],
\end{equation}
where $q \colon N_Q^* \to \mc{N}_K^*$ is the restriction of the Springer map to $N_Q^* \subset T^*\mc{B}_\mb{C}$. (Although we will not need this fact, by \cite[(2.26)]{Schmid75} or \cite[Corollary 1.4]{DavisVilonen23}, the higher direct images all vanish, so the right hand side is just the class of the sheaf $q_*(\mc{O}_{N_Q^*} \otimes \gamma \otimes \omega_{Q/\mc{B}_\mb{C}})$.)

Now, our discrete series is not arbitrary: we know that it is \emph{generic} for some choice of Whittaker datum. By \cite[Theorem D]{Kos78} and \cite[Theorem 8.4]{Vog91}, for example, it follows that the support of $\pi^{\mrm{HC}}$ contains a regular nilpotent element in $\mf{g}^*$, i.e.\ it contains an irreducible component of $\mc{N}_K^*$. From the above, the support of $\pi^{\mrm{HC}}$ is just $q(N_Q^*)$, so we are saying that this contains a regular nilpotent. Since the Springer map is birational over the regular locus, we conclude that $N_Q^*$ maps birationally to an irreducible component of $\mc{N}_K^*$ under $q$.

\begin{lem} \label{lem:fixed nilpotent}
There exists a regular nilpotent element $x \in (\mf{g}/(\mf{b}' + \mf{k}))^* \subset q(N_Q^*)$ such that $A(x) = x$.
\end{lem}
\begin{proof}
By a result of Steinberg \cite{Ste68end} (see also \cite[Theorem 1.1.A]{KS99} or \cite[Proposition 3.4.2 (6)]{KalLLCD}), the variety $\mc{B}^A$ of $A$-stable Borel subgroups in $G_\mb{C}$ is isomorphic to the flag variety of $G_\mb{C}^A$; in particular, it is a single orbit under the connected group $G_\mb{C}^A$. Since $A$ preserves a pinning of the real $A$-stable Borel $B$, it therefore follows that it also preserves a pinning of the $A$-stable Borel $\mf{s} + \mf{u}'$. Since $q(N_Q^*)$ contains a regular nilpotent element as argued above, it follows that the conormal fibre at $B'$
\[(\mf{g}/(\mb{b}' + \mf{k}))^* = ((\mf{g}/\mf{b}')^*)^{\theta = -1}\]
also contains a regular nilpotent. Since $S$ is anisotropic, $\theta$ preserves each root space in the associated root space decomposition, so the above is only possible if $\theta$ acts by $-1$ on each simple root space. Hence, taking $x$ to be the sum of the simple root vectors in an $A$-stable pinning of $B'$ gives a regular nilpotent element in the desired space.
\end{proof}

Since $x \in \mc{N}_K^*$ is an $A$-fixed point whose $K$-orbit is open in $\mc{N}_K^*$ (and hence an irreducible component) we have a well-defined homomorphism
\[ (-)_x \colon \mrm{K}_{K \rtimes \langle A \rangle}(\mc{N}_K^*) \to \mrm{K}(\Rep \langle A \rangle)\]
given by restricting to $x$. From \eqref{eq:discrete associated cycle} we deduce:

\begin{lem}
We have
\[ \mrm{AC}(\tilde{\pi}^{\mrm{HC}})_x = [\gamma_{B'} \otimes \det (\mf{g}/(\mf{b}' + \mf{k}))],\]
where we write $\gamma_0$ for the representation of $\langle A \rangle$ on the fibre of $\gamma$ at the point $B' \in \mc{B}$.
\end{lem}

Finally, let's compare the associated cycle with the $\mf{u}'$-homology $\mrm{H}(\tilde{\pi})$.

\begin{lem} \label{lem:AC vs H}
We have
\[ \mrm{H}(\tilde{\pi}) = \gamma_{B'} \]
as representations of $\langle A \rangle$. Hence,
\[ \mrm{AC}(\tilde{\pi}^{\mrm{HC}})_x = [\mrm{H}(\tilde{\pi}) \otimes \det (\mf{g}/(\mf{b}' + \mf{k}))].\]
\end{lem}
\begin{proof}
This is a standard computation, which comes down to the relation between $\mf{u}'$-homology and $\mc{D}$-module stalks. For $\lambda$ regular and integrally dominant, the inverse $\Delta$ to $\Gamma \colon \Mod(\mc{D}_\lambda, K\rtimes \langle A \rangle) \to \Mod(\mf{g}, K \rtimes \langle A \rangle)_{\chi_\lambda}$ is given by
\[ \Delta(M) = (\tilde{\mc{D}} \otimes_{U(\mf{g})} M)_{\widetilde{\lambda - \rho}} = (\tilde{\mc{D}} \overset{\mrm{L}}\otimes_{U(\mf{g})} M)_{\widetilde{\lambda - \rho}},\]
where $(-)_{\widetilde{\lambda - \rho}}$ denotes the summand on which $S(\mf{h}) \subset \tilde{\mc{D}}$ acts with generalised eigenvalues $\lambda - \rho$. In particular, $\Delta(\tilde{\pi}^{\mrm{HC}}) = \mc{M}$ in the notation above.

Now, let $\mc{I}$ denote the ideal sheaf of the point $\mf{s} + \mf{u}'$ in $\mc{B}$. Then it is an easy consequence of the PBW theorem that $\tilde{\mc{D}}/\mc{I}\tilde{\mc{D}}$ (a sheaf supported at a point) is canonically isomorphic as a $(U(\mf{h}) = U(\mf{s}), U(\mf{g}))$-bimodule to $\mb{C} \otimes_{U(\mf{u}')} U(\mf{g})$. Thus, we have
\[ H_j(\mf{u}', \pi^{\mrm{HC}})_{\widetilde{if + \rho_f}} = H_j(\mf{u}', \pi^{\mrm{HC}})_{\widetilde{\lambda - \rho}} = \mrm{Tor}_j^{\tilde{\mc{D}}}(\tilde{\mc{D}}/\mc{I}\tilde{\mc{D}}, \mc{M}) = \mrm{Tor}_j^{\mc{D}_\lambda}(\mc{D}_\lambda/\mc{I}\mc{D}_\lambda, \mc{M})\]
for all $j$. Here we note that, according to our conventions, $\rho_f = -\rho$ as elements of $\mf{s}^* = \mf{h}^*$. Since $\mc{M}$ is the direct image of the connection $\gamma$ along a closed immersion of codimension $q(G)$, it is an elementary computation that
\[ \mrm{Tor}_j^{\mc{D}_\lambda}(\mc{D}_\lambda/\mc{I}\mc{D}_\lambda, \mc{M}) = \begin{cases} \gamma_0, & \text{if $j = q(G)$}, \\ 0, & \text{otherwise}.\end{cases}\]
The statement of the lemma now follows.
\end{proof}

\begin{lem} \label{lem:determinant}
The action of $A$ on $\det(\mf{g}/(\mf{b}'+ \mf{k}))$ is by $(-1)^{q(G) - q(G^A)}$.
\end{lem}
\begin{proof}
Observe that since $S$ is $A$-stable and, by Lemma \ref{lem:fixed nilpotent}, $A$ fixes a regular nilpotent element $(\mf{g}/\mf{b}')^*$, so $A$ acts by permuting the root spaces and acts by the identity on any simple root space that is fixed. Hence, $A$ fixes a pinning of $(S, B')$ and hence the associated real structure on $\mf{g}/(\mf{b}'+ \mf{k})$ coming from the corresponding split form of $\mf{g}$. Now, for any finite order automorphism $T$ of a real vector space $V$, we have
\[ \det T = (-1)^{\dim V - \dim V^T}.\]
So $A$ acts on $\det(\mf{g}/(\mf{b}' + \mf{k})$ by
\[ (-1)^{\dim \mf{g}/(\mf{b}' + \mf{k}) - \dim (\mf{g}/(\mf{b}' + \mf{k}))^{A}} = (-1)^{q(G) - q(G^A)}.\]
\end{proof}

\subsection{Comparison to Whittaker vectors}

By Lemmas \ref{lem:AC vs H} and \ref{lem:determinant}, to complete the proof of Proposition \ref{pro:signs}, it suffices to show the following:

\begin{lem} \label{lem:AC vs Wh}
Let $\tilde{\pi}$ be any lift of the discrete series $\pi$ to $G(\mb{R}) \rtimes \langle A \rangle$. Then
\begin{equation} \label{eq:AC vs Wh 1}
 \mrm{AC}(\tilde{\pi}^{\mrm{HC}})_x = [\mrm{Wh}(\tilde{\pi})^*] \in \mrm{K}(\Rep \langle A \rangle).
 \end{equation}
\end{lem}

\begin{rem}
The equality in Lemma \ref{lem:AC vs Wh} has the following conceptual (if not currently rigorous) explanation. The class $[\mrm{AC}(\tilde{\pi}^{\mrm{HC}})]_x$ can be interpreted as a rudimentary version of the ``microlocal stalk'' of a Harish-Chandra module at a generic point in the $K$-nilpotent cone. This can also be described as the dual to a space of functionals $\tilde{\pi}^{\mrm{HC}} \to \mb{C}$ that are $\psi^K$-equivariant for an \emph{algebraic} Whittaker datum $(\mf{u}^-, \psi^K)$, where $\mf{u}^- \subset \mf{g}$ is the unipotent radical of a $\theta$-stable Borel subalgebra opposite to $\mf{u}'$ and $\psi^K \colon \mf{u}^- \to \mb{C}$ is a generic character such that $\psi^K\circ \theta = -\psi^K$. I expect that, for a general admissible representation $\Pi$, one has a natural isomorphism
\begin{equation} \label{eq:expectation}
\hom_{\mf{u}^-}(\Pi^{\mrm{HC}}, \mb{C}_{\psi^K}) \cong \hom_{U(\mb{R})}(\Pi^\infty, \mb{C}_\psi),
\end{equation}
where the algebraic and real Whittaker data $(\mf{u}^-, \psi^K)$ and $(U, \psi)$ are related by a version of the Kostant-Sekiguchi bijection for nilpotent orbits (cf. \cite[\S 3.5]{AdamsAfgoustidis24b}).
One subtlety in proving a statement like \eqref{eq:expectation} is that, while the left hand side is purely algebraic and thus in some sense ``easy'', the right hand side is of a more delicate analytic nature. For example, \eqref{eq:expectation} is easily seen to be false even for $\mrm{SL}_2(\mb{R})$ if we replace the smooth vectors $\Pi^\infty$ with the analytic vectors $\Pi^\omega$.
\end{rem}

The proof of Lemma \ref{lem:AC vs Wh} occupies the remainder of this appendix. The first step is to establish the lemma when $\pi$ is replaced with a principal series representation. This proceeds by directly calculating both sides of \eqref{eq:AC vs Wh 1} for such a representation. For the Whittaker vectors, we have:

\begin{lem} \label{lem:PS Wh}
Let
\[ \Pi = \mrm{Ind}_{B(\mb{R})}^{G(\mb{R})}(\chi) = \mrm{Ind}_{B(\mb{R})\rtimes \langle A \rangle }^{G(\mb{R})\rtimes \langle A \rangle}(\chi)\]
for some character $\chi$ of $B(\mb{R}) \rtimes \langle A \rangle$. Then
\[ \mrm{Wh}(\Pi)^* \cong \mb{C}_\chi\]
as representations of $\langle A \rangle$.
\end{lem}
\begin{proof}
We may assume without loss of generality that $\chi$ is trivial on $\langle A \rangle$. By \cite[Theorem 6.6.2]{Kos78}, $\dim \mrm{Wh}(\Pi) = 1$. To compute the $A$-action on this space, fix a non-zero Whittaker functional $v \colon \Pi^\infty \to \mb{C}$. Note that $\Pi^\infty$ is the space of smooth functions on $G(\mb{R})/U(\mb{R})$ transforming by the $\chi$ under the right action of $T(\mb{R})$ (where $T\cong B/U$ is a maximal torus in $B$). Inside here, we have the space of such functions supported inside the open Bruhat cell $U(\mb{R}) w_0T(\mb{R})$, where $w_0$ is a lift of the longest element of the Weyl group to an $A$-fixed element in $N_G(T)(\mb{R})$ (which exists since $A$ fixes a pinning); this is isomorphic by restricting to $U(\mb{R})w_0$ to the space $C^\infty_c(U(\mb{R}))$ of compactly supported smooth functions on $U(\mb{R})$. Now, it is easy to see that there can be no non-zero Whittaker vectors whose restriction to this subspace is zero, so
\[ v|_{C^\infty_c(U(\mb{R}))} \colon C^\infty_c(U(\mb{R})) \to \mb{C} \]
defines a non-zero distribution on $U(\mb{R})$, which by $\psi$-equivariance must be equal to $\psi$ times a (complex) Haar measure on $U(\mb{R})$. Since $A$ fixes both $\psi$ and any Haar measure on $U(\mb{R})$, we conclude that $A$ must fix $v$. So $\mrm{Wh}(\Pi)$ is the trivial representation of $\langle A \rangle$.
\end{proof}

For the associated cycle, on the other hand, we have:

\begin{lem} \label{lem:PS AC}
Let $\Pi$ be as in Lemma \ref{lem:PS Wh}. Then
\[ \mrm{AC}(\Pi^{\mrm{HC}}) = [\mc{O} _{\mc{N}_K^*} \otimes \mrm{R}p_*\mc{L}] \]
where $p \colon T^*\mc{B} \times_{\mf{g}^*} (\mf{g}/\mf{k})^* \to \mc{N}_K^*$ is the pullback of the Springer resolution and $\mc{L}$ is (the pullback of) a $K \ltimes \langle A\rangle$-equivariant line bundle on $\mc{B}$ such that $A$ acts on the fibre at $B$ by $\chi$. In particular
\begin{equation} \label{eq:PS AC 1}
\mrm{AC}(\Pi^{\mrm{HC}})_x = [\mb{C}_\chi].
\end{equation}
\end{lem}
\begin{proof}
Let $\mu = d\chi + \rho$. Then $\chi$ defines a $K \rtimes \langle A \rangle$-equivariant $\mc{D}_\mu$-module $\mc{O}(\chi)$ (an equivariant line bundle with $\mu$-twisted connection) on the open $K$-orbit $\mb{O} = K \cdot B \subset \mc{B}$ such that $(K \cap B)\rtimes \langle A \rangle$ acts on the fibre at $B$ through $\chi$. The Harish-Chandra module $\Pi^{\mrm{HC}}$ is then given by
\[ \Pi^{\mrm{HC}} = \Gamma(\mc{B}, j_*\mc{O}(\chi)) = \Gamma(\mb{O}, \mc{O}(\chi)),\]
where $j \colon \mb{O} \to \mc{B}$ is the inclusion. Note that since $B$ is a real Borel, $K \cap B = K \cap T$ is reductive, so $\mb{O} = K/(K \cap B)$ is affine. So the $\mc{D}_\mu$-module $j_*\mc{O}(\chi)$ has no higher cohomology (whether or not $\mu$ is integrally dominant).

To compute the associated cycle, we apply the methods of \cite{Ginsburg86} for computing the characteristic cycles of direct images of $\mc{D}$-modules under open immersions. Consider a $K \rtimes \langle A \rangle$-invariant boundary equation $g \in \Gamma(\mc{B}, \mc{O}(n\rho))$ for the open $K$-orbit $\mb{O}$; such an equation exists since $G$ is quasi-split. We then have an associated module $j_*g^s\mc{O}(\chi)[s]$ over $\mc{D}' := \tilde{\mc{D}} \otimes_{S(\mf{h}), h \mapsto (\mu - \rho)(h) + n\rho(h)s} \mb{C}[s]$. As in \cite[\S 3.5]{Ginsburg86}, choose a $K \rtimes \langle A \rangle$-stable $\mc{O}_{\mc{B}}$-coherent subsheaf $\mc{F} \subset j_*\mc{O}(\chi)$ generating $j_*\mc{O}(\chi)$ as a $\mc{D}_\mu$-module and set
\[ \mc{N} = \mc{D}' \cdot g^s\mc{F} \subset j_*g^s\mc{O}(\chi)[s].\]
Then, replacing $\mc{F}$ with $g^{-k}\mc{F}$ for some $k$ if necessary, we have by \cite[Proposition 3.6]{Ginsburg86} that $\mc{N}$ is a coherent $\mc{D}'$-module, flat over $\mb{C}[s] \subset \mc{D}'$, such that $\mc{N}/s\mc{N} = j_*\mc{O}(\chi)$. Taking derived global sections, we get a complex
\[ N = \mrm{R}\Gamma(\mc{N}) \in \mrm{D}^b\Mod(U(\mf{g}) \otimes_{Z(U(\mf{g}))} \mb{C}[s], K)_{fg} \]
whose derived restriction to $s = 0$ is $\Pi^{\mrm{HC}}$.

Choosing any $K \rtimes \langle A \rangle$-stable good $\mc{D}'$-module filtration $F_\bullet \mc{N}$, compatible with the order filtration on $\mc{D}'$ for which $s$ has degree $1$, we obtain a good filtration $F_\bullet N = \mrm{R}\Gamma(F_\bullet\mc{N})$ on the complex $N$. The associated graded of $\mc{N}$ is a $K\rtimes \langle A \rangle$-equivariant coherent sheaf on $\tilde{\mf{g}}^*\times_{\mf{h}^*} \mb{C}_s = \operatorname{Spec} \Gr^F\mc{D}'$ supported on the pre-image of $(\mf{g}/\mf{k})^* \subset \mf{g}^*$; here $\tilde{\mf{g}}^* = \spec \Gr^F\tilde{\mc{D}}$ is the Grothendieck-Springer space. Similarly, $\Gr^F N = \mrm{R}m_*(\Gr^F\mc{N})$ is a complex of $K \rtimes \langle A \rangle$-equivariant coherent sheaves on $\mf{g}^* \times_{\mf{h}^*/W} \mb{C}_s$ whose cohomologies are supported on $(\mf{g}/\mf{k})^* \times_{\mf{h}^*/W} \mb{C}_s$; here $m \colon \tilde{\mf{g}}^*\times_{\mf{h}^*} \mb{C}_s \to \mf{g}^* \times_{\mf{h}^*/W} \mb{C}_s$ is the pullback of the Grothendieck-Springer map $\tilde{\mf{g}}^* \to \mf{g}^*$ and $W$ is the Weyl group. Arguing as in \cite[Lemma 3.5]{Ginsburg86}, we have that
\[ \mrm{AC}(\Pi^{\mrm{HC}}) = [\Gr^F N|_{s = 0}] \in \mrm{K}_{K \rtimes \langle A \rangle}(\mc{N}_K^*)\]
where $|_{s = 0}$ denotes the derived restriction to $s = 0$.

Now, by Lemma \ref{lem:specialisation} below, we have
\[ [\Gr^F N|_{s = 0}] = [\mc{G}|_{s = 0}] \]
if $\mc{G}$ is any other coherent sheaf such that $[\mc{G}|_{s \neq 0}] = [\Gr^F N|_{s \neq 0}]$. So let us determine the latter. First note that $g \in \mrm{H}^0(\mc{B}, \mc{O}(n\rho))$ defines an $H$-semi-invariant function on $\tilde{\mc{B}}$ transforming via the character $n \rho$. So $d \log g$ determines a well-defined section of $\tilde{\mf{g}}^* = T^*\tilde{\mc{B}}/H$ over $\mb{O}$, with the property that its image under $\tilde{\mf{g}}^* \to \mf{h}^*$ is equal to $n\rho$. Let $\Gamma$ be the image of the map
\begin{align*}
\mb{O} \times \mb{C} &\to \tilde{\mf{g}}^* \times_{\mf{h}^*} \mb{C} \\
(y, s) &\mapsto (s(\operatorname d \log g)_y, s).
\end{align*}
Then by \cite[Theorem 2.3]{Ginsburg86}, we have
\[ \operatorname{Supp} \Gr^F\mc{N} = \overline{\Gamma}.\]
Restricting to $s \neq 0$, $\Gamma|_{s \neq 0}$ is already closed inside $\tilde{\mf{g}}^* \times_{\mf{h}^*} \mb{C}^\times$, so in particular, $\Gr^F\mc{N}|_{s \neq 0}$ is equal to the direct image of its restriction to the pre-image of $\mb{O}$. Over $\mb{O}$, we have $\mc{N}|_{\mb{O}} = g^s\mc{O}(\chi)[s]$, so we have a good filtration given by order in $s$ whose associated graded is $\mc{O}(\chi)\otimes \mc{O}_\Gamma|_{s \neq 0}$. So
\[ [\Gr^F N|_{s \neq 0}] = [\mrm{R}m_*\mc{O}(\chi) \otimes \mc{O}_\Gamma|_{s \neq 0}].\]
Now, observe that since $n\rho$ is regular, $m \colon \tilde{\mf{g}}^* \times_{\mf{h}^*} \mb{C} \to \mf{g}^* \times_{\mf{h}^*/W} \mb{C}$ is an isomorphism over $s \neq 0$. The image $m(\Gamma|_{s \neq 0})$ is simply the $K \times \mb{C}^\times$-orbit of a regular semisimple element $\omega \in (\mf{g}/\mf{k})^*$. In other words, by Kostant-Rallis \cite{KosRal71}, we have
\[ m(\Gamma|_{s \neq 0}) = (\mf{g}/\mf{k})^* \times_{\mf{a}^*/W_\mf{a}} \mb{C}^\times \hookrightarrow \mf{g}^* \times_{\mf{h}^*/W} \mb{C}^\times,\]
where $\mf{a}^*$ is a Cartan subspace of $(\mf{g}/\mf{k})^*$ containing $\omega$ and $W_\mf{a}$ is the associated Weyl group. Choose any extension $\mc{L}$ of $\mc{O}(\chi)$ to a $K \rtimes \langle A \rangle$-equivariant line bundle on $\mc{B}$; such an extension exists since the map $K \cap B \to H$ is injective, so we may extend the restriction of $\chi$ to the compact part to an algebraic character of $H$. Then setting
\[ \mc{G} = \mc{O}_{(\mf{g}/\mf{k})^* \times_{\mf{a}^*/W_\mf{a}} \mb{C}} \otimes \mrm{R}m_*\mc{L},\]
we have $[\mc{G}|_{s \neq 0}] = [\Gr^F N|_{s \neq 0}]$ and hence
\[ \mrm{AC}(\Pi^{\mrm{HC}}) = [\mc{G}|_{s = 0}] = [\mc{O}_{\mc{N}_K^*} \otimes \mrm{R}p_*\mc{L}].\]

Finally, to deduce \eqref{eq:PS AC 1}, taking the stalk at our regular nilpotent $x$, we get
\[ \mrm{AC}(\Pi^{\mrm{HC}})_x = [\mc{L}_{B'}],\]
the class of a $1$-dimensional representation of $\langle A \rangle$. Since, as observed in the proof of Lemma \ref{lem:fixed nilpotent}, the variety $\mc{B}^A$ is connected, the action of $A$ on the fibre $\mc{L}_{B'}$ is the same as the action on the fibre $\mc{L}_{B}$, which is through $\chi$ by construction.
\end{proof}

\begin{lem} \label{lem:specialisation}
Let $X \to \spec \mb{C}[s]$ be a morphism of schemes, invariant under the action of a group $H$ on $X$, and let $\mc{F}, \mc{F}' \in \mrm{D}^b_{\mrm{coh}, H}(X)$ be bounded complexes of $H$-equivariant coherent sheaves. If the restrictions to $s \neq 0$ agree in the Grothendieck group:
\[ [\mc{F}|_{s \neq 0}] = [\mc{F}'|_{s \neq 0}] \in \mrm{K}_H(X \times_{\spec \mb{C}[s]} \spec \mb{C}[s, s^{-1}]) \]
then the derived restrictions to $s = 0$ agree in the Grothendieck group:
\[ [\mc{F}|_{s = 0}] = [\mc{F}'|_{s = 0}] \in \mrm{K}_H(X \times_{\spec \mb{C}[s]} \{0\}).\]
\end{lem}
\begin{proof}
This is similar to \cite[Proposition 1.6]{Ginsburg86}. The statement reduces immediately to the case where $\mc{F}$ and $\mc{F}'$ are honest coherent sheaves whose restrictions to $s \neq 0$ are isomorphic. Multiplying such an isomorphism by an appropriate power of $s$, we can assume that we have a map $\mc{F} \to \mc{F}'$ whose kernel and cokernel are supported set-theoretically on $X \times_{\spec \mb{C}[s]} \{0\}$. Since derived restriction is additive in the Grothendieck group, the statement then reduces to showing that $[\mc{G}|_{s = 0}] = 0$ if $\mc{G}$ is set-theoretically supported on $X \times_{\spec \mb{C}[s]} \{0\}$, which in turn reduces to showing the same statement when $\mc{G}$ is supported scheme-theoretically on $X \times_{\spec \mb{C}[s]} \{0\}$. But in this case
\[ \mc{G}|_{s = 0} := \mc{G} \overset{\mrm{L}}\otimes_{\mb{C}[s]} \mb{C} = \mc{G} \oplus \mc{G}[1],\]
which does indeed have trivial class in the Grothendieck group.
\end{proof}

In order to reduce Lemma \ref{lem:AC vs Wh} to the study of principal series above, we use the following.

\begin{lem} \label{lem:PS to DS}
The generic discrete series $\pi$ is infinitesimally equivalent to a quotient of the principal series representation
\[ \Pi := \mrm{Ind}_{B(\mb{R})}^{G(\mb{R})}(\lambda - \rho).\]
\end{lem}
\begin{proof}
Recall that the localisation of $\pi^{\mrm{HC}}$ is the direct image $i_+\gamma$ of a $K$\-equivariant $(\lambda - \rho)$-twisted local system $\gamma$ under the inclusion $i \colon Q \to \mc{B}$ of a closed $K$-orbit $Q$. Moreover, the stabiliser under $K$ of a point in $Q$ is a Borel subgroup in $K$, hence connected since we have assumed $G$ simply connected. So $\gamma \cong \mc{O}_Q(\lambda - \rho)$ must be isomorphic to the trivial local system on $Q$ twisted by the $G$-equivariant line bundle $\mc{O}(\lambda - \rho)$. We will show by induction on $\dim Q \leq n \leq \dim \mc{B}$ that there exists a $K$-orbit $i' \colon Q' \to \mc{B}$ with $\dim Q' = n$ and a surjection
\begin{equation} \label{eq:PS to DS 1}
\Gamma(i'_+\mc{O}_{Q'}(\lambda - \rho)) \to \pi^{\mrm{HC}}.
\end{equation}
When $n = \dim \mc{B}$, $Q'$ is the open orbit, so the source of \eqref{eq:PS to DS 1} is $\Pi^{\mrm{HC}}$.

As a base case, when $n = \dim Q$, we can take $Q' = Q$ and \eqref{eq:PS to DS 1} to be the identity map. So suppose $n > \dim Q$ and, by induction, fix $i'' \colon Q'' \to \mc{B}$ of dimension $n - 1$ with the desired property. Since $\dim Q'' < \dim \mc{B}$, 
there exists a simple root $\alpha$ with associated $\mb{P}^1$-fibration $\pi_\alpha \colon \mc{B} \to \mc{P}_\alpha$ such that $\pi_\alpha|_{Q''}$ is finite onto its image (cf.\ \cite[Lemma 5.1]{VogIC3}). So we have a decomposition
\[ \pi_\alpha^{-1}\pi_\alpha(Q'') = Q' \cup Q'' \cup Q''',\]
where $Q'$ is a $K$-orbit with $\dim Q' = \dim Q'' + 1 = n$ and $Q'''$ is a possibly empty disjoint union of $K$-orbits of dimension $\dim Q''$. (In fact, $Q'''$ will be either empty or a single $K$-orbit, but we will not use this.) Thus, we have a short exact sequence
\[ 0 \to \mc{O}_{\pi_\alpha^{-1}\pi_\alpha(Q'')}(\lambda - \rho) \to a'_+\mc{O}_{Q'}(\lambda - \rho) \to a''_+\mc{O}_{Q''}(\lambda - \rho) \oplus a'''_+\mc{O}_{Q'''}(\lambda - \rho) \to 0\]
of $K$-equivariant $\mc{D}_\lambda$-modules on $\pi_\alpha^{-1}\pi_\alpha(Q'')$, where $a', a'', a'''$ are the inclusions into $\pi_\alpha^{-1}\pi_\alpha(Q'')$. Taking the direct image under $b \colon  \pi_\alpha^{-1}\pi_\alpha(Q'') \to \mc{B}$, we get a long exact sequence
\begin{align*}
  0 &\to \mc{H}^0b_+\mc{O}_{\pi_\alpha^{-1}\pi_\alpha(Q'')}(\lambda - \rho) \to i'_+\mc{O}_{Q'}(\lambda - \rho) \to \\
  & \to i''_+\mc{O}_{Q''}(\lambda - \rho) \oplus i'''_+\mc{O}_{Q'''}(\lambda - \rho) \to \mc{H}^1b_+\mc{O}_{\pi_\alpha^{-1}\pi_\alpha(Q'')}(\lambda - \rho) \to 0.
\end{align*}
Note that, unlike the inclusions of the individual $K$-orbits, the inclusion $b$ need not be affine, so the $\mc{D}$-module direct image $b_+$ can have higher cohomology.

Now, the cokernel above can be rewritten as
\[ \mc{H}^1b_+\mc{O}_{\pi_\alpha^{-1}\pi_\alpha(Q'')}(\lambda - \rho) = \mc{O}(\lambda - \rho) \otimes \pi_\alpha^*\mc{H}^1c_+\mc{O}_{\pi_\alpha(Q'')},\]
where $c \colon \pi_\alpha(Q'') \to \mc{P}_\alpha$ is the inclusion. In particular, the characteristic cycle is contained in $T^*\mc{P}_\alpha \times_{\mc{P}_\alpha}\mc{B} \subset T^*\mc{B}$. Since the image of this under the Springer map does not contain a regular nilpotent element, it follows that
\[ \mrm{AC}(\Gamma(\mc{H}^1b_+\mc{O}_{\pi_\alpha^{-1}\pi_\alpha(Q'')}(\lambda - \rho)))_x = 0.\]
In particular, $\Gamma(\mc{H}^1b_+\mc{O}_{\pi_\alpha^{-1}\pi_\alpha(Q'')}(\lambda - \rho))$ cannot have $\pi^{\mrm{HC}}$ as a composition factor. Hence, the map
\[ \Gamma(i'_+\mc{O}_{Q'}(\lambda - \rho)) \to \Gamma(i''_+\mc{O}_{Q''}(\lambda - \rho)) \to \pi^{\mrm{HC}} \]
must be non-zero, hence surjective.
\end{proof}

We now complete the proof of Lemma \ref{lem:AC vs Wh} and hence Proposition \ref{pro:signs}.

\begin{proof}[Proof of Lemma \ref{lem:AC vs Wh}]

By Lemma \ref{lem:PS to DS}, $\pi$ is a quotient of the principal series representation $\Pi = \mrm{Ind}_{B(\mb{R}})^{G(\mb{R})}(\lambda - \rho)$.  Since $\lambda - \rho$ is fixed under $A$, we can extend it to a character $\chi$ of $B(\mb{R}) \rtimes \langle A \rangle$ and hence $\Pi$ to a representation of $G(\mb{R}) \rtimes \langle A \rangle$. Since both $\pi$ and $\Pi$ have a unique Whittaker model by \cite[Theorem 6.6.2]{Kos78}, $\pi$ must have multiplicity one as a quotient of $\pi$, so there is a unique lift $\tilde{\pi}$ for which the quotient $\Pi \to \tilde{\pi}$ is $A$-equivariant. With this choice,
\[ \mrm{Wh}(\tilde{\pi}) = \mrm{Wh}(\Pi) = \mb{C}_\chi\]
by Lemma \ref{lem:PS Wh}. Similarly, $\mrm{AC}(\tilde{\pi}^{\mrm{HC}})_x$ is a summand of $\mrm{AC}(\Pi^{\mrm{HC}})_x = [\mb{C}_\chi]$ by Lemma \ref{lem:PS AC}, so it is either zero or $[\mb{C}_\chi]$. But $\mrm{AC}(\tilde{\pi}^{\mrm{HC}})_x \neq 0$. Since both sides of \eqref{eq:AC vs Wh 1} transform in the same way under tensoring with characters of $\langle A \rangle$, this proves the lemma for any lift $\tilde{\pi}$.
\end{proof}

\bibliographystyle{amsalpha}
\bibliography{../../../../TexMain/bibliography.bib}

\providecommand{\bysame}{\leavevmode\hbox to3em{\hrulefill}\thinspace}
\providecommand{\MR}{\relax\ifhmode\unskip\space\fi MR }
\providecommand{\MRhref}[2]{%
  \href{http://www.ams.org/mathscinet-getitem?mr=#1}{#2}
}
\providecommand{\href}[2]{#2}
\begin{thebibliography}{DHV84}

\bibitem[AA]{AdamsAfgoustidis24b}
Jeffrey Adams and Alexandre Afgoustidis, \emph{Nilpotent invariants for generic
  discrete series of real groups}, arXiv:2410.04134.

\bibitem[AK24]{DPR}
Jeffrey Adams and Tasho Kaletha, \emph{Discrete series {L}-packets for real
  reductive groups}, Preprint, {arXiv}:2409.13375 [math.{RT}] (2024), 2024.

\bibitem[AV16]{AV16}
Jeffrey Adams and David~A. Vogan, Jr., \emph{Contragredient representations and
  characterizing the local {L}anglands correspondence}, Amer. J. Math.
  \textbf{138} (2016), no.~3, 657--682. \MR{3506381}

\bibitem[BB81]{BB81}
Alexandre Beilinson and Joseph Bernstein, \emph{Localisation de {$g$}-modules},
  C. R. Acad. Sci. Paris S\'er. I Math. \textbf{292} (1981), no.~1, 15--18.
  \MR{610137}

\bibitem[BB93]{BB93}
A.~Beilinson and J.~Bernstein, \emph{A proof of {J}antzen conjectures}, I. {M}.
  {G}elfand {S}eminar, Adv. Soviet Math., vol. 16, Part 1, Amer. Math. Soc.,
  Providence, RI, 1993, pp.~1--50. \MR{1237825}

\bibitem[Bou87]{Bouaziz87}
Abderrazak Bouaziz, \emph{Sur les caract\`eres des groupes de {L}ie
  r\'{e}ductifs non connexes}, J. Funct. Anal. \textbf{70} (1987), no.~1,
  1--79. \MR{870753}

\bibitem[DHV84]{DHV84}
Michel Duflo, Gerrit Heckman, and Mich\`ele Vergne, \emph{Projection d'orbites,
  formule de {K}irillov et formule de {B}lattner}, no.~15, 1984, Harmonic
  analysis on Lie groups and symmetric spaces (Kleebach, 1983), pp.~65--128.
  \MR{789081}

\bibitem[Dil23]{Dillery23}
Peter Dillery, \emph{Rigid inner forms over local function fields}, Adv. Math.
  \textbf{430} (2023), Paper No. 109204, 100. \MR{4617942}

\bibitem[Duf82]{Duflo82}
Michel Duflo, \emph{Construction de repr\'{e}sentations unitaires d'un groupe
  de {L}ie}, Harmonic analysis and group representations, Liguori, Naples,
  1982, pp.~129--221. \MR{777341}

\bibitem[DV]{DavisVilonen23}
Dougal Davis and Kari Vilonen, \emph{Unitary representations of real groups and
  localization theory for hodge modules}, arXiv:2309.13215.

\bibitem[Gin86]{Ginsburg86}
V.~Ginsburg, \emph{Characteristic varieties and vanishing cycles}, Invent.
  Math. \textbf{84} (1986), no.~2, 327--402. \MR{833194}

\bibitem[Hum80]{Humphreys80}
James~E. Humphreys, \emph{Introduction to {Lie} algebras and representation
  theory. 3rd printing, rev}, Grad. Texts Math., vol.~9, Springer, Cham, 1980
  (English).

\bibitem[Kal16a]{KalSimons}
Tasho Kaletha, \emph{The local {L}anglands conjectures for non-quasi-split
  groups}, Families of automorphic forms and the trace formula, Simons Symp.,
  Springer, 2016, pp.~217--257. \MR{3675168}

\bibitem[Kal16b]{KalRI}
\bysame, \emph{Rigid inner forms of real and {$p$}-adic groups}, Ann. of Math.
  (2) \textbf{184} (2016), no.~2, 559--632. \MR{3548533}

\bibitem[Kal19]{KalDC}
\bysame, \emph{On {$L$}-embeddings and double covers of tori over local
  fields}, arXiv:1907.05173 (2019).

\bibitem[Kal22]{KalLLCD}
\bysame, \emph{On the local {L}anglands conjectures for disconnected groups},
  arxiv:2210.02519 (2022).

\bibitem[Kna86]{KnappSS}
Anthony~W. Knapp, \emph{Representation theory of semisimple groups}, Princeton
  Mathematical Series, vol.~36, Princeton University Press, Princeton, NJ,
  1986, An overview based on examples. \MR{855239}

\bibitem[Kna02]{KnappLie}
\bysame, \emph{Lie groups beyond an introduction}, second ed., Progress in
  Mathematics, vol. 140, Birkh\"{a}user Boston, Inc., Boston, MA, 2002.
  \MR{1920389}

\bibitem[Kos78]{Kos78}
Bertram Kostant, \emph{On {W}hittaker vectors and representation theory},
  Invent. Math. \textbf{48} (1978), no.~2, 101--184. \MR{507800 (80b:22020)}

\bibitem[Kot83]{Kot83}
Robert~E. Kottwitz, \emph{Sign changes in harmonic analysis on reductive
  groups}, Trans. Amer. Math. Soc. \textbf{278} (1983), no.~1, 289--297.
  \MR{697075 (84i:22012)}

\bibitem[KR71]{KosRal71}
B.~Kostant and S.~Rallis, \emph{Orbits and representations associated with
  symmetric spaces}, Amer. J. Math. \textbf{93} (1971), 753--809. \MR{311837}

\bibitem[KS]{KS12}
Robert~E. Kottwitz and Diana Shelstad, \emph{On splitting invariants and sign
  conventions in endoscopic transfer}, arXiv:1201.5658.

\bibitem[KS99]{KS99}
\bysame, \emph{Foundations of twisted endoscopy}, Ast\'erisque (1999), no.~255,
  vi+190. \MR{1687096 (2000k:22024)}

\bibitem[Lan89]{Lan89}
R.~P. Langlands, \emph{On the classification of irreducible representations of
  real algebraic groups}, Representation theory and harmonic analysis on
  semisimple {L}ie groups, Math. Surveys Monogr., vol.~31, Amer. Math. Soc.,
  Providence, RI, 1989, pp.~101--170. \MR{1011897 (91e:22017)}

\bibitem[LH17]{HL17}
Bertrand Lemaire and Guy Henniart, \emph{Repr\'esentations des espaces tordus
  sur un groupe r\'eductif connexe $p$-adique}, Ast\'erisque (2017), no.~386,
  ix+366. \MR{3632513}

\bibitem[LS87]{LS87}
R.~P. Langlands and D.~Shelstad, \emph{On the definition of transfer factors},
  Math. Ann. \textbf{278} (1987), no.~1-4, 219--271. \MR{909227 (89c:11172)}

\bibitem[Luo23]{YiLuoThesis}
Yi~Luo, \emph{On the multiplicity formula for discrete automorphic
  representations of disconnected tori}, preprint, arXiv:2312.16389, 2023.

\bibitem[Mez13]{Mezo13}
Paul Mezo, \emph{Character identities in the twisted endoscopy of real
  reductive groups}, Mem. Amer. Math. Soc. \textbf{222} (2013), no.~1042,
  vi+94. \MR{3076427}

\bibitem[Ree10]{Ree10}
Mark Reeder, \emph{Torsion automorphisms of simple {L}ie algebras}, Enseign.
  Math. (2) \textbf{56} (2010), no.~1-2, 3--47. \MR{2674853 (2012b:17040)}

\bibitem[Ren97]{Ren97}
David Renard, \emph{Twisted orbital integrals on real reductive {Lie} groups},
  J. Funct. Anal. \textbf{145} (1997), no.~2, 374--454 (French).

\bibitem[Sch75]{Schmid75}
Wilfried Schmid, \emph{Some properties of square-integrable representations of
  semisimple {L}ie groups}, Ann. of Math. (2) \textbf{102} (1975), no.~3,
  535--564. \MR{579165}

\bibitem[She81]{She81}
D.~Shelstad, \emph{Embeddings of {$L$}-groups}, Canad. J. Math. \textbf{33}
  (1981), no.~3, 513--558. \MR{627641 (83e:22022)}

\bibitem[She08]{SheTE1}
\bysame, \emph{Tempered endoscopy for real groups. {I}. {G}eometric transfer
  with canonical factors}, Representation theory of real reductive {L}ie
  groups, Contemp. Math., vol. 472, Amer. Math. Soc., Providence, RI, 2008,
  pp.~215--246. \MR{2454336 (2011d:22013)}

\bibitem[She10]{SheTE2}
\bysame, \emph{Tempered endoscopy for real groups. {II}. {S}pectral transfer
  factors}, Automorphic forms and the {L}anglands program, Adv. Lect. Math.
  (ALM), vol.~9, Int. Press, Somerville, MA, 2010, pp.~236--276. \MR{2581952}

\bibitem[She12]{She12}
\bysame, \emph{On geometric transfer in real twisted endoscopy}, Ann. of Math.
  (2) \textbf{176} (2012), no.~3, 1919--1985. \MR{2979862}

\bibitem[Ste68]{Ste68end}
Robert Steinberg, \emph{Endomorphisms of linear algebraic groups}, Memoirs of
  the American Mathematical Society, No. 80, American Mathematical Society,
  Providence, R.I., 1968. \MR{0230728 (37 \#6288)}

\bibitem[Ta{\"i}22]{TaibiIHES2022}
Olivier Ta{\"i}bi, \emph{The local {L}anglands conjecture}, Proceedings of the
  IHES Summer School on the Langlands Program, 2022.

\bibitem[Tat79]{TateCor}
J.~Tate, \emph{Number theoretic background}, Automorphic forms, representations
  and {$L$}-functions ({P}roc. {S}ympos. {P}ure {M}ath., {O}regon {S}tate
  {U}niv., {C}orvallis, {O}re., 1977), {P}art 2, Proc. Sympos. Pure Math.,
  XXXIII, Amer. Math. Soc., Providence, R.I., 1979, pp.~3--26. \MR{546607
  (80m:12009)}

\bibitem[Vog78]{Vog78}
David~A. Vogan, Jr., \emph{Gelfand-{K}irillov dimension for {H}arish-{C}handra
  modules}, Invent. Math. \textbf{48} (1978), no.~1, 75--98. \MR{0506503 (58
  \#22205)}

\bibitem[Vog83]{VogIC3}
David~A. Vogan, \emph{Irreducible characters of semisimple {L}ie groups. {III}.
  {P}roof of {K}azhdan-{L}usztig conjecture in the integral case}, Invent.
  Math. \textbf{71} (1983), no.~2, 381--417. \MR{689650 (84h:22036)}

\bibitem[Vog91]{Vog91}
David~A. Vogan, Jr., \emph{Associated varieties and unipotent representations},
  Harmonic analysis on reductive groups ({B}runswick, {ME}, 1989), Progr.
  Math., vol. 101, Birkh\"auser Boston, Boston, MA, 1991, pp.~315--388.
  \MR{1168491}

\end{thebibliography}

\end{document}